\documentclass{article}
\usepackage[utf8]{inputenc}

\usepackage{amsfonts}
\usepackage{amssymb}
\usepackage{amsmath}
\usepackage{amsthm}
\usepackage{mathtools}
\usepackage{cleveref}
\usepackage{url}
\usepackage{tikz}
\usepackage{tikz-cd}
\usepackage{adjustbox}
\usepackage{enumitem}
\usepackage{subcaption}
\usepackage{graphicx}
\usepackage{afterpage}
\usepackage{changepage}
\usepackage{authblk}
\usepackage{longtable}
\usepackage{bm}
\usepackage{multicol}

\tikzcdset{scale cd/.style={every label/.append style={scale=#1},cells={nodes={scale=#1}}}}

\usetikzlibrary{knots}

\usepackage[pass]{geometry}

\usepackage[
    backend=biber,
    style=alphabetic,
    sorting=nyt,
    doi=false,
    url=false,
    isbn=false,
    maxbibnames=10,
]{biblatex}
\DeclareFieldFormat[misc]{title}{``#1"}

\theoremstyle{plain}
    \newtheorem{theorem}{Theorem}
    \newtheorem{proposition}{Proposition}[section]
    \newtheorem{corollary}[proposition]{Corollary}
    \newtheorem{lemma}[proposition]{Lemma}
    \newtheorem{claim}[proposition]{Claim}

\theoremstyle{definition}
    \newtheorem{definition}[proposition]{Definition}
    \newtheorem{algorithm}[proposition]{Algorithm}
    
    \newtheorem{example}[proposition]{Example}

\theoremstyle{remark}
	\newtheorem{remark}[proposition]{Remark}%

\crefname{claim}{Claim}{Claims}

\newtheoremstyle{restate-theorem}
    {\topsep}{\topsep} 
    {\itshape}         
    {}                 
    {\bfseries}        
    {.}                
    { }                
    {\thmname{#1} \ref{#3} {\normalfont(Restated)}}
    
\theoremstyle{restate-theorem}
    \newtheorem{restate-theorem}{Theorem}
    \newtheorem{restate-proposition}{Proposition}
    \newtheorem{restate-corollary}{Corollary}

\AddToHook{env/corollary/begin}{\crefalias{proposition}{corollary}}
\AddToHook{env/lemma/begin}{\crefalias{proposition}{lemma}}
\AddToHook{env/claim/begin}{\crefalias{proposition}{claim}}
\AddToHook{env/definition/begin}{\crefalias{proposition}{definition}}
\AddToHook{env/algorithm/begin}{\crefalias{proposition}{algorithm}}
\AddToHook{env/condition/begin}{\crefalias{proposition}{condition}}
\AddToHook{env/example/begin}{\crefalias{proposition}{example}}
\AddToHook{env/question/begin}{\crefalias{proposition}{question}}
\AddToHook{env/remark/begin}{\crefalias{proposition}{remark}}
\AddToHook{env/note/begin}{\crefalias{proposition}{note}}
    
\numberwithin{equation}{section}
\numberwithin{table}{section}

\newcommand{\ZZ}{\mathbb{Z}}
\newcommand{\QQ}{\mathbb{Q}}
\newcommand{\RR}{\mathbb{R}}
\newcommand{\CC}{\mathbb{C}}
\newcommand{\FF}{\mathbb{F}}

\newcommand{\id}{\mathit{id}}
\newcommand{\isom}{\cong}
\newcommand{\htpy}{\simeq}

\newcommand{\del}{\partial}

\renewcommand{\emptyset}{\varnothing}

\renewcommand{\epsilon}{\varepsilon}

\DeclareMathOperator{\Hom}{Hom}

\DeclareMathOperator{\End}{End}

\DeclareMathOperator{\Tor}{Tor}

\DeclareMathOperator{\fchar}{char}
\DeclareMathOperator{\rank}{rank}

\DeclareMathOperator{\gr}{gr}

\DeclareMathOperator{\Cone}{Cone}
\DeclareMathOperator{\Fix}{Fix}
\DeclareMathOperator{\Cob}{Cob}

\newcommand{\CKh}{\mathit{CKh}}
\newcommand{\CKhI}{\mathit{CKhI}}
\newcommand{\Kh}{\mathit{Kh}}
\newcommand{\KhI}{\mathit{KhI}}

\newcommand{\rCKh}{\CKh_r}
\newcommand{\rCKhI}{\CKhI_r}
\newcommand{\rKh}{\Kh_r}
\newcommand{\rKhI}{\KhI_r}

\newcommand{\BN}{\mathit{BN}}
\newcommand{\BNI}{\mathit{BNI}}
\newcommand{\rBN}{\BN_r}
\newcommand{\rBNI}{\BNI_r}

\newcommand{\ca}{\alpha}
\newcommand{\cb}{\beta}

\newcommand{\uca}{\overline{\ca}{}}
\newcommand{\dca}{\underline{\ca}{}}
\newcommand{\ucb}{\overline{\cb}{}}
\newcommand{\dcb}{\underline{\cb}{}}

\newcommand{\cz}{\zeta}
\newcommand{\ucz}{\overline{\cz}{}}
\newcommand{\dcz}{\underline{\cz}{}}

\newcommand{\cx}{\xi}
\newcommand{\ucx}{\overline{\cx}{}}
\newcommand{\dcx}{\underline{\cx}{}}

\newcommand{\ud}{\overline{d}}
\newcommand{\dd}{\underline{d}}
\newcommand{\ds}{\underline{s}}
\newcommand{\us}{\overline{s}}

\newcommand{\dual}{\mathsf{D}}

\title{
    Involutive Khovanov homology and equivariant knots
}

\author{Taketo Sano}

\newcommand{\addresses}{{
  \bigskip
  \noindent
  Taketo Sano, \textsc{RIKEN iTHEMS, Wako, Saitama 351-0198, Japan }\par\nopagebreak
  \noindent
  \textit{E-mail address}: \url{taketo.sano@riken.jp}
}}

\date{August 7, 2026}

\begin{document}
\pagenumbering{roman}
\providecommand{\Ot}{O^\tau}
\providecommand{\Otb}{\overline{O}{}^\tau}

\begin{center}
    {\LARGE Erratum\par}
    \vspace{.5em}
    {\Large\itshape Involutive Khovanov homology and equivariant knots\par}
    \vspace{1.5em}
    {\large Taketo Sano\par}
    \vspace{1.5em}
    {\large \today \par}
\end{center}

\vspace{1.5em}

This note concerns the published version of this paper [1], which contains an error: Proposition~3.3 therein, stated below. The error is fixed in the present version [2], in which Section~3.1 is largely rewritten. The main theorems of [1] remain valid as stated. This note explains the error and lists the changes made.

\begin{quote}
    \textbf{Proposition} (Proposition~3.3 of [1]). \itshape
    For a strongly invertible link diagram $D$, the Lee cycles are invariant under $\tau$.
\end{quote}

\noindent
Here, $\tau$ denotes the involution on the Khovanov complex $\CKh(D)$ induced by the involution of $D$. For example, let $D$ be the strongly invertible $3$-component unlink diagram, consisting of one symmetric circle intersecting the axis twice and two circles disjoint from the axis that are interchanged by $\tau$. The orientation $o$ that orients the left and the center circles counterclockwise and the right one clockwise has Lee cycle $\ca(D, o) = X \otimes X \otimes Y$ but $\tau \ca(D, o) = Y \otimes X \otimes X$, hence $\ca(D, o)$ is not $\tau$-invariant, contradicting Proposition~3.3 of [1].

To fix the error, Section~3.1 has been largely rewritten in the present version [2]. Let $O(D)$ denote the set of all orientations on the underlying unoriented diagram of $D$. Let $\tau(o) \in O(D)$ denote the orientation obtained by applying $\tau$ to $D$ reoriented by $o$, and $-o$ the reversal of $o$. Define 
\begin{align*}
    \Ot(D) &= \{\ o \in O(D) \mid \tau(o) = -o \ \}, \\
    \Otb(D) &= (O(D) \setminus \Ot(D)) \,/\, (\tau(o) \sim -o).
\end{align*}
The substantive changes are the following.
\begin{itemize}
    \item Proposition~3.3 of [1] is corrected as Proposition~3.5 of [2], which follows from the new Lemma~3.3.
    \item Definition~3.4 of [1] is replaced by Proposition~3.6 and Definition~3.7 of [2]: for $[o'] \in \Otb(D)$, the equivariant Lee cycle $\dca(D, o')$ is replaced by $(1 + \tau)\dca(D, o')$, which is now indeed a cycle in $\CKhI(D)$.
    \item Proposition~3.7 of [1] is corrected as Proposition~3.8 of [2], which gives a basis of $\KhI(D)$ indexed by $\Ot(D)$ and $\Otb(D)$. See also Remark~3.9 of [2] for the corrected rank formula.
    \item Proposition~3.6 of [1] is restated as Proposition~3.10 of [2], in terms of the induced long exact sequence.
    \item Propositions~3.8, 3.9, 3.10, 3.13, 4.1 and Definition~3.15 of [1] are restated for the relevant orientations only (Propositions~3.12, 3.26, 3.27, 3.30, 4.1 and Definition~3.14 of [2]).
\end{itemize}
Now, all results concerning strongly invertible knots and links, and in particular the main theorems of [1], remain valid as stated. The same correction applies to the $2$-periodic setting of Section~6 of [1]: Section~6 of [2] is revised accordingly.

\bigskip
\noindent
\textbf{Acknowledgements.}
A separate erratum has been submitted to the journal. The author thanks Yonghan Xiao for pointing out the error.

\bigskip
\noindent
\textbf{References}
\begin{list}{}{%
    \setlength{\labelwidth}{2em}%
    \setlength{\labelsep}{0.5em}%
    \setlength{\leftmargin}{2.5em}%
    \setlength{\itemindent}{0pt}%
}
    \item[{[1]}] T. Sano,
    \textit{Involutive Khovanov homology and equivariant knots},
    Algebr. Geom. Topol. \textbf{25} (2025), no.~8, 5059--5111.
    \url{https://doi.org/10.2140/agt.2025.25.5059}.

    \item[{[2]}] T. Sano,
    \textit{Involutive Khovanov homology and equivariant knots} (v5),
    \url{https://arxiv.org/abs/2404.08568v5}.
\end{list}

\clearpage
\pagenumbering{arabic}

    \maketitle
    \begin{abstract}
    For strongly invertible knots, we define an involutive version of Khovanov homology, and from it derive a pair of integer-valued invariants $(\underline{s}, \bar{s})$, which is an equivariant version of Rasmussen's $s$-invariant. Using these invariants, we reprove that the infinite family of knots $J_n$ introduced by Hayden each admits exotic pairs of slice disks. Our construction is intended to give a Khovanov-theoretic analogue of the formalism given by Dai, Mallick and Stoffregen in involutive knot Floer theory. 
\end{abstract}

    \setcounter{tocdepth}{2}

    \section{Introduction}
\label{sec:intro}

A \textit{strongly invertible knot} is a knot $K$ in $S^3$ equipped with an involution $\tau$ of $S^3$ that reverses the orientation of $K$. While strongly invertible knots have been studied for decades (see \cite{Sakuma:86}), recent developments in knot theory and low-dimensional topology gave rise to new directions in its research and its applications, as in  \cite{Watson:2017,Snape:2018,Hayden:2021,Hayden-Sundberg:2021,Alfieri-Boyle:2021,Boyle-Issa:2022,LS:2022,DMS:2023}. 

In \cite{DMS:2023}, Dai, Mallick and Stoffregen use \textit{involutive knot Floer homology} to define integer-valued invariants $\underline{V}{}_0^\tau, \overline{V}{}_0^\tau, \underline{V}{}_0^{\iota\tau}, \overline{V}{}_0^{\iota\tau}$ of strongly invertible knots, and show that there is an infinite family of knots $J_n$, each admitting exotic pairs of slice disks. Previously, the result for the special case $J_0 = 17nh_{73}$ (also known as the \textit{positron knot}) has been proved by Hayden and Sundberg \cite{Hayden-Sundberg:2021} using Khovanov homology, by distinguishing the cobordism maps induced from the two slice disks. In this paper, we adapt the formalism of involutive knot Floer homology to the Khovanov side, and recover the general result systematically.

For strongly invertible knots, we define an involutive version of Khovanov homology, called the \textit{involutive Khovanov homology}, and from it derive a pair of integer-valued invariants $(\underline{s}, \bar{s})$, which is an equivariant version of Rasmussen's \textit{$s$-invariant} for ordinary knots \cite{Rasmussen:2010}.

\begin{theorem}
\label{thm:1}
    For each $n \geq 0$, the strongly invertible knot $J_n$ of \Cref{fig:knotJn} has
    \[
        \ds(J_n) = 0 < 2 \leq \us(J_n).
    \]
\end{theorem}

\begin{figure}[t]
    \centering
    \resizebox{.24\textwidth}{!}{
        \tikzset{every picture/.style={line width=1pt}} 

\begin{tikzpicture}[x=1pt,y=1pt,yscale=-1,xscale=1]

\draw [thin, color=red ,draw opacity=1 ]   (75.88,147.75) .. controls (75.98,144.54) and (75.76,143.93) .. (75.63,139.5) .. controls (75.49,135.07) and (74.63,122) .. (75.88,117.25) .. controls (77.13,112.5) and (91.13,94.5) .. (93.63,89.5) .. controls (96.13,84.5) and (93.38,63.5) .. (96.63,60.25) .. controls (99.88,57) and (107.53,51.83) .. (114.38,48.25) .. controls (121.22,44.67) and (122.13,39.5) .. (121.88,34) .. controls (121.63,28.5) and (112.13,20.75) .. (107.88,20.25) .. controls (103.63,19.75) and (84.13,21.75) .. (78.63,30) .. controls (73.13,38.25) and (73.63,64.75) .. (73.63,73.5) .. controls (73.63,82.25) and (74.13,95) .. (76.13,101.5) .. controls (78.13,108) and (114.13,125.75) .. (118.63,128.5) .. controls (123.13,131.25) and (125.13,137.25) .. (124.63,147) .. controls (124.13,156.75) and (110.88,161.25) .. (98.88,160.75) .. controls (86.88,160.25) and (75.75,153.26) .. (75.88,147.75) -- cycle ;
\draw [thin, color=blue ,draw opacity=1 ]   (144.74,147.75) .. controls (144.63,144.54) and (144.85,143.93) .. (144.99,139.5) .. controls (145.12,135.07) and (145.99,122) .. (144.74,117.25) .. controls (143.49,112.5) and (129.48,94.5) .. (126.99,89.5) .. controls (124.49,84.5) and (127.24,63.5) .. (123.99,60.25) .. controls (120.74,57) and (113.09,51.83) .. (106.24,48.25) .. controls (99.39,44.67) and (98.49,39.5) .. (98.74,34) .. controls (98.99,28.5) and (108.49,20.75) .. (112.74,20.25) .. controls (116.99,19.75) and (136.49,21.75) .. (141.99,30) .. controls (147.49,38.25) and (146.99,64.75) .. (146.99,73.5) .. controls (146.99,82.25) and (146.49,95) .. (144.49,101.5) .. controls (142.49,108) and (106.49,125.75) .. (101.99,128.5) .. controls (97.49,131.25) and (95.49,137.25) .. (95.99,147) .. controls (96.49,156.75) and (109.74,161.25) .. (121.74,160.75) .. controls (133.74,160.25) and (144.87,153.26) .. (144.74,147.75) -- cycle ;

\draw  [thin, dash pattern={on 4.5pt off 4.5pt}]  (110.88,4.5) -- (110.38,193.25) ;

\begin{knot}[
  end tolerance=1pt,
  clip width=2.5pt,
]

\strand    (70.88,147.25) .. controls (69.63,133.75) and (80.13,132) .. (80.38,127.75) .. controls (80.63,123.5) and (70.2,119.37) .. (71.88,116.25) .. controls (73.55,113.13) and (76.17,108.14) .. (79.63,104.5) .. controls (83.08,100.86) and (87.63,92) .. (90.13,88.25) .. controls (92.63,84.5) and (90.38,60.5) .. (92.88,56) .. controls (95.38,51.5) and (114.88,45) .. (117.38,40.25) .. controls (119.88,35.5) and (115.08,27.99) .. (111.13,26.25) .. controls (107.17,24.51) and (96,25.03) .. (91.88,26.5) .. controls (87.75,27.97) and (86.38,27.75) .. (81.88,33.75) .. controls (77.38,39.75) and (80.38,83.75) .. (80.13,90.5) ;
\strand    (80.63,147.75) .. controls (80.73,144.54) and (80.51,143.93) .. (80.38,139.5) .. controls (80.24,135.07) and (71.13,131.5) .. (70.63,127.5) .. controls (70.13,123.5) and (80.2,119.12) .. (81.88,116) .. controls (83.55,112.88) and (86.17,107.89) .. (89.63,104.25) .. controls (93.08,100.61) and (95.63,93) .. (98.38,89.5) .. controls (101.13,86) and (98.13,63.5) .. (101.38,60.25) .. controls (104.63,57) and (114.03,52.83) .. (120.88,49.25) .. controls (127.72,45.67) and (127.63,35.75) .. (125.88,30.5) .. controls (124.13,25.25) and (118.38,21) .. (112.13,18.25) .. controls (105.88,15.5) and (81.36,16.35) .. (74.38,27.5) .. controls (67.39,38.65) and (69.13,71.5) .. (70.13,90.75) ;
\strand    (91.09,147.75) .. controls (91.09,141.45) and (90.77,134.85) .. (93.15,131.25) .. controls (95.53,127.65) and (116.61,117.21) .. (126.88,111.25) .. controls (137.14,105.29) and (139.88,102.5) .. (140.38,90.5) ;
\strand    (101.34,147.5) .. controls (101.34,142.95) and (100.21,137.89) .. (101.65,134.75) .. controls (103.1,131.61) and (108.33,128.95) .. (110.59,127.25) .. controls (112.86,125.55) and (110.96,126.99) .. (111.17,126.87) .. controls (111.38,126.74) and (132.82,116.92) .. (137.88,113) .. controls (142.93,109.08) and (151.13,101.75) .. (150.38,90.75) ;
\strand    (101.34,147.5) -- (120.68,147.5) ;
\strand    (80.63,147.75) -- (91.09,147.75) ;
\strand    (151.15,147.25) .. controls (164.22,158.06) and (151.63,168.5) .. (143.88,172.75) .. controls (136.13,177) and (117.34,177.03) .. (111.63,177.25) .. controls (105.91,177.47) and (85.63,176.75) .. (77.38,172.25) .. controls (69.13,167.75) and (57.63,158.75) .. (70.88,147.25) ;
\strand    (149.63,147.25) .. controls (150.88,133.75) and (140.38,132) .. (140.13,127.75) .. controls (139.88,123.5) and (150.3,119.37) .. (148.63,116.25) .. controls (146.95,113.13) and (144.33,108.14) .. (140.88,104.5) .. controls (137.43,100.86) and (132.88,92) .. (130.38,88.25) .. controls (127.88,84.5) and (130.13,60.5) .. (127.63,56) .. controls (125.13,51.5) and (105.63,45) .. (103.13,40.25) .. controls (100.63,35.5) and (105.42,27.99) .. (109.38,26.25) .. controls (113.33,24.51) and (124.5,25.03) .. (128.63,26.5) .. controls (132.75,27.97) and (134.13,27.75) .. (138.63,33.75) .. controls (143.13,39.75) and (140.13,83.75) .. (140.38,90.5) ;
\strand    (139.88,147.75) .. controls (139.78,144.54) and (139.99,143.93) .. (140.13,139.5) .. controls (140.26,135.07) and (149.38,131.5) .. (149.88,127.5) .. controls (150.38,123.5) and (140.3,119.12) .. (138.63,116) .. controls (136.95,112.88) and (134.33,107.89) .. (130.88,104.25) .. controls (127.43,100.61) and (124.88,93) .. (122.13,89.5) .. controls (119.38,86) and (122.38,63.5) .. (119.13,60.25) .. controls (115.88,57) and (106.47,52.83) .. (99.63,49.25) .. controls (92.78,45.67) and (92.88,35.75) .. (94.63,30.5) .. controls (96.38,25.25) and (102.13,21) .. (108.38,18.25) .. controls (114.63,15.5) and (139.14,16.35) .. (146.13,27.5) .. controls (153.12,38.65) and (151.13,79.75) .. (150.38,90.75) ;
\strand    (129.41,147.75) .. controls (129.41,141.45) and (129.73,134.85) .. (127.35,131.25) .. controls (124.97,127.65) and (103.89,117.21) .. (93.63,111.25) .. controls (83.37,105.29) and (80.63,102.5) .. (80.13,90.5) ;
\strand    (119.16,147.5) .. controls (119.16,142.95) and (120.29,137.89) .. (118.85,134.75) .. controls (117.4,131.61) and (112.17,128.95) .. (109.91,127.25) .. controls (107.64,125.55) and (109.54,126.99) .. (109.33,126.87) .. controls (109.12,126.74) and (87.68,116.92) .. (82.63,113) .. controls (77.57,109.08) and (69.38,101.75) .. (70.13,90.75) ;
\strand    (139.88,147.75) -- (129.41,147.75) ;

\flipcrossings{1,7,8,13,14,3,5,9,11,15,16,19,20,24}

\end{knot}

\draw  [fill={rgb, 255:red, 255; green, 255; blue, 255 }  ,fill opacity=1 ] (64.88,64) -- (106.13,64) -- (106.13,81.25) -- (64.88,81.25) -- cycle ;
\draw  [fill={rgb, 255:red, 255; green, 255; blue, 255 }  ,fill opacity=1 ] (116.13,64) -- (157.38,64) -- (157.38,81.25) -- (116.13,81.25) -- cycle ;

\draw (78,68) node [anchor=north west][inner sep=0.75pt]    {$-n$};
\draw (130,68) node [anchor=north west][inner sep=0.75pt]    {$-n$};

\end{tikzpicture}
    }
    \caption{The knot $J_n$ and the two slice disks $D_n, D'_n$ obtained by compressing along the two colored circles.}
    \label{fig:knotJn}
\end{figure}

As we shall see later, \Cref{thm:1} implies that each $J_n$ never admits a \textit{simple isotopy-equivariant} slice disk. In particular, the two slice symmetric disks $D_n, D'_n$ of $J_n$ depicted in \Cref{fig:knotJn} are not smoothly isotopic rel $J_n$. Combining \Cref{thm:1} with the fact that $D_n$ and $D'_n$ are topologically isotopic, which is proved in \cite{Hayden:2021} using the result of Conway and Powell \cite{Conway-Powell:2021}, we may conclude that these disks form an exotic pair of slice disks of $J_n$. This argument is completely analogous to the proof of \cite[Theorem 7.11]{DMS:2023}. 

The definition of involutive Khovanov homology follows the formalism of involutive knot Floer homology \cite{Hendricks-Manolescu:2017,Zemke:2019,Alfieri-Kang-Stipsicz:2019,DMS:2023}. Given a strongly invertible knot diagram $(D, \tau)$, there is an induced involution $\tau$ on the Khovanov complex $\CKh(D)$ over $\FF_2$. Using this we define,

\begin{definition}
\label{def:1-1}
    The \textit{involutive Khovanov complex} of $(D, \tau)$ is defined by
    \[
        \CKhI(D, \tau) = \Cone(\ \CKh(D) \xrightarrow{Q(1 + \tau)} Q\CKh(D)\ )
    \]
    where $Q$ is a formal variable of $Q^2 = 0$. The homology of $\CKhI(D, \tau)$ is denoted $\KhI(D, \tau)$ and is called the \textit{involutive Khovanov homology} of $(D, \tau)$.
\end{definition}

\begin{theorem}
\label{thm:kh-inv}
    The isomorphism class of $\KhI(D, \tau)$ is an invariant of the strongly invertible knot $(K, \tau)$. 
\end{theorem}

Hereafter, we make $\tau$ implicit and omit it from $\CKhI$ and $\KhI$. We note that \Cref{def:1-1} is also valid for the deformed versions of Khovanov homology obtained by replacing the defining Frobenius algebra, and \Cref{thm:kh-inv} still holds. There are also reduced versions, denoted $\rCKhI(D)$ and $\rKhI(D)$, which are also invariants of strongly invertible knots. 

Our equivariant invariant $(\ds, \us)$ is defined using the Bar-Natan's deformation, given by the Frobenius algebra $A = R[X]/(X(X + H))$ over the ring $R = \FF_2[H]$ with $\deg(H) = -2$. Let us first recall the definition of the \textit{$\FF_2$-Rasmussen invariant} $s = s^{\FF_2}$, as characterized in \cite[Proposition 3.8]{KWZ:2019}. For any knot $K$, it is known that the reduced Bar-Natan homology $\rBN(K)$ has a single $\FF_2[H]$-tower in homological grading $0$. Then $s(K)$ is defined as the quantum grading of its generator:
\[
    \rBN(K) \isom h^0q^{s(K)}\FF_2[H] \oplus (\Tor). 
\]

For a strongly invertible knot $K$, it is proved that the \textit{reduced involutive Bar-Natan homology} $\rBNI(K)$ has two $\FF_2[H]$-towers, one in homological grading $0$ and another in homolgical grading $1$. Thus we may define $\ds(K), \us(K)$ as the quantum gradings of their generators:
\[
    \rBNI(K) \isom h^0q^{\ds(K)}\FF_2[H] \oplus h^1q^{\us(K)}\FF_2[H] \oplus (\Tor). 
\]
The pair $(\ds(K), \us(K))$ is called the \textit{equivariant Rasmussen invariant} of $K$. 

\begin{theorem}
\label{thm:s-inv}
    The equivariant Rasmussen invariant $(\ds(K), \us(K))$ is an invariant of the strongly invertible knot $K$, satisfying
    \[
        \ds(K) \leq s(K) \leq \us(K).
    \]
\end{theorem}

\begin{table}[t]
    \resizebox{\textwidth}{!}{
    \centering
\begin{tabular}{r|llllllllllllll}
$18$ & $.$ & $.$ & $.$ & $.$ & $.$ & $.$ & $.$ & $.$ & $.$ & $.$ & $.$ & $.$ & $\FF$ & $\FF$ \\
$16$ & $.$ & $.$ & $.$ & $.$ & $.$ & $.$ & $.$ & $.$ & $.$ & $.$ & $.$ & $.$ & $.$ & $.$ \\
$14$ & $.$ & $.$ & $.$ & $.$ & $.$ & $.$ & $.$ & $.$ & $.$ & $\FF$ & $\FF^{2}$ & $\FF$ & $.$ & $.$ \\
$12$ & $.$ & $.$ & $.$ & $.$ & $.$ & $.$ & $.$ & $.$ & $\FF$ & $\FF^{2}$ & $\FF$ & $.$ & $.$ & $.$ \\
$10$ & $.$ & $.$ & $.$ & $.$ & $.$ & $.$ & $.$ & $\FF$ & $\FF$ & $.$ & $.$ & $.$ & $.$ & $.$ \\
$8$ & $.$ & $.$ & $.$ & $.$ & $.$ & $\FF$ & $\FF^{3}$ & $\FF^{3}$ & $\FF$ & $.$ & $.$ & $.$ & $.$ & $.$ \\
$6$ & $.$ & $.$ & $.$ & $.$ & $.$ & $\FF$ & $\FF$ & $.$ & $.$ & $.$ & $.$ & $.$ & $.$ & $.$ \\
$4$ & $.$ & $.$ & $.$ & $\FF$ & $\FF^{2}$ & $\FF$ & $.$ & $.$ & $.$ & $.$ & $.$ & $.$ & $.$ & $.$ \\
${\color{orange} 2}$ & $.$ & $.$ & $\FF$ & ${\color{orange} \FF[H]} \oplus \FF$ & $\FF$ & $.$ & $.$ & $.$ & $.$ & $.$ & $.$ & $.$ & $.$ & $.$ \\
${\color{orange} 0}$ & $.$ & $.$ & ${\color{orange} \FF[H]}$ & $.$ & $.$ & $.$ & $.$ & $.$ & $.$ & $.$ & $.$ & $.$ & $.$ & $.$ \\
$-2$ & $\FF$ & $\FF$ & $.$ & $.$ & $.$ & $.$ & $.$ & $.$ & $.$ & $.$ & $.$ & $.$ & $.$ & $.$ \\
\hline
$ $ & $-2$ & $-1$ & ${\color{orange} 0}$ & ${\color{orange} 1}$ & $2$ & $3$ & $4$ & $5$ & $6$ & $7$ & $8$ & $9$ & $10$ & $11$ \\
\end{tabular}
    }
    \caption{$\rBNI(J_0)$}
    \label{tab:knotJ}
\end{table}

As is true for the ordinary $s$-invariant, our equivariant $s$-invariant is directly computable using computers. In particular, we have computed the invariants for the three strongly invertible knots given in \cite{Hayden-Sundberg:2021}, which are proved therein to admit pairs of non-smoothly isotopic slice disks. 

\begin{proposition}
\label{prop:calc}
    The three strongly invertible slice knots $K = m(9_{46})$, $15n_{103488}$, $17nh_{73}$ have 
    \[
        \ds(K) = 0 < 2 = \us(K).
    \]
\end{proposition}

\Cref{tab:knotJ} shows the computed $\rBNI$ for $J_0 = 17nh_{73}$, from which we can see that there are indeed two $\FF_2[H]$ summands, one in bigrading $(0, 0)$ and another in $(1, 2)$, and the remaining summands are copies of $\mathbb{F}_2[H]/(H) = \FF_2$. The computed $\rBNI$ for the other two knots will be given in \Cref{sec:5}. Combined with the later stated \Cref{cor:intro-s-obst}, we recover that these knots admit pairs of non-smoothly isotopic slice disks. 

We further study properties of the equivariant invariant $(\ds, \us)$, that are analogous to that of the ordinary $s$. 

\begin{proposition}
    For the mirror $K^*$ of $K$, we have 
    \[
        \ds(K^*) = -\us(K)
    \]
\end{proposition}

\begin{proposition}
    For another strongly invertible knot $K'$, 
    \[
        \ds(K) + \ds(K') 
        \leq \ds(K \# K') 
        \leq \ds(K) + \us(K') 
        \leq \us(K \# K') 
        \leq \us(K) + \us(K') 
    \]
    Here $\#$ denotes the equivariant connected sum. 
\end{proposition}

\begin{proposition}
    Let $K^+, K^-$ be strongly invertible knots such that $K^-$ is obtained by applying an `equivariant negative crossing change' to $K^+$ (see \Cref{def:equiv-x-ch}). Then
    \[
        \ds(K^-) \leq \ds(K^+) \leq \ds(K^-) + 2a
    \]
    where $a = 1$ if the move is performed on-axis, and $a = 2$ if performed off-axis. The same statement holds for $\us$. 
\end{proposition}

\begin{proposition}
\label{prop:intro-neg-gxch}
    Let $K^+, K^-$ be strongly invertible knots such that $K^-$ is obtained by applying a `$4$-strand equivariant generalized negative crossing change' to $K^+$ (see \Cref{def:equiv-gen-x-ch}). Then
    \[
        \ds(K^-) \leq \ds(K^+) \leq \ds(K^-) + 4a
    \]
    where $a = 1$ if the move is performed on-axis, and $a = 2$ if performed off-axis. The same statement holds for $\us$. 
\end{proposition}

\begin{proposition}
\label{prop:torus-knot}
    The positive $(p, q)$-torus knot $T_{p, q}$ has 
    \[
        \ds(T_{p, q}) 
        = \us(T_{p, q})
        = (p - 1)(q - 1)
    \]
    with respect to the unique inverting involution.
\end{proposition}

The lower bound for the $4$-genus
\[
    |s(K)| \leq 2 g_4(K)
\]
is one of the significant properties of $s$, that led to the reproof of the \textit{Milnor conjecture} \cite{Milnor:1968}. This is an implication of its behavior under cobordisms. In order to prove an analogous result for the equivariant invariant, we impose the following condition on cobordisms. 

\begin{definition}   
    Let $L, L'$ be two strongly invertible links sharing the same involution $\tau$ of $S^3$. A \textit{simple equivariant cobordism} between $L$ and $L'$ is an oriented cobordism $S$ in $S^3 \times I$ between $L, L'$ satisfying $(\tau \times \id)(S) = S$. A \textit{simple isotopy-equivariant cobordism} between $L, L'$ is defined similarly, except that $(\tau \times \id)(S)$ is only required to be isotopic to $S$ rel boundary. 
\end{definition}

\begin{proposition}
\label{prop:intro-s-cobordism}
    Suppose there is a connected, simple isotopy-equivariant cobordism $S$ between two strongly invertible knots $K, K'$. Then,
    \begin{align*}
        |\ds(K) - \ds(K')| &\leq -\chi(S), \\
        |\us(K) - \us(K')| &\leq -\chi(S).
    \end{align*}
    In particular, $\ds$ and $\us$ are invariant under simple isotopy-equivariant concordances. 
\end{proposition}

We may naturally define the \textit{simple equivariant genus} $\widetilde{sg}_4(K)$ and the \textit{simple isotopy-equivariant genus} $\widetilde{sig}_4(K)$ of a strongly invertible knot $K$. Obviously, we have inequalities
\[
    \begin{tikzcd}[row sep = 0.7em, column sep = 1em]
    & \tilde{g}_4(K) \arrow[dr, "\leq"{rotate = -30, description}, phantom] & \\
    g_4(K) \arrow[ur, "\leq"{rotate = 30} description, phantom] \arrow[dr, "\leq"{rotate = -30} description, phantom] & & \widetilde{sg}_4(K) \arrow[r, "\leq" description, phantom] & \tilde{g}_3(K) \\
    & \widetilde{sig}_4(K) \arrow[ur, "\leq"{rotate = 30} description, phantom]&                          
    \end{tikzcd}
\]
where $g_4$ is the ordinary $4$-genus, and $\tilde{g}_3$, $\tilde{g}_4$ are the equivariant $3$- and $4$-genera respectively. \Cref{prop:intro-s-cobordism} implies,

\begin{corollary}
\label{cor:intro-s-bound-g}
    Both $|\ds(K)|$ and $|\us(K)|$ bound $2\widetilde{sig}_4(K)$ from below. 
\end{corollary}

\begin{corollary}
\label{cor:intro-s-obst}
    If either $|\ds(K)|$ or $|\us(K)|$ is greater than $2g_4(K)$, then no slice surface $S$ of $K$ realizing $g(S) = g_4(K)$ are simple isotopy-equivariant. In particular, $S$ and $(\tau \times \id)(S)$ are not smoothly isotopic rel $K$.  
\end{corollary}

Now \Cref{thm:1} and the implication that the two slice disks $D_n, D'_n$ of \Cref{fig:knotJn} are not smoothly isotopic rel $J_n$ follows from \Cref{thm:s-inv}, \Cref{prop:calc,prop:intro-neg-gxch} and \Cref{cor:intro-s-obst}, for we have 
\[
    0 = \ds(J_0) \leq \ds(J_1) \leq \cdots \leq \ds(J_n) \leq s(J_n) = 0
\]
and
\[
    2 = \us(J_0) \leq \us(J_1) \leq \cdots \leq \us(J_n).
\]

We finally remark that the restriction to \textit{simple} equivariant cobordisms is due to the way the cobordism maps on Khovanov homology are defined. However, this restriction may lead to a deeper understanding of the \textit{equivariant concordance group} $\tilde{\mathcal{C}}$, defined by Sakuma in \cite[Section 4]{Sakuma:86}. It is known that the Smith conjecture fails in higher dimensions, in particular, there is an involution $\bar{\tau}$ on $S^4$ whose fixed point set is a knotted sphere \cite{Gordon:1974}. This implies that non-simple equivariant cobordisms do exist. We may define the \textit{simple equivariant concordance group} $\widetilde{s\mathcal{C}}$, whose elements are simple equivariant concordance classes of directed strongly invertible knots, with the operation given by the equivariant connected sum. We question whether the surjective homomorphism $\widetilde{s\mathcal{C}} \rightarrow \tilde{\mathcal{C}}$ is injective or not, and also the existence of knots such that $\tilde{g}_4(K) < \widetilde{sg}_4(K)$. 

\medskip
\noindent 
\textbf{Organization.}
This paper is organized as follows. In \Cref{sec:2}, we define the \textit{involutive Khovanov complex} $\CKhI$ for \textit{involutive link diagrams}, and prove its invariance up to chain homotopy under the \textit{involutive Reidemeister moves}. The reduced version is also introduced therein. In \Cref{sec:3}, we define the \textit{equivariant Rasmussen invariant} $(\ds, \us)$ for \textit{strongly invertible links}, not in the way presented above, but instead using the divisibility of the \textit{equivariant Lee classes}. In \Cref{sec:4}, the above-stated properties of $(\ds, \us)$ are proved, including the proof that the two definitions of $(\ds, \us)$ coincide. In \Cref{sec:5}, the proof of \Cref{thm:1} will be restated, followed by an observation of the result. In \Cref{sec:6}, we briefly state that an analogous construction is possible for $2$-periodic links by considering a modified involution. In \Cref{appendix}, we give a list of $\KhI$ and $\BNI$ for prime knots with up to $7$ crossings, obtained by direct computations. 
    
\medskip
\noindent 
\textbf{Acknowledgements.}
The author thanks Masaki Taniguchi, Kouki Sato and Makoto Sakuma for the helpful discussions, and the anonymous referee for their helpful comments. The author thanks Yonghan Xiao for pointing out an error in Section 3.1 of the previous version. The author was supported by JSPS KAKENHI Grant Numbers 23K12982, RIKEN iTHEMS Program and academist crowdfunding.

    \section{Involutive Khovanov homology}
\label{sec:2}

Throughout this paper, we work in the smooth category and assume all objects and maps to be smooth. We assume that the reader is familiar with the construction of Khovanov homology \cite{Khovanov:2000}, in particular Bar-Natan's reformulation given in \cite{BarNatan:2004}. Let $R$ be a commutative ring with unity of $\fchar{R} = 2$, and $A$ a Frobenius algebra of the form $A = R[X]/(X(X + h))$ with $h \in R$, determined by $\epsilon(1) = 0,\ \epsilon(X) = 1$. For a link diagram $D$, let $\CKh(D; R, h)$ denote the Khovanov chain complex of $D$ obtained from the above Frobenius algebra $A$, and $\Kh(D; R, h)$ its homology. Typically we consider the following cases:
\[
    (R, h) = (\FF_2, 0),\ (\FF_2, 1),\ (\FF_2[H], H),
\]
each corresponding to the original Khovanov homology \cite{Khovanov:2000}, the filtered and the bigraded Bar-Natan homology \cite{BarNatan:2004} over $\FF_2$. For the third case, we assume that $H$ has $\deg(H) = -2$, so that the chain complex admits a bigrading. We usually make $(R, h)$ implicit and omit it from the notations. 

\subsection{Definition}
\label{subsec:1-1}

Hereafter, we assume that any involution $\tau$ on $S^3$ has fixed point set $S^1$. (From the resolution of the Smith conjecture \cite{Waldhausen:1969,Morgan-Bass:1979}, the fixed point set of an involution of $S^3$ with non-empty fixed-point set is necessarily an unknotted circle.) The fixed-point set $\Fix(\tau)$ is called the \textit{axis} of $\tau$. We follow \cite{Lobb-Watson:2021} for the definitions of involutive links and their equivalence. 

\begin{definition}
    An \textit{involutive link} $(L, \tau)$ in $S^3$ is an oriented link $L$ equipped with an involution $\tau$ on $S^3$ such that $\tau(L) = L$, possibly altering the orientations on some of its components. Two involutive links are \textit{equivalent} if they are isotopic through involutive links. 
    An involutive link $(L, \tau)$ is \textit{strongly invertible} if $\tau$ reverses the orientation of $L$, \textit{$2$-periodic} if $\tau$ preserves the orientation of $L$. Strongly invertible links and $2$-periodic links are together called \textit{equivariant links}. 
\end{definition}

We will draw involutive link diagrams so that the axis is projected as a straight vertical line, and the diagram is symmetric with respect to the axis, as in \Cref{fig:knotJn}. For an involutive link diagram $(D, \tau)$, the induced involution $\tau$ on $\CKh(D)$ is defined as follows. For each state $s$ for $D$, there is a unique state $s'$ such that $D(s)$ and $D(s')$ are symmetric with respect to the axis $\Fix(\tau)$. For each standard generator $x$ in state $s$, its image $\tau(x)$ is defined to be the generator in the state $s'$ that corresponds to $x$ under the symmetry. It is obvious that $\tau$ commutes with $d$ and is an involution on $\CKh(D)$ with bidegree $(0, 0)$. 

\begin{definition}
    The \textit{involutive Khovanov complex} of $(D, \tau)$ is defined by
    \[
        \CKhI(D, \tau) = \Cone(\ \CKh(D) \xrightarrow{Q(1 + \tau)} Q\CKh(D)\ )
    \]
    where $Q$ is a formal variable of $Q^2 = 0$. The homology of $\CKhI(D, \tau)$ is denoted $\KhI(D, \tau)$ and is called the \textit{involutive Khovanov homology} of $(D, \tau)$. 
\end{definition}

\begin{remark}
    Our $\CKhI(D, \tau)$ is different from the triply graded complex $\CKh_\tau(D)$ given in \cite{Lobb-Watson:2021}, where $\CKh_\tau(D) = \CKh(D; \FF_2, 0)$ as $\FF_2$-modules, but the differential is defined as $\del = d + 1 + \tau$.
\end{remark}

Hereafter we make $\tau$ implicit and omit it from $\CKhI$ and $\KhI$. When explicitly describing elements and maps, we often regard $\CKhI(D)$ as the direct sum of two copies of $\CKh(D)$, 
\[
    \CKhI(D) = \CKh(D) \oplus \CKh(D)[1]
\]
with the differential given by 
\[
    \begin{pmatrix}
        d & \\
        1 + \tau & d
    \end{pmatrix}.
\]
As in the non-involutive case, the differential preserves the quantum grading if $h = 0$ or $\deg(h) = -2$; or is quantum grading non-decreasing if $h \neq 0$ and $\deg(h) = 0$. In the former case, we regard $\CKhI(D)$ as a bigraded complex, and in the latter as a filtered complex. 

Formally, a chain complex over $\FF_2$ equipped with an involution $\tau$ is called a \textit{$\tau$-complex}. For $\tau$-complex $C$, let $C_\tau$ denote the complex
\[
    C_\tau = \Cone(1 + \tau). 
\]
This correspondence is functorial in the following sense. For homotopic chain maps $f, g$ (of any degree) with homotopy $h$, we write
\[
    f \htpy_h g
\]
to mean 
\[
    f + g = dh + hd.
\]
A chain map $f$ between $\tau$-complexes $C, C'$ is \textit{homotopy $\tau$-equivariant} if $\tau f \htpy f \tau$. A \textit{$\tau$-conjugate} of $f$ is defined by 
\[
    f^\tau = \tau f \tau.
\]
Obviously $f$ is homotopy $\tau$-equivariant if and only if $f \htpy f^\tau$. 
Two homotopy $\tau$-equivariant chain maps $f, g$ with homotopies
\[
    \tau f \htpy_{h_f} f \tau, \ 
    \tau g \htpy_{h_g} g \tau
\]
are \textit{coherently homotopic} with homotopy $h$ if  $f \htpy_h g$ and
\[
    \tau h + h \tau \htpy h_g + h_f.
\]
The following lemmas are easy to verify and will be used throughout this paper.

\begin{lemma}
\label{lem:induced-ch-map}
    A homotopy $\tau$-equivariant chain map
        $f\colon C \rightarrow C'$ 
    with homotopy
        $\tau f \htpy_{h_f} f \tau$
    induces a chain map
    \[
        f_\tau \colon C_\tau \rightarrow C'_\tau
    \]
    given by 
    \[
        f_\tau = \begin{pmatrix}
            f & \\
            h_f & f
        \end{pmatrix}.
    \]
    In particular we have 
    \[
        1_\tau = 1.
    \]
    For another homotopy $\tau$-equivariant chain map
        $g\colon C' \rightarrow C''$ 
    with homotopy
        $\tau g \htpy_{h_g} g \tau$, 
    the composition $gf$ is homotopy $\tau$-equivariant with homotopy
    \[
        h_{gf} = h_g f + g h_f
    \]
    and the above correspondence give
    \[
        (gf)_\tau = g_\tau f_\tau.
    \]
\end{lemma}

\begin{lemma}
\label{lem:induced-ch-htpy}
    Suppose 
        $f, g\colon C \rightarrow C'$
    are homotopy $\tau$-equivariant chain maps with homotopies
        $\tau f \htpy_{h_f} f \tau$,\
        $\tau g \htpy_{h_g} g \tau$
    and are coherently homotopic with homotopies $h, k$ satisfying
        $f \htpy_h g$,\ 
        $\tau h + h \tau \htpy_k h_g + h_f$.
    Then the induced maps
    \[
        f_\tau, g_\tau
        \colon C_\tau \rightarrow C'_\tau
    \]
    are homotopic with homotopy 
    \[
        \begin{pmatrix}
            h & \\
            k & h
        \end{pmatrix}. 
    \]
\end{lemma}

\begin{lemma}
\label{lem:induced-ch-htpy-eq}
    Let $f$ be a homotopy $\tau$-equivariant homotopy equivalence with a homotopy $\tau$-equivariant homotopy inverse $g$ and coherent homotopies
    \[
        g f \htpy 1,\ 
        f g \htpy 1.
    \]
    Then the induced map 
    \[
        f_\tau
        \colon C_\tau \rightarrow C'_\tau
    \]
    is a homotopy equivalence with homotopy inverse $g_\tau$. 
\end{lemma}

\subsection{Invariance}

\begin{figure}[t]
    \centering
    \resizebox{\textwidth}{!}{
    \input{tikzpictures/IR-moves}
    }
    \caption{Involutive Reidemeister moves}
    \label{fig:inv-reidemeister}
\end{figure}

Next, we prove the invariance of $\KhI$. In \cite{Lobb-Watson:2021}, Lobb and Watson introduced the eight \textit{involutive Reidemeister moves} (see \Cref{fig:inv-reidemeister}) and proved that two involutive links are equivalent if and only if those diagrams are related by a sequence of involutive Reidemeister moves and equivariant planar isotopies. \Cref{thm:kh-inv} follows from the following stronger proposition.

\begin{proposition}
\label{prop:invariance}
    Let $D, D'$ be two involutive link diagrams related by one of the involutive Reidemeister moves. Then there is a chain homotopy equivalence
    \[
        \rho\colon\CKhI(D) \rightarrow \CKhI(D')
    \]
    such that the following diagram commutes
    \[ 
\begin{tikzcd}
{Q\CKh(D)[1]} \arrow[r, hook] \arrow[d, "\rho"] & \CKhI(D) \arrow[r, two heads] \arrow[d, "\rho"] & \CKh(D) \arrow[d, "\rho"] \\
{Q\CKh(D')[1]} \arrow[r, hook]                   & \CKhI(D') \arrow[r, two heads]                            & \CKh(D').
\end{tikzcd} 
    \]
    Here the two vertical arrows on the left and the right are given by the composition of the standard chain homotopy equivalences given in \cite{BarNatan:2004} that corresponds to some decomposition of the involutive move into a sequence of ordinary Reidemeister moves. 
\end{proposition}

In order to prove \Cref{prop:invariance}, instead of explicitly constructing chain homotopy equivalences and chain homotopies for each of the moves, we prove the existences of the desired maps in a uniform way. 

\subsubsection{General Strategy}

Any involutive Reidemeister move can be decomposed into a sequence of ordinary Reidemeister moves. By composing the corresponding maps given in \cite{BarNatan:2004} we get explicit homotopy equivalences $F, G$ between $\CKh(D)$ and $\CKh(D')$, together with homotopies $GF \htpy_H I$ and $FG \htpy_{H'} I$. In general, these maps are not strictly $\tau$-equivariant. Nevertheless, we can prove the following.

\begin{claim}
\label{cl:htpy-1}
    $F$ and $G$ are homotopy $\tau$-equivariant, i.e.\ there exists homotopies $\tau F \htpy_{h_F} F \tau$ and $\tau G \htpy_{h_G} G \tau$. 
\end{claim}

\begin{claim}
\label{cl:htpy-2}
    $H$ and $H'$ are coherent homotopies, i.e.\ there exists homotopies $\tau H + H \tau \htpy h_G F + G h_F$ and $\tau H' + H' \tau \htpy h_F G + F h_G$. 
\end{claim}

Then the invariance follows from \Cref{lem:induced-ch-htpy-eq}. Thus the proof is reduced to proving \Cref{cl:htpy-1,cl:htpy-2} for each of the moves. 
Recall from \cite[Definition 8.5]{BarNatan:2004} that a tangle diagram $T$ is \textit{$\Kh$-simple} if any degree $0$ automorphism (up to homotopy) of $C(T)$ is homotopic to $\pm I$. Here $C(T)$ is the \textit{formal Khovanov complex} introduced in \cite{BarNatan:2004}, which is a complex in the additive closure of the category $\Cob^3(\del T)$. Here we strengthen the condition as follows.
\begin{definition}
    A chain complex $C$ (in any additive category) is \textit{simple} if any degree $0$ automorphism (up to homotopy) of $C$ is homotopic to $\pm I$, and any degree $n \neq 0$ self-chain map $C \rightarrow C$ is null homotopic. A tangle diagram $T$ is \textit{$\Kh$-simple} if the complex $C(T)$ is simple.
\end{definition}

A complex $C$ being simple is equivalent to the condition that the homology of the End-complex $\End(C) = \Hom(C, C)$ has only $\pm I$ as the degree $0$ multiplicative units (with respect to the composition) and its homology is supported only on degree $0$. One can easily prove that \cite[Lemmas 8.6 -- 8.9]{BarNatan:2004} also holds for our stronger definition, namely,
\begin{lemma}
    Simplicity is preserved under chain homotopy equivalences.
\end{lemma}
\begin{lemma}
    Parings are $\Kh$-simple. Here, a \textit{pairing} is a tangle diagram that has no crossings and no closed components.
\end{lemma}
\begin{lemma}
    A tangle diagram $T$ is $\Kh$-simple if and only if $TX$ is $\Kh$-simple. Here $TX$ is a tangle diagram obtained by adding one extra crossing $X$ somewhere along the boundary of $T$.
\end{lemma}

The following lemmas will also be useful. Here $C$ and $C'$ are complexes in any additive category. 

\begin{lemma}
\label{lem:htpy-1}
    Suppose there are two degree $0$ chain homotopy equivalences $F, F'\colon C \rightarrow C'$ with homotopy inverses $G, G'$ respectively. If $C$ is simple, then $F$ and $F'$, as well as $G$ and $G'$, are homotopic up to sign. 
\end{lemma}

\begin{proof}
    The map $GF'$ is an automorphism on $C$ with a homotopy inverse $G'F$. Since $C$ is simple, we have $GF' \htpy \pm I$. Assuming that $GF' \htpy I$, we get $F \htpy FGF' \htpy F'$ and $G \htpy GF'G' \htpy G'$. 
\end{proof}

\begin{lemma}
\label{lem:htpy-2}
    Suppose there are degree $0$ chain maps $F, F'\colon C \rightarrow C'$ and $G, G'\colon C' \rightarrow C$, together with homotopies $F \htpy_{h_F} F'$, $G \htpy_{h_G} G'$, $GF \htpy_H I$, $G'F' \htpy_{H'} I$. If $C$ is simple, then $H - H'$ and $h_G F + G' h_F$ are homotopic. 
\end{lemma}

\begin{proof}
    On one hand, we have 
    \[
        GF - G'F' = d(H - H') + (H - H')d.
    \]
    On the other hand, we have 
    \begin{align*}
        GF - G'F' &= (G - G')F + G'(F - F') \\
            &= (d h_G + h_G d) F + G'(d h_F + h_F d) \\
            &= d(h_G F + G' h_F) + (h_G F + G' h_F)d.
    \end{align*}
    Thus $(H - H') - (h_G F + G' h_F)$ is a degree $-1$ chain map, which is necessarily null-homotopic since $C$ is simple. Therefore $H - H' \htpy h_G F + G' h_F$. 
\end{proof}

\subsubsection{IR1 -- IR3}
\begin{proof}[Proof of \Cref{prop:invariance}, cases IR1 -- IR3]
    \begin{figure}[t]
        \centering
        \resizebox{.7\linewidth}{!}{
            \tikzset{every picture/.style={line width=0.75pt}} 

\begin{tikzpicture}[x=1,y=1pt,yscale=-0.8,xscale=1]

\begin{knot}[
  clip width=2.5pt,
  consider self intersections
]

\strand    (41.5,11.05) .. controls (56.5,11.5) and (78.61,58.33) .. (78.5,26.5) .. controls (78.39,-5.33) and (55.5,41.5) .. (41.5,41.61) ;
\strand    (132,11.05) .. controls (117,11.5) and (94.89,58.33) .. (95,26.5) .. controls (95.11,-5.33) and (118,41.5) .. (132,41.61) ;

\strand    (211.5,10.55) .. controls (226.5,11) and (248.61,57.83) .. (248.5,26) .. controls (248.39,-5.83) and (225.5,41) .. (211.5,41.11) ;
\strand    (302,10.55) .. controls (287,11) and (264.89,57.83) .. (265,26) .. controls (265.11,-5.83) and (288,41) .. (302,41.11) ;

\strand    (131,70.55) .. controls (116,71) and (93.89,117.83) .. (94,86) .. controls (94.11,54.17) and (117,101) .. (131,101.11) ;

\strand    (211.5,70.55) .. controls (226.5,71) and (248.61,117.83) .. (248.5,86) .. controls (248.39,54.17) and (225.5,101) .. (211.5,101.11) ;

\flipcrossings{2,4,5}

\end{knot}

\draw    (42.5,71.55) .. controls (61,72) and (65,98.5) .. (42.5,102.11) ;
\draw    (301.03,71.05) .. controls (283.53,72.5) and (278.53,98) .. (301.03,101.61) ;

\draw    (42.5,131.05) .. controls (61,131.5) and (65,158) .. (42.5,161.61) ;
\draw    (131.03,131.05) .. controls (113.53,132.5) and (108.53,158) .. (131.03,161.61) ;
\draw    (212.5,130.55) .. controls (231,131) and (235,157.5) .. (212.5,161.11) ;
\draw    (301.03,130.55) .. controls (283.53,132) and (278.53,157.5) .. (301.03,161.11) ;

\draw    (153.5,26.5) -- (189.5,26.5) ;
\draw [shift={(191.5,26.5)}, rotate = 180] [color={rgb, 255:red, 0; green, 0; blue, 0 }  ][line width=0.75]    (10.93,-3.29) .. controls (6.95,-1.4) and (3.31,-0.3) .. (0,0) .. controls (3.31,0.3) and (6.95,1.4) .. (10.93,3.29)   ;
\draw [shift={(151.5,26.5)}, rotate = 0] [color={rgb, 255:red, 0; green, 0; blue, 0 }  ][line width=0.75]    (10.93,-3.29) .. controls (6.95,-1.4) and (3.31,-0.3) .. (0,0) .. controls (3.31,0.3) and (6.95,1.4) .. (10.93,3.29)   ;
\draw    (153.5,86.5) -- (189.5,86.5) ;
\draw [shift={(191.5,86.5)}, rotate = 180] [color={rgb, 255:red, 0; green, 0; blue, 0 }  ][line width=0.75]    (10.93,-3.29) .. controls (6.95,-1.4) and (3.31,-0.3) .. (0,0) .. controls (3.31,0.3) and (6.95,1.4) .. (10.93,3.29)   ;
\draw [shift={(151.5,86.5)}, rotate = 0] [color={rgb, 255:red, 0; green, 0; blue, 0 }  ][line width=0.75]    (10.93,-3.29) .. controls (6.95,-1.4) and (3.31,-0.3) .. (0,0) .. controls (3.31,0.3) and (6.95,1.4) .. (10.93,3.29)   ;
\draw    (153.5,146.5) -- (189.5,146.5) ;
\draw [shift={(191.5,146.5)}, rotate = 180] [color={rgb, 255:red, 0; green, 0; blue, 0 }  ][line width=0.75]    (10.93,-3.29) .. controls (6.95,-1.4) and (3.31,-0.3) .. (0,0) .. controls (3.31,0.3) and (6.95,1.4) .. (10.93,3.29)   ;
\draw [shift={(151.5,146.5)}, rotate = 0] [color={rgb, 255:red, 0; green, 0; blue, 0 }  ][line width=0.75]    (10.93,-3.29) .. controls (6.95,-1.4) and (3.31,-0.3) .. (0,0) .. controls (3.31,0.3) and (6.95,1.4) .. (10.93,3.29)   ;
\draw    (61,43.73) -- (61,62.73) ;
\draw [shift={(61,64.73)}, rotate = 270] [color={rgb, 255:red, 0; green, 0; blue, 0 }  ][line width=0.75]    (10.93,-3.29) .. controls (6.95,-1.4) and (3.31,-0.3) .. (0,0) .. controls (3.31,0.3) and (6.95,1.4) .. (10.93,3.29)   ;
\draw    (282.5,44.23) -- (282.5,63.23) ;
\draw [shift={(282.5,65.23)}, rotate = 270] [color={rgb, 255:red, 0; green, 0; blue, 0 }  ][line width=0.75]    (10.93,-3.29) .. controls (6.95,-1.4) and (3.31,-0.3) .. (0,0) .. controls (3.31,0.3) and (6.95,1.4) .. (10.93,3.29)   ;
\draw    (111.5,106.23) -- (111.5,125.23) ;
\draw [shift={(111.5,127.23)}, rotate = 270] [color={rgb, 255:red, 0; green, 0; blue, 0 }  ][line width=0.75]    (10.93,-3.29) .. controls (6.95,-1.4) and (3.31,-0.3) .. (0,0) .. controls (3.31,0.3) and (6.95,1.4) .. (10.93,3.29)   ;
\draw    (230,104.23) -- (230,123.23) ;
\draw [shift={(230,125.23)}, rotate = 270] [color={rgb, 255:red, 0; green, 0; blue, 0 }  ][line width=0.75]    (10.93,-3.29) .. controls (6.95,-1.4) and (3.31,-0.3) .. (0,0) .. controls (3.31,0.3) and (6.95,1.4) .. (10.93,3.29)   ;

\draw (66.5,44.9) node [anchor=north west][inner sep=0.75pt]  [font=\small]  {$F_{1}$};
\draw (289.5,47.63) node [anchor=north west][inner sep=0.75pt]  [font=\small]  {$F^\tau_{1}$};
\draw (118.5,107.9) node [anchor=north west][inner sep=0.75pt]  [font=\small]  {$F_{2}$};
\draw (236,105.9) node [anchor=north west][inner sep=0.75pt]  [font=\small]  {$F^\tau_{2}$};
\draw (165,5.4) node [anchor=north west][inner sep=0.75pt]    {$\tau $};
\draw (165,66.9) node [anchor=north west][inner sep=0.75pt]    {$\tau $};
\draw (167,125.4) node [anchor=north west][inner sep=0.75pt]    {$\tau $};
\draw (10,17.4) node [anchor=north west][inner sep=0.75pt]    {$D$};
\draw (10,137.9) node [anchor=north west][inner sep=0.75pt]    {$D'$};

\end{tikzpicture}
        }
        \vspace{1em}
        \caption{$F$ and $F^\tau$.}
        \label{fig:IR-map}
    \end{figure}
    We define maps $F, G$ by the compositions
    \[
        F = F_2 F_1,\ 
        G = G_1 G_2
    \]
    where $F_1, F_2$ are the standard maps corresponding to the Reidemsiter moves performed on the left and the right part of the diagram respectively (see \Cref{fig:IR-map}), and $G_1, G_2$ are those homotopy inverses. 
    The conjugate maps are given by
    \[
        F^\tau = F^\tau_2 F^\tau_1,\ 
        G^\tau = G^\tau_1 G^\tau_2.
    \]
    Since $F_1$ and $F_2$ (resp.\ $G_1$ and $G_2$) are commutative, we may write
    \[
        F^\tau = F^\tau_1 F^\tau_2,\ 
        G^\tau = G^\tau_2 G^\tau_1.
    \]
    Since the tangle parts appearing in the move are $\Kh$-simple, using \Cref{lem:htpy-1,lem:htpy-2} we obtain homotopies with the desired properties such as $F_1 \htpy F^\tau_2$ and $F_2 \htpy F^\tau_1$. Thus we obtain the desired properties stated in \Cref{cl:htpy-1,cl:htpy-2}, such as $F \htpy F^\tau$. Here we are implicitly extending the maps defined locally for tangle diagrams to maps defined globally for the link diagrams, using the planar algebra structures explained in \cite[Section 5]{BarNatan:2004}.
\end{proof}

\subsubsection{R1, R2 and M1 -- M3}

\begin{proof}[Proof of \Cref{prop:invariance}, cases R1, R2 and M1 -- M3]
    Observe that each of these moves occurs locally on a disk that intersects the axis, and the tangles appearing in the move are $\Kh$-simple. Thus \Cref{cl:htpy-1,cl:htpy-2} immediately follows from \Cref{lem:htpy-1,lem:htpy-2}.
\end{proof}

\begin{remark}
    By using the explicit description of the maps, for some of the moves we may take the desired homotopies in a much simpler form. For R1 and R2, the homotopy equivalences $F, G$ and the homotopies $H, H'$ are strictly $\tau$-equivariant. For R3, may take $\tau F \htpy_{h_F} F \tau$ and $\tau G \htpy_{h_G} G \tau$ so that $G h_F = 0$, $F h_G = 0,$ and $\tau H + H \tau = h_G F$, $\tau H' + H' \tau = h_F G$.
\end{remark}

\subsection{Reduced version}
\label{subsec:reduced}

Next, we define a reduced version of the involutive Khovanov homology. Let us first recall the definition of the reduced complex in the non-involutive setting. For a pointed link diagram $D$, the \textit{reduced Khovanov complex}\footnote{
    The conventional notation for the reduced Khovanov complex is $\widetilde{\CKh}$. Here we changed the notation to avoid putting too much decorations on the letters. 
} $\rCKh(D)$ is defined as the subcomplex of $\CKh(D)$ generated by the standard generators each labeled $X$ on the pointed circle. The \textit{coreduced complex} $\rCKh'(D)$ is defined as the quotient $\CKh(D) / \rCKh(D)$. 

\begin{definition}
    A \textit{pointed involutive link} $(L, \tau)$ is an involutive link equipped with a basepoint on $\Fix(\tau) \cap L$. 
\end{definition}

Note that $2$-periodic links cannot be pointed, since $\Fix(\tau) \cap L = \emptyset$. For a pointed involutive link diagram $D$, it is obvious that the (co)reduced complexes are invariant under the involution $\tau$. Thus we may define,

\begin{definition}
    Let $(D, \tau)$ be a pointed involutive link diagram. The \textit{reduced involutive Khovanov complex} is defined as
    \[
        \rCKhI(D, \tau) = \Cone(\rCKh(D) \xrightarrow{Q(1 + \tau)} Q\rCKh(D)). 
    \]
    Similarly, the \textit{coreduced involutive Khovanov complex} is defined as 
    \[
        \rCKhI'(D, \tau) = \Cone(\rCKh'(D) \xrightarrow{Q(1 + \tau)} Q\rCKh'(D)). 
    \]
\end{definition}

Again, we usually omit $\tau$ from the notations of $\rCKhI$. In the non-involutive setting, the reduced and the coreduced complexes can be interchanged by the following automorphism $\sigma$. 

\begin{definition}
\label{def:sigma}
    The Frobenius algebra automorphism $\sigma$ on $A$ is defined by 
    \[
        1 \mapsto 1,\quad
        X \mapsto X + h.
    \]
    Its induced automorphism on $\CKh(D)$ is also denoted $\sigma$. 
\end{definition}

\begin{proposition}
    $\sigma$ is an involution that commutes with $\tau$. 
\end{proposition}

Thus $\sigma$ induces an involution on $\CKhI(D)$, which is again denoted $\sigma$. Analogous to the non-involutive case we have,

\begin{proposition}
\label{prop:sigma-swap}
    There are isomorphisms
    \begin{align*}
        \rCKhI'(D) 
        &\isom \sigma(\rCKhI(D)),\\
        \rCKhI(D) 
        &\isom \CKhI(D) / \sigma(\rCKhI(D)).
    \end{align*}
\end{proposition}

\begin{proof}
    Immediate from \cite[Propositions 3.12, 3.14]{Sano-Sato:2023}. 
\end{proof}

\begin{proposition}
\label{prop:CKhI-ses}
    There is a short exact sequence,
    \[
\begin{tikzcd}
\rCKhI(D) \arrow[r, hook] & \CKhI(D) \arrow[r, two heads] & \rCKhI'(D).
\end{tikzcd}
    \]
\end{proposition}

\begin{proof}
    The following short exact sequence is $\tau$-equivariant,
    \[
\begin{tikzcd}
\rCKh(D) \arrow[r, hook] & \CKh(D) \arrow[r, two heads] & \rCKh'(D).
\end{tikzcd}
    \]
    and hence we obtain maps between exact sequences,
    \[
\begin{tikzcd}[scale cd=0.8]
\rCKhI \arrow[r, hook] \arrow[d, equal]       & \CKhI \arrow[r, two heads] \arrow[d, equal]      & \CKhI/\rCKhI \arrow[d, dashed]                \\
{ \Cone(\rCKh \rightarrow Q\rCKh)} \arrow[r, hook] & { \Cone(\CKh \rightarrow Q\CKh)} \arrow[r, two heads] & { \Cone(\rCKh' \rightarrow Q\rCKh')}.
\end{tikzcd}
    \]
    The right dashed arrow is an isomorphism from the five lemma. 
\end{proof}

\begin{proposition}
\label{prop:CKhI-split}
    The short exact sequence of \Cref{prop:CKhI-ses} splits.
\end{proposition}

Before proceeding to the proof, let us review the corresponding results in the non-involutive setting. Shumakovitch \cite{Shumakovitch:2014} proved that the $\FF_2$-Khovanov homology splits (i.e.\ when $(R, h) = (\FF_2, 0)$), and lately Wigderson \cite{Wigderson:2016} extended this result to the $\FF_2$-bigraded Bar-Natan homology (i.e.\ $(R, h) = (\FF_2[H], H)$). Here we briefly review Wigderson's construction, which works whenever $\fchar{R} = 2$. First, as an $R$-module, the coreduced complex $\rCKh'(D)$ can be identified with the submodule of $\CKh(D)$ generated by the standard generators each labeled $1$ on the pointed circle. Then $\CKh(D) = \rCKh'(D) \oplus \rCKh(D)$ as $R$-modules, and the differential $d$ of $\CKh(D)$ can be described as
\[
    d = \begin{pmatrix}
        d_1 & \\
        f & d_X
    \end{pmatrix},
\]
where $d_X, d_1$ are differentials of $\rCKh(D), \rCKh'(D)$ respectively, and
\[
    f\colon\rCKh'(D) \rightarrow \rCKh(D)
\]
is the map given by restricting $d$ to $\rCKh'(D)$ and then projecting onto $\rCKh(D)$. It follows from $d^2 = 0$ that $f$ is a chain map, and thus $\CKh(D)$ may be regarded as the cone of $f$. Next, a null homotopy $\kappa$ of $f$ is constructed as follows. For each $i \geq 0$, define
\[
    \kappa_i: \rCKh'(D) \rightarrow \rCKh(D), \quad
    x = \underline{1} \otimes \cdots \mapsto \sum_{\text{\textcircled{0}}^X, \ldots, \text{\textcircled{$i$}}^X} \underline{X} \otimes \cdots.
\]
Here, the underline indicates the label for the pointed circle, and the sum runs over all choices of $i + 1$ circles $C_j$ labeled $X$ in $x$. Inside the summation, the label on the pointed circle is changed from $1$ to $X$ while the labels of $C_j$ are changed from $X$ to $1$. Then define
\[
    \kappa = \sum_{i \geq 0} h^i \kappa_i\colon\rCKh'(D) \rightarrow \rCKh(D).
\]
It is proved in \cite{Wigderson:2016} that $\kappa$ gives a null homotopy of $f$, and hence $1 + \kappa$ gives a section of the quotient map $\CKh(D) \rightarrow \rCKh'(D)$.

\begin{proof}[Proof of \Cref{prop:CKhI-split}]
    Suppose $D$ is a pointed involutive link diagram. It is obvious from the construction of $\kappa$ that it commutes with $\tau$, and hence 
    \[
        \begin{pmatrix}
            1 + \kappa & \\
            & 1 + \kappa
        \end{pmatrix}
    \]
    gives a section of the quotient map $\CKhI(D) \rightarrow \rCKhI'(D)$.
\end{proof}

Finally, we prove the invariance of the (co)reduced involutive homologies. In order to prove an analogue of \Cref{prop:invariance}, we need extra consideration for the moves that involve the basepoint, which are R1 and M1 with the basepoint placed on the horizontal strand. In fact, one can check that the corresponding maps do not restrict to the reduced complexes
\[
\begin{tikzcd}
\rCKh(D) \arrow[r, hook] \arrow[d, "\times", dotted] & \CKh(D) \arrow[d, "\rho"] \\
\rCKh(D') \arrow[r, hook]                                        & \CKh(D).                  
\end{tikzcd}
\]
Thus we restrict the diagrams and the moves that are allowed for the reduced complexes. 

\begin{definition}
\label{def:normal-diagram}
    A diagram of a pointed involutive link diagram is \textit{normal} if the basepoint of the link is placed at the bottommost on the axis, and the horizontal strand containing the basepoint is directed rightwards. 
\end{definition}

Obviously, any pointed involutive link possesses a normal diagram. Moreover, 

\begin{proposition}
\label{prop:normal-inv-move}
    Suppose $D, D'$ are normal pointed involutive link diagrams that represent the same pointed involutive link. Then there is a sequence of involutive Reidemeister moves whose intermediate diagrams are also normal.  
\end{proposition}

\begin{proof}
    Take any sequence of involutive Reidemeister moves between $D$ and $D'$.
    \[
        D = D_0 \rightarrow D_1 \rightarrow \cdots \rightarrow D_{N} = D'. 
    \]
    We modify this sequence, by increasing the number of moves if necessary, so that the intermediate diagrams are all normal. 
    
    \textbf{Step 1.}
    For each move $D_i \rightarrow D_{i+1}$, we may transform the diagrams by pulling down the horizontal strands that contain the basepoints so that they are placed at the bottommost on the axes (which can be realized by sequences of involutive Reidemeister moves). Moreover, we can show that the two transformed diagrams $D'_i, D'_{i+1}$ can be related by a sequence of involutive Reidemeister moves that fix the basepoints. 
        \[
\begin{tikzcd}
D_i \arrow[rr, "\text{move}"] \arrow[d, "\text{pull}"'] &                                 & D_{i+1} \arrow[d, "\text{pull}"] \\
D'_i \arrow[r, "\text{move}"]                                & \cdots \arrow[r, "\text{move}"] & D'_{i+1}.                    \end{tikzcd}
    \]
    The claim is clear when the original move occurs above the pointed strand, or when the moves are off-axis which are IR1 -- IR3. We must consider the case where the move occurs below the pointed strand, or contains the pointed strand itself. \Cref{fig:modif-R1} depicts the modification for the R1 move that contains the pointed strand. The claim for the remaining moves R2 and M1 -- M3 can be checked similarly. 

    \begin{figure}[t]
        \centering
        \resizebox{.7\textwidth}{!}{
            \tikzset{every picture/.style={line width=0.75pt}} 

\begin{tikzpicture}[x=0.75pt,y=0.75pt,yscale=-1,xscale=1]

\draw    (50.24,16.5) .. controls (4.71,15.83) and (66.26,40.66) .. (95.09,40.83) ;
\draw    (317.77,38.7) .. controls (328.47,39.36) and (347.77,26.86) .. (361.96,26.86) .. controls (376.15,26.86) and (397.29,38.36) .. (407.29,39.03) ;
\draw    (6.43,40.16) .. controls (35.09,39.5) and (95.76,17.16) .. (50.24,16.5) ;
\draw    (50.6,134.62) .. controls (41.74,135.02) and (41.1,131.68) .. (41.26,125.38) .. controls (41.41,119.08) and (42.54,91.52) .. (38.92,93.13) .. controls (35.29,94.74) and (34.11,96.23) .. (36.9,98.37) .. controls (39.7,100.5) and (77.18,115.81) .. (95.4,115.92) ;
\draw    (6.12,115.25) .. controls (29.42,114.7) and (75.17,97.96) .. (65.82,93.92) .. controls (56.47,89.88) and (63.16,119.32) .. (62.68,126.16) .. controls (62.2,133) and (60.27,135.82) .. (50.6,134.62) ;
\draw    (161.06,127.05) .. controls (157,127) and (155.22,125.25) .. (155.32,121.37) .. controls (155.41,117.5) and (154.2,107.8) .. (155.86,105.76) .. controls (157.52,103.72) and (177.4,115.49) .. (188.6,115.56) ;
\draw    (133.72,115.15) .. controls (148.04,114.81) and (164.77,103.53) .. (167.33,106.5) .. controls (169.88,109.47) and (168.64,118.88) .. (168.49,121.85) .. controls (168.34,124.83) and (167.8,127.4) .. (161.06,127.05) ;
\draw    (252.62,126.89) .. controls (247.6,127.12) and (247.24,125.23) .. (247.33,121.66) .. controls (247.41,118.09) and (256.27,108.87) .. (258.33,107.73) .. controls (260.38,106.59) and (267.68,116.24) .. (278,116.3) ;
\draw    (227.43,115.92) .. controls (240.63,115.61) and (246.62,105.45) .. (248.97,108.19) .. controls (251.32,110.92) and (259.6,119.36) .. (259.47,122.1) .. controls (259.33,124.84) and (258.1,127.58) .. (252.62,126.89) ;
\draw    (316.97,114.5) .. controls (329.4,114) and (347,97.2) .. (352.6,102) .. controls (358.2,106.8) and (352.2,132) .. (357,134) .. controls (361.8,136) and (363.4,136) .. (369.4,134.4) .. controls (375.4,132.8) and (369.4,106.4) .. (373.8,101.6) .. controls (378.2,96.8) and (395.8,114) .. (406.49,114.83) ;
\draw [color={rgb, 255:red, 128; green, 128; blue, 128 }  ,draw opacity=1 ]   (50.4,49) -- (50.59,83.4) ;
\draw [shift={(50.6,85.4)}, rotate = 269.69] [color={rgb, 255:red, 128; green, 128; blue, 128 }  ,draw opacity=1 ][line width=0.75]    (10.93,-3.29) .. controls (6.95,-1.4) and (3.31,-0.3) .. (0,0) .. controls (3.31,0.3) and (6.95,1.4) .. (10.93,3.29)   ;
\draw [color={rgb, 255:red, 128; green, 128; blue, 128 }  ,draw opacity=1 ]   (106,115.2) -- (117.8,115.2) ;
\draw [shift={(119.8,115.2)}, rotate = 180] [color={rgb, 255:red, 128; green, 128; blue, 128 }  ,draw opacity=1 ][line width=0.75]    (10.93,-3.29) .. controls (6.95,-1.4) and (3.31,-0.3) .. (0,0) .. controls (3.31,0.3) and (6.95,1.4) .. (10.93,3.29)   ;
\draw [color={rgb, 255:red, 0; green, 0; blue, 0 }  ,draw opacity=1 ]   (139.4,41) -- (273.6,41) ;
\draw [shift={(275.6,41)}, rotate = 180] [color={rgb, 255:red, 0; green, 0; blue, 0 }  ,draw opacity=1 ][line width=0.75]    (10.93,-3.29) .. controls (6.95,-1.4) and (3.31,-0.3) .. (0,0) .. controls (3.31,0.3) and (6.95,1.4) .. (10.93,3.29)   ;
\draw [color={rgb, 255:red, 128; green, 128; blue, 128 }  ,draw opacity=1 ]   (200.6,114.8) -- (212.4,114.8) ;
\draw [shift={(214.4,114.8)}, rotate = 180] [color={rgb, 255:red, 128; green, 128; blue, 128 }  ,draw opacity=1 ][line width=0.75]    (10.93,-3.29) .. controls (6.95,-1.4) and (3.31,-0.3) .. (0,0) .. controls (3.31,0.3) and (6.95,1.4) .. (10.93,3.29)   ;
\draw [color={rgb, 255:red, 128; green, 128; blue, 128 }  ,draw opacity=1 ]   (288.4,114.8) -- (300.2,114.8) ;
\draw [shift={(302.2,114.8)}, rotate = 180] [color={rgb, 255:red, 128; green, 128; blue, 128 }  ,draw opacity=1 ][line width=0.75]    (10.93,-3.29) .. controls (6.95,-1.4) and (3.31,-0.3) .. (0,0) .. controls (3.31,0.3) and (6.95,1.4) .. (10.93,3.29)   ;
\draw [color={rgb, 255:red, 128; green, 128; blue, 128 }  ,draw opacity=1 ]   (364,49) -- (364.19,83.4) ;
\draw [shift={(364.2,85.4)}, rotate = 269.69] [color={rgb, 255:red, 128; green, 128; blue, 128 }  ,draw opacity=1 ][line width=0.75]    (10.93,-3.29) .. controls (6.95,-1.4) and (3.31,-0.3) .. (0,0) .. controls (3.31,0.3) and (6.95,1.4) .. (10.93,3.29)   ;

\draw (192.6,21.4) node [anchor=north west][inner sep=0.75pt]   [align=left] {R1};

\end{tikzpicture}
        }
        \caption{Modifying the R1-move}
        \label{fig:modif-R1}
    \end{figure}
    
    \textbf{Step 2.} The remaining work is to undo the changes in the direction of the pointed strand due to the R1 moves. Note that applying an R1 move to the bottom strand is equivalent to half-twisting the upper parts and then applying an overall half-rotation with respect to the axis. Since the diagram is involutive, a half-rotation has the effect of only changing the orientations on the components of the link. We modify each R1 move by only twisting the upper parts (which can be represented by a sequence of involutive moves) while keeping the bottom strand fixed. Since the bottom strands of $D$ and $D'$ are both pointed rightwards, the R1 moves in total must be applied an even number of times. Thus the effects of skipping the overall half-rotations cancel, and we obtain a desired sequence of involutive Reidemeister moves between $D$ and $D'$. 
\end{proof}

\begin{proposition}
\label{prop:invariance-reduced}
    Let $D, D'$ be normal diagrams related by an involutive Reidemeister move. The chain homotopy equivalence $\rho$ and the corresponding chain homotopies given in \Cref{prop:invariance} restrict to 
    \[
        \rho\colon\rCKhI(D) \rightarrow \rCKhI(D')
    \]
    and the following diagram commutes.  
    \[
    \begin{tikzcd}
        {Q\rCKh(D)[1]} \arrow[r, hook] \arrow[d, "\rho"] & \rCKhI(D) \arrow[r, two heads] \arrow[d, "\rho"] & \rCKh(D) \arrow[d, "\rho"] \\
        {Q\rCKh(D')[1]} \arrow[r, hook]                   & \rCKhI(D') \arrow[r, two heads]                            & \rCKh(D').
    \end{tikzcd} 
    \]
    The same statement holds for the coreduced counterparts. 
\end{proposition}

\begin{proof}
    Since the basepoints are fixed by the move, the chain maps and chain homotopies of \Cref{prop:invariance} restrict to the (co)reduced complexes. 
\end{proof}

We conclude that for a pointed involutive link $L$ with normal diagram $D$, the chain homotopy equivalence classes of the (co)reduced complexes $\rCKhI(D)$, $\rCKhI'(D)$ are invariants of $L$. Those homologies are denoted $\rKhI(L)$ and $\rKhI'(L)$ and called the \textit{(co)reduced involutive Khovanov homologies} of $L$. 

\subsection{Mirrors}

Next, we study the behavior of the involutive complexes under mirrors. The arguments are straightforward extensions of \cite[Section 3.5.2]{Sano-Sato:2023} to the involutive setting. 
Consider the standard perfect pairing on $A$
\[
    \langle \cdot, \cdot \rangle\colon A \otimes A \rightarrow R
\]
given by $\langle x, y \rangle = \epsilon(xy)$. The associated duality isomorphism $\dual\colon A \rightarrow A^*$ such that $\langle x, y \rangle = \dual(x)(y)$ is given by 
\[
    \dual(1) = X^*,\ 
    \dual(X) = 1^* + hX^* 
\]
where $\{1^*, X^*\}$ is the dual basis for $A^*$ to the basis $\{1, X\}$ for $A$. For convenience, put $Y = X + h$, and consider another basis $\{1, Y\}$ for $A$ and its dual basis $\{1^\dag, Y^\dag\}$ for $A^*$. Then we have 
\[
    \dual(1) = Y^\dag,\ 
    \dual(X) = 1^\dag
\]
and 
\[
    \langle X, X \rangle = h,\ 
    \langle X, Y \rangle = \langle Y, X \rangle = 0,\ 
    \langle Y, Y \rangle = h.
\]
Note that when $h = 0$, the duality isomorphism $\dual$ coincides with the ordinary self-dual isomorphism. 

For a link diagram $D$, the duality isomorphism $\dual$ induces a chain isomorphism 
\[
    \dual\colon\CKh(D^*) \xrightarrow{\isom} \CKh(D)^*
\]
where $D^*$ denotes the mirror of $D$, and $\CKh(D)^*$ the dual complex of $\CKh(D)$ with bigrading $(\CKh(D)^*)^{i, j} = \CKh^{-i, -j}(D)^*$. 
This gives a perfect pairing
\[
    \langle \cdot, \cdot \rangle\colon\CKh(D) \otimes \CKh(D^*) \rightarrow R.
\]

Now suppose $(D, \tau)$ is an involutive link diagram. 

\begin{lemma}
    \[
        \dual \tau = \tau^* \dual\colon\CKh(D^*) \rightarrow \CKh(D)^*.
    \]
\end{lemma}

\begin{proof}
    By considering an $1X$-labeled generator $x$ for $D^*$ and an $1Y$-labeled generator $y$ for $D$, we easily see that $\dual(\tau x)(y) = \dual(x)(\tau y)$ holds. 
\end{proof}

\begin{lemma}
    \[
        \CKhI(D)^*[1] \isom \Cone(\CKh(D)^* \xrightarrow{1 + \tau^*} \CKh(D)^*)
    \]
\end{lemma}

\begin{proof}
    Obvious. 
\end{proof}

\begin{proposition}
\label{prop:CKhI-dual}
    There is an isomorphism
    \[
        \dual\colon\CKhI(D^*) \xrightarrow{\isom} \CKhI(D)^*[1].
    \]
\end{proposition}

\begin{proof}
    The above results give isomorphisms,
    \begin{align*}
        \CKhI(D^*) 
            &= \Cone(\CKh(D^*) \xrightarrow{1 + \tau} \CKh(D^*)) \\
            &\isom \Cone(\CKh(D)^* \xrightarrow{1 + \tau^*} \CKh(D)^*) \\
            &\isom \CKhI(D)^*[1].
    \end{align*}    
\end{proof}

Next we see that the duality isomorphism also respects the (co)reduced complexes. Recall from \Cref{prop:sigma-swap} that the coreduced complex $\rCKhI'(D)$ may be identified with the subcomplex $\sigma(\rCKhI(D))$, and the reduced complex $\rCKhI(D)$ with the quotient $\CKhI(D) / \sigma(\rCKhI(D))$. Thus there is a short exact sequence in the reversed direction,
    \[
\begin{tikzcd}
\rCKhI'(D) \arrow[r, "i'"] & \CKhI(D) \arrow[r, "p'"] & \rCKhI(D)
\end{tikzcd}
    \]

\begin{proposition}
    The isomorphism $\dual$ of \Cref{prop:CKhI-dual} induces isomorphisms on the (co)reduced complexes, such that the following diagram commutes
    \[
\begin{tikzcd}
\rCKhI(D^*) \arrow[r, "i"] \arrow[d, "\dual", dashed]       & \CKhI(D^*) \arrow[r, "p"] \arrow[d, "\dual"]      & \rCKhI'(D^*) \arrow[d, "\dual", dashed]                \\
\rCKhI(D)^* \arrow[r, "(p')^*"] & \CKhI(D)^* \arrow[r, "(i')^*"] & \rCKhI'(D)^*.
\end{tikzcd}
    \]
\end{proposition}

\begin{proof}
    Combine \Cref{prop:CKhI-dual} with \cite[Proposition 3.36]{Sano-Sato:2023}. 
\end{proof}

\begin{proposition}
\label{prop:perf-pairing}
    There are perfect pairings
    \[
        \langle \cdot, \cdot \rangle\colon\CKhI(D) \otimes \CKhI(D^*) \rightarrow R
    \]
    and 
    \[
        \langle \cdot, \cdot \rangle_r\colon\rCKhI(D) \otimes \rCKhI(D^*) \rightarrow R
    \]
    such that the following diagram commutes:
    \[
    \begin{tikzcd}
    \rCKhI(D) \otimes \rCKhI(D^*) \arrow[r, "{\langle \cdot, \cdot \rangle_r}"] \arrow[d, "i \otimes i"] & R \arrow[d, "h \cdot "] \\ 
    \CKhI(D) \otimes \CKhI(D^*) \arrow[r, "{\langle \cdot, \cdot \rangle}"] & R.         
    \end{tikzcd}
    \]
\end{proposition}

\begin{proof}
    Combine \Cref{prop:CKhI-dual} with \cite[Propositions 3.33, 3.37]{Sano-Sato:2023}. 
\end{proof}
    \section{Equivariant Rasmussen Invariants}
\label{sec:3}


The focus of this section is strongly invertible knots and links. We mainly consider the case where $h \neq 0$, typically
\[
    (R, h) = (\FF_2, 1),\ (\FF_2[H], H),\ (\FF_2[H^\pm], H).
\]

\subsection{Equivariant Lee classes}
\label{subsec:equiv-lee-class}

For an oriented link diagram $D$, let $O(D)$ denote the set of all orientations on the underlying unoriented diagram of $D$, which has $2^l$ elements when the corresponding link has $l$ components. Recall that in the non-involutive setting, if $h \in R$ is invertible, then the homology $\Kh(D)$ is generated by the \textit{Lee classes} $\ca(D, o)$ of $D$ with $o \in O(D)$ \cite{Lee:2005,Turner:2020}. Here we recall the construction.

\begin{figure}[t]
    \centering
    \resizebox{.9\linewidth}{!}{
    \tikzset{every picture/.style={line width=0.75pt}} 

\begin{tikzpicture}[x=0.75pt,y=0.75pt,yscale=-1,xscale=1]

\clip (0,0) rectangle (420,120);

\draw [line width=1.5]    (70.58,41.39) .. controls (63.5,-18) and (4,25.5) .. (42.36,77.01) ;
\draw [line width=1.5]    (53.15,86.58) .. controls (94,115) and (123,32) .. (35.5,45.5) ;
\draw [line width=1.5]    (25.13,46.95) .. controls (-12.3,54.35) and (15.68,141.29) .. (69.06,54.44) ;
\draw   (114.5,43.69) -- (125.5,43.69) -- (125.5,37.79) -- (137.72,49.59) -- (125.5,61.39) -- (125.5,55.49) -- (114.5,55.49) -- cycle ;
\draw  [color={rgb, 255:red, 0; green, 0; blue, 0 }  ,draw opacity=1 ][fill={rgb, 255:red, 155; green, 155; blue, 155 }  ,fill opacity=1 ] (197,11) .. controls (210,10) and (215,18) .. (217,35) .. controls (219,52) and (226.77,42.04) .. (237,53) .. controls (247.23,63.96) and (246,72) .. (244,81) .. controls (242,90) and (230,98) .. (213,89) .. controls (196,80) and (203,81) .. (188,90) .. controls (173,99) and (164,89) .. (158,82) .. controls (152,75) and (153,60) .. (160,55) .. controls (167,50) and (178,50) .. (180,35) .. controls (182,20) and (184,12) .. (197,11) -- cycle ;
\draw  [fill={rgb, 255:red, 255; green, 255; blue, 255 }  ,fill opacity=1 ] (196.5,42) .. controls (208.5,42) and (222,39) .. (215.5,61) .. controls (209,83) and (192,85) .. (182,65) .. controls (172,45) and (184.5,42) .. (196.5,42) -- cycle ;
\draw  [color={rgb, 255:red, 208; green, 2; blue, 27 }  ,draw opacity=1 ][line width=1.5]  (346.5,12) .. controls (359.5,11) and (364.5,19) .. (366.5,36) .. controls (368.5,53) and (376.27,43.04) .. (386.5,54) .. controls (396.73,64.96) and (395.5,73) .. (393.5,82) .. controls (391.5,91) and (379.5,99) .. (362.5,90) .. controls (345.5,81) and (352.5,82) .. (337.5,91) .. controls (322.5,100) and (313.5,90) .. (307.5,83) .. controls (301.5,76) and (302.5,61) .. (309.5,56) .. controls (316.5,51) and (327.5,51) .. (329.5,36) .. controls (331.5,21) and (333.5,13) .. (346.5,12) -- cycle ;
\draw  [color={rgb, 255:red, 0; green, 116; blue, 255 }  ,draw opacity=1 ][line width=1.5]  (346.5,43) .. controls (358.5,43) and (372,40) .. (365.5,62) .. controls (359,84) and (342,86) .. (332,66) .. controls (322,46) and (334.5,43) .. (346.5,43) -- cycle ;
\draw    (200.5,11.5) -- (195.5,11.5) ;
\draw [shift={(192.5,11.5)}, rotate = 360] [fill={rgb, 255:red, 0; green, 0; blue, 0 }  ][line width=0.08]  [draw opacity=0] (8.93,-4.29) -- (0,0) -- (8.93,4.29) -- cycle    ;
\draw    (54.5,13) -- (48.5,13) ;
\draw [shift={(45.5,13)}, rotate = 360] [fill={rgb, 255:red, 0; green, 0; blue, 0 }  ][line width=0.08]  [draw opacity=0] (8.93,-4.29) -- (0,0) -- (8.93,4.29) -- cycle    ;
\draw    (203.5,42) -- (197.5,42.33) ;
\draw [shift={(194.5,42.5)}, rotate = 356.82] [fill={rgb, 255:red, 0; green, 0; blue, 0 }  ][line width=0.08]  [draw opacity=0] (8.93,-4.29) -- (0,0) -- (8.93,4.29) -- cycle    ;
\draw   (267,43.69) -- (278,43.69) -- (278,37.79) -- (290.22,49.59) -- (278,61.39) -- (278,55.49) -- (267,55.49) -- cycle ;

\draw (369,4.4) node [anchor=north west][inner sep=0.75pt]    {$\textcolor[rgb]{0.82,0.01,0.11}{a}$};
\draw (367.5,57.4) node [anchor=north west][inner sep=0.75pt]  [color={rgb, 255:red, 0; green, 117; blue, 255 }  ,opacity=1 ]  {$b$};

\end{tikzpicture}
    }
    \caption{$ab$-coloring on the Seifert circles of $K$.}
    \label{fig:equiv-lee}
\end{figure}

\begin{algorithm}
\label{algo:ab-color}
    Given a link diagram $D$, the \textit{$ab$-coloring} on its Seifert circles is defined as follows: separate $\mathbb{R}^2$ into regions by the Seifert circles of $D$, and color the regions in the checkerboard fashion, with the unbounded region colored white. For each Seifert circle, let it inherit the orientation from $D$, and assign to it $a$ if it sees a black region to the left with respect to the orientation, or $b$ otherwise  (see \Cref{fig:equiv-lee}). 
\end{algorithm}

\begin{definition}
\label{def:lee-cycle}
    Let $D$ be a link diagram. There is a unique state $s$ of $D$ where the resolved diagram $D(s)$ gives the Seifert circles of $D$. With the $ab$-coloring on the Seifert circles, define an element $\ca(D)$ in $\CKh(D)$ for that state by labeling each circle by $X$ if it is colored $a$, and $Y = X + h$ if it is colored $b$. Similarly, for any orientation $o \in O(D)$, we define an element $\ca(D, o)$ by the same procedure after reorienting $D$ by $o$. These elements $\ca(D, o)$ are in fact cycles in $\CKh(D)$, and are called the \textit{Lee cycles} of $D$. The homology classes are called the \textit{Lee classes} of $D$. 
\end{definition}

That $\ca(D, o)$ is a cycle in $\CKh(D)$ can be seen from the fact that each crossing of $D$ connects two differently colored strands in the resolved diagram, and merging the strands results in $XY = 0$. Note that the automorphism $\sigma$ on $\CKh(D)$ interchanges $\ca(D, o)$ and $\ca(D, -o)$, where $-o$ is the reversed orientation of $o$. The homological gradings of $\ca(D, o)$ are all even. 

Now, let us consider the case when $(D, \tau)$ is an involutive link diagram. For each $o \in O(D)$, let $\tau(o) \in O(D)$ denote the orientation obtained by applying $\tau$ to the involutive link diagram $D$ reoriented with $o$.
\begin{lemma}
\label{lem:ca-tau-comm}
    $\ca(D, \tau(o)) = \sigma \tau (\ca(D, o))$. 
\end{lemma}

\begin{proof}
    For each crossing $c$ of $D$, let $c'$ denote the crossing corresponding to $c$ under $\tau$ (possibly $c$ = $c'$). Observe that the orientation-preserving resolution of $c$ for $o$ is equal to that of $c'$ for $\tau(o)$. Thus, from the definition of the algebraic involution $\tau$ on $\CKh(D)$, it associates the orientation-preserving state $s$ of $(D, o)$ to the orientation-preserving state $s'$ of $(D, \tau(o))$. The Seifert circles of $(D, \tau(o))$ coincide with the images of the Seifert circles of $(D, o)$ under the reflection with respect to the axis, but with the orientations reversed. 
\end{proof}

\begin{remark}
    Here, it would be more natural to redefine the algebraic involution $\tau$ by $\sigma \tau$, then \Cref{lem:ca-tau-comm} would be written more cleanly as $\ca(D, \tau(o)) = \tau(\ca(D, o))$. See also \Cref{sec:6}. 
\end{remark}

Define the subset of $O(D)$,
\[
    O^\tau(D) := \{\ o \in O(D) \mid \tau(o) = -o \ \}
\]
and the quotient set 
\[
    \overline{O}{}^\tau(D) := (O(D) \setminus O^\tau(D)) / (\tau(o) \sim -o).
\]
Here, each equivalence class $[o] \in \overline{O}{}^\tau(D)$ is the two-element set $\{o, -\tau(o)\}$, and
\[
    (1 + \tau) \ca(D, o) = \ca(D, o) + \ca(D, -\tau(o))
\]
is well-defined, independent of the representative $o$. The following proposition is immediate from \Cref{lem:ca-tau-comm}. 

\begin{proposition}%
\label{prop:lee-tau-invariant}
    $\ca(D, o)$ is $\tau$-invariant for any $o \in O^\tau(D)$, and $(1 + \tau)\ca(D, o')$ is (non-zero and) $\tau$-invariant for any $[o'] \in \overline{O}{}^\tau(D)$. 
\end{proposition}

In the involutive complex $\CKhI(D)$, for each orientation $o \in O(D)$, the two elements $\ca(D, o)$, $Q\ca(D, o)$ in $\CKhI(D)$ will be denoted hereafter by $\dca(D, o)$ and $\uca(D, o)$ respectively. Note that $\dca(D, o)$ has even homological grading whereas $\uca(D, o)$ has odd. The following proposition is immediate from \Cref{prop:lee-tau-invariant}.

\begin{proposition}
\label{prop:equiv-lee-cycles}
    $\dca(D, o)$ is a cycle for any $o \in O^\tau(D)$, and $(1 + \tau)\dca(D, o')$ is a (non-zero) cycle for any $[o'] \in \overline{O}{}^\tau(D)$. $\uca(D, o)$ is a cycle for any $o \in O(D)$.
\end{proposition}

\begin{definition}
    The cycles of \Cref{prop:equiv-lee-cycles} are called the \textit{equivariant Lee cycles} of $D$, and those homology classes the \textit{equivariant Lee classes} of $D$. 
\end{definition}

It is proved in \cite[Theorem 4.2]{Lee:2005} that the (non-involutive) \textit{$\QQ$-Lee homology} of $D$ is freely generated by the \textit{Lee classes} of $D$. More generally, the same statement holds whenever $h \in R$ is invertible \cite[Theorem 4.2]{Turner:2020}, \cite[Proposition 2.9]{Sano:2020}. The following proposition is an analogue of this statement in the involutive setting.

\begin{proposition}
\label{prop:khi-structure}
    If $h \in R$ is invertible, the involutive Khovanov homology $\KhI(D)$ is generated by the equivariant Lee classes. More strongly, a basis is given by the collection of the following classes: (i) $\dca(D, o)$ for $o \in O^\tau(D)$, (ii) $(1 + \tau)\dca(D, o')$ for $[o'] \in \overline{O}{}^\tau(D)$, (iii) $\uca(D, o)$ for $o \in O^\tau(D)$, and (iv) $\uca(D, o')$ for $[o'] \in \overline{O}{}^\tau(D)$ (one representative $o'$ chosen for each equivalence class $[o']$). 
\end{proposition}

\begin{proof}
\begin{figure}[t]
\centering
\includegraphics[width=0.5\linewidth]{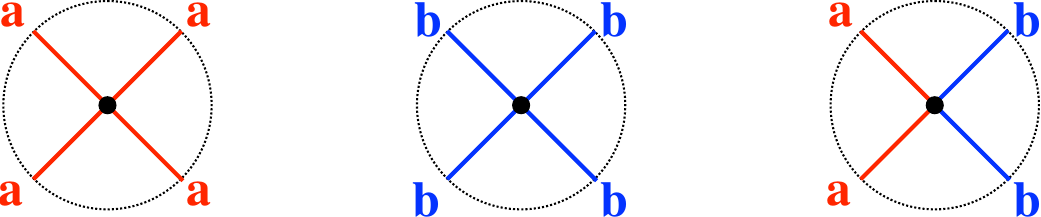}
\caption{Local picture of an admissibly colored diagram.}
\label{fig:adm-coloring}
\end{figure}
    The proof is similar to the proof for the non-involutive case given in \cite{Wehrli:2008}, also explained in detail in \cite{Lewark:2009}. 
    To explain briefly, first note that when $h$ is invertible, we may take $\{X, Y\}$ as a basis for the Frobenius algebra $A$. An \textit{admissible coloring} of $D$ is a coloring with $a$ or $b$ on the edges of $D$ such that each crossing admits a resolution that determines the colors of the two arc segments accordingly (see \Cref{fig:adm-coloring}). $\CKh(D)$ can be decomposed into subcomplexes, each corresponding to an admissible coloring of $D$, generated by the $XY$-labeled generators that match the coloring. If there is a crossing such that the four incident edges are colored the same (as in the first two pictures of \Cref{fig:adm-coloring}), then the generators can be canceled in pairs, resulting in a trivial complex. By contracting all such subcomplexes, we will be left with subcomplexes each corresponding to an admissible coloring such that the four incident edges at each crossing have two different colors (as in the third picture of \Cref{fig:adm-coloring}), which in turn corresponds one-to-one to an orientation $o$ of $D$. Such a subcomplex is generated by the single Lee cycle $\alpha(D, o)$, and thus $\CKh(D)$ is chain homotopy equivalent to a complex generated by the Lee cycles with trivial differential. 
    
    We may perform the same cancellation in both $\CKh(D)$ and $Q\CKh(D)$. The resulting complex is freely generated by the elements $\{\dca(D, o)\} \cup \{\uca(D, o)\}$. For $o \in O^\tau(D)$, the cycles $\dca(D, o)$ and $\uca(D, o)$ are both isolated. For $o' \not\in O^\tau(D)$, there are non-zero arrows of the form
    \[
\begin{tikzcd}
{\dca(D, o')} \arrow[d, "1"'] \arrow[rd] & {\tau\dca(D, o')} \arrow[ld, "\tau" description] \arrow[d, "1"] \\
{\uca(D, o')} & {\tau\uca(D, o')}
\end{tikzcd}
    \]
    By a basis change, this transforms into
    \[
\begin{tikzcd}
{(1 + \tau)\dca(D, o')} \arrow[d, "0"', dotted] \arrow[rd, dotted] & {\dca(D, o')} \arrow[ld, dotted] \arrow[d, "1 + \tau"] \\
{\uca(D, o')} & {(1 + \tau)\uca(D, o')}.
\end{tikzcd}
    \]
    For each class $[o'] \in \overline{O}{}^\tau$, choose a representative $o'$ and cancel the right vertical arrow shown above. Then we obtain a complex with trivial differential, freely generated by the cycles described in the statement. 
\end{proof}

\begin{remark}
    If $D$ has $l$ symmetric components and $2m$ asymmetric components, we have $|O(D)| = 2^{l + 2m}$, $|O^\tau(D)| = 2^{l + m}$ and $|\overline{O}{}^\tau(D)| = 2^{l + m - 1}(2^m - 1)$. Therefore, $\rank \KhI(D) = 2^{l + m}(2^m + 1) \leq 2 \rank \Kh(D) = 2^{l + 2m + 1}$.
\end{remark}

\begin{proposition}
\label{prop:ca-SES}
    The short exact sequence
    \[
    \begin{tikzcd}
        0 \arrow[r] & {\CKh(D)[1]} \arrow[r, "Q"] & \CKhI(D) \arrow[r, "q"] & \CKh(D) \arrow[r] & 0
    \end{tikzcd} 
    \]
    induces a long exact sequence
    \[
    \begin{tikzcd}
        \cdots \arrow[r] & {\Kh^{i - 1}(D)} \arrow[r, "Q_*"] & \KhI^i(D) \arrow[r, "q_*"] & \Kh^i(D) \arrow[r, "\ (1 + \tau)_*\ "] & {\Kh^i(D)} \arrow[r] & \cdots
    \end{tikzcd} 
    \]
    such that if $h$ is invertible, it decomposes into segments of isomorphisms mapping 
    \[
\begin{tikzcd}[row sep=".5em"]
{\ca(D, o)} \arrow[r, "Q_*", maps to] & {\uca(D, o)} & & \\
{\ca(D, o')} \arrow[r, maps to] & {\uca(D, o')} & & \\
 & {\dca(D, o)} \arrow[r, "q_*", maps to] & {\ca(D, o)} & \\
 & {(1 + \tau)\dca(D, o')} \arrow[r, maps to] & {(1 + \tau)\ca(D, o')} & \\
 & & {\ca(D, o')} \arrow[r, maps to] & {(1 + \tau)\ca(D, o')}
\end{tikzcd}
    \]
    where $o \in O^\tau$ and $[o'] \in \overline{O}{}^\tau$. In particular, when $D$ has no asymmetric components, the long exact sequence splits as
    \[
    \begin{tikzcd}
        0 \arrow[r] & {\Kh(D)[1]} \arrow[r] & \KhI(D) \arrow[r] & \Kh(D) \arrow[r] & 0
    \end{tikzcd} 
    \]
    mapping 
    \[
\begin{tikzcd}[row sep=".5em"]
{\ca(D, o)} \arrow[r, "Q_*", maps to] & {\uca(D, o)} & \\
 & {\dca(D, o)} \arrow[r, "q_*", maps to] & {\ca(D, o).} \\
\end{tikzcd}
    \]
\end{proposition}

Hereafter we assume that $(D, \tau)$ is a strongly invertible link diagram. Then the given orientation $o_D$ of $D$ and its reversal $-o_D$ both belong to $O^\tau(D)$. We write $\ca(D) := \ca(D, o_D)$ and $\cb(D) := \ca(D, -o_D)$ as cycles in $\CKh^0(D)$, and correspondingly $\dca(D), \dcb(D) \in \CKhI^0(D)$ and $\uca(D), \ucb(D) \in \CKhI^1(D)$.

\begin{example}
\label{ex:trefoil-lee-cycles}
    For the left-handed trefoil diagram of \Cref{fig:equiv-lee}, we have 
    \[
        \dca(D) = X \otimes Y \in \CKhI^0(D)
    \]
    and 
    \[
        \uca(D) = Q(X \otimes Y) \in \CKhI^1(D)
    \]
    where $X$ corresponds to the outer circle, and $Y$ to the inner circle. The cycles $\dcb(D), \ucb(D)$ are obtained by swapping $X$ and $Y$. 
\end{example}

In \cite[Proposition 2.13]{Sano:2020} we showed in the non-involutive setting that the behaviors of the Lee classes under the Reidemeister moves can be described explicitly. The same formula also holds in the involutive setting. Hereafter, $w(D)$ denotes the writhe of $D$, and $r(D)$ denotes the number of Seifert circles of $D$. The difference function $\delta f$ of a unary function $f$ is defined as $\delta f(x, y) = f(y) - f(x)$.  

\begin{proposition}
\label{prop:ca-reid}
    Suppose $h \in R$ is invertible. Let $D, D'$ be strongly invertible link diagrams related by an involutive Reidemeister move. Under the isomorphism
    \[
        \rho\colon \KhI(D) \rightarrow \KhI(D')
    \]
    given in \Cref{prop:invariance}, the equivariant Lee classes modulo torsions correspond as 
    \begin{align*}
        \dca(D) &\xmapsto{\ \rho\ } h^j \dca(D'),&
        \uca(D) &\xmapsto{\ \rho\ } h^j \uca(D'), \\       
        \dcb(D) &\xmapsto{\ \rho\ } h^j \dcb(D'),&
        \ucb(D) &\xmapsto{\ \rho\ } h^j \ucb(D'),
    \end{align*}
    where 
    \[ 
        j = \frac{\delta w(D, D') - \delta r(D, D')}{2}.
    \]
\end{proposition}

\begin{proof}
    From \Cref{prop:invariance}, the isomorphism $\rho$ induces an isomorphism between the long exact sequence of \Cref{prop:ca-SES},
    \[
    \begin{tikzcd}
        \cdots \arrow[r] & { \Kh^{i - 1}(D)} \arrow[r] \arrow[d, "\rho"] & \KhI^i(D) \arrow[r] \arrow[d, "\rho", dashed] & \Kh^i(D) \arrow[r] \arrow[d, "\rho"] & \cdots \\
        \cdots \arrow[r] & { \Kh^{i - 1}(D')} \arrow[r] & \KhI^i(D') \arrow[r] & \Kh^i(D') \arrow[r] & \cdots.
    \end{tikzcd} 
    \]
    Together with \cite[Proposition 2.13]{Sano:2020}, the non-involutive and the involutive Lee classes correspond as 
    \[
    \begin{tikzcd}[arrows=mapsto]
        {\ca(D)} \arrow[r, "Q_*"] \arrow[d, "\rho"] & \uca(D) \arrow[d, "\rho", dashed] & \dca(D) \arrow[r, "q_*"] \arrow[d, "\rho", dashed] & \ca(D) \arrow[d, "\rho"] \\
        {h^j\ca(D')} \arrow[r] & h^j\uca(D') & h^j\dca(D') \arrow[r] & h^j\ca(D').
    \end{tikzcd}
    \]
    The proof for the $\cb$-classes is similar. 
\end{proof}

\subsection{Divisibility of equivariant Lee classes}

Hereafter we assume that $R$ is a PID and $h$ is prime (hence non-zero and non-invertible), typically $(R, h) = (\FF_2[H], H)$. 

\begin{definition}
    Let $M$ be a finitely generated free $R$-module. The \textit{$h$-divisibility} of an element $z$ in $M$ is defined by
    \[
        d_h(z) = \max \{\ k \geq 0 \mid z \in h^k M \ \}.
    \]
\end{definition}

Divisibilities can be compared by homomorphisms. Suppose $M, N$ are finitely generated free $R$-modules, $f\colon M \rightarrow N$ is a homomorphism, and $z \in M$, $w \in N$ are elements such that $f(z) = h^j w$ for some $j \in \ZZ$. Then we have
\[
    d_h(z) \leq j + d_h(w).
\]
In particular when $f$ is an isomorphism, the equality holds. Divisibilities can also be compared after inverting $h$. Namely, if $z, w$ are elements in $M$ such that $z \otimes 1 = h^j (w \otimes 1)$ in $h^{-1}M = M \otimes_R (h^{-1}R)$, then we have
\[
    d_h(z) = j + d_h(w).
\]
See \cite[Lemmas 4.2 -- 4.7]{Sano-Sato:2023} for details. 

In \cite{Sano:2020}, for a link diagram $D$, we defined the \textit{$h$-divisibility} of the Lee class $\ca(D) \in \Kh(D) / \Tor{}$ as
\[
    d_h(D) = \max \{\ 
        k \geq 0 \ \mid \ 
        \ca(D) \in (h^k) (\Kh(D) / \Tor{}) 
    \ \}.
\]
The following is an involutive analogue of this definition. 

\begin{definition}
\label{def:equiv-div}
    For a strongly invertible link diagram $D$, define non-negative integers $\dd_h(D),\ \ud_h(D)$ by the \textit{$h$-divisibility} (modulo torsion) of the equivariant Lee classes $\dca(D), \uca(D) \in \KhI(D) / \Tor{}$ respectively, i.e. 
    \begin{align*}
        \dd_h(D) &= \max \{\ 
            k \geq 0 \ \mid \ 
            \dca(D) \in (h^k) (\KhI(D) / \Tor{}) 
        \ \}, \\ 
        \ud_h(D) &= \max \{\ 
            k \geq 0 \ \mid \ 
            \uca(D) \in (h^k) (\KhI(D) / \Tor{}) 
        \ \}.
    \end{align*}
\end{definition}

Since the automorphism $\sigma$ exchanges $\dca(D, o)$ with $\dcb(D, o)$, and $\uca(D, o)$ with $\ucb(D, o)$, $\dd_h$ and $\ud_h$ may be equivalently defined using the $\cb$-classes. 

\begin{example}
\label{ex:trefoil-divisibility}
    \begin{figure}
        \centering
        \tikzset{every picture/.style={line width=0.75pt}} 

\begin{tikzpicture}[x=0.75pt,y=0.75pt,yscale=-1,xscale=1]

\draw  [color={rgb, 255:red, 208; green, 2; blue, 27 }  ,draw opacity=1 ][line width=1.5]  (52.5,15) .. controls (65.5,14) and (70.5,22) .. (72.5,39) .. controls (74.5,56) and (82.27,46.04) .. (92.5,57) .. controls (102.73,67.96) and (101.5,76) .. (99.5,85) .. controls (97.5,94) and (74.2,96.8) .. (68.33,93.67) .. controls (62.47,90.53) and (62.93,87.67) .. (62.6,84) .. controls (62.27,80.33) and (65.53,77.87) .. (67.67,74.33) .. controls (69.8,70.8) and (73.09,61.71) .. (73,56.33) .. controls (72.91,50.96) and (66.33,47) .. (55,47.33) .. controls (43.67,47.67) and (32.82,48.18) .. (34.33,57.33) .. controls (35.85,66.49) and (42.2,76) .. (44.33,78) .. controls (46.47,80) and (48.6,81) .. (48.6,84) .. controls (48.6,87) and (48.4,90.8) .. (43.5,94) .. controls (38.6,97.2) and (19.5,93) .. (13.5,86) .. controls (7.5,79) and (8.5,64) .. (15.5,59) .. controls (22.5,54) and (33.5,54) .. (35.5,39) .. controls (37.5,24) and (39.5,16) .. (52.5,15) -- cycle ;
\draw  [color={rgb, 255:red, 208; green, 2; blue, 27 }  ,draw opacity=1 ][line width=1.5]  (214.5,15) .. controls (227.5,14) and (232.5,22) .. (234.5,39) .. controls (236.5,56) and (244.27,46.04) .. (254.5,57) .. controls (264.73,67.96) and (263.5,76) .. (261.5,85) .. controls (259.5,94) and (247.5,102) .. (230.5,93) .. controls (213.5,84) and (220.5,85) .. (205.5,94) .. controls (190.5,103) and (181.5,93) .. (175.5,86) .. controls (169.5,79) and (170.5,64) .. (177.5,59) .. controls (184.5,54) and (195.5,54) .. (197.5,39) .. controls (199.5,24) and (201.5,16) .. (214.5,15) -- cycle ;
\draw  [color={rgb, 255:red, 208; green, 2; blue, 27 }  ,draw opacity=1 ][line width=1.5]  (214.5,46) .. controls (226.5,46) and (240,43) .. (233.5,65) .. controls (227,87) and (210,89) .. (200,69) .. controls (190,49) and (202.5,46) .. (214.5,46) -- cycle ;
\draw  [draw opacity=0][fill={rgb, 255:red, 128; green, 128; blue, 128 }  ,fill opacity=1 ] (54.42,84.54) .. controls (54.42,83.64) and (55.14,82.92) .. (56.04,82.92) .. controls (56.94,82.92) and (57.67,83.64) .. (57.67,84.54) .. controls (57.67,85.44) and (56.94,86.17) .. (56.04,86.17) .. controls (55.14,86.17) and (54.42,85.44) .. (54.42,84.54) -- cycle ;
\draw    (115.33,70) -- (159.67,70) ;
\draw [shift={(161.67,70)}, rotate = 180] [color={rgb, 255:red, 0; green, 0; blue, 0 }  ][line width=0.75]    (10.93,-3.29) .. controls (6.95,-1.4) and (3.31,-0.3) .. (0,0) .. controls (3.31,0.3) and (6.95,1.4) .. (10.93,3.29)   ;

\draw (75,7.4) node [anchor=north west][inner sep=0.75pt]    {$\textcolor[rgb]{0.82,0.01,0.11}{X}$};
\draw (237,7.4) node [anchor=north west][inner sep=0.75pt]    {$\textcolor[rgb]{0.82,0.01,0.11}{X}$};
\draw (233.33,68.4) node [anchor=north west][inner sep=0.75pt]    {$\textcolor[rgb]{0.82,0.01,0.11}{X}$};
\draw (130.33,49.4) node [anchor=north west][inner sep=0.75pt]    {$\Delta $};

\end{tikzpicture}
        \caption{Elements $x$ and $dx$.}
        \label{fig:trefoil-divisibility}
    \end{figure}
    Consider the left-handed trefoil diagram of \Cref{fig:equiv-lee}. In \Cref{ex:trefoil-lee-cycles} we had
    \[
        \dca(D) = X \otimes Y \in \CKhI^0(D).
    \]
    Now, consider the element $x$ of homological grading $-1$ depicted in the left side of \Cref{fig:trefoil-divisibility}. Note that $x$ is $\tau$-invariant and hence $(d + 1 + \tau)x = dx = X \otimes X$. Now 
    \begin{align*}
        \dca(D) 
        &\sim X
        \otimes Y + X \otimes X \\
        &= h(X \otimes 1)
    \end{align*}
    and we have $\dd_h(D) \geq 1$. In fact in this case the equality holds. Similarly we have $\ud_h(D) = 1$. 
\end{example}

\begin{proposition}
\label{prop:div-ineq}
    \[
        \dd_h(D) \leq d_h(D) \leq \ud_h(D).
    \]
\end{proposition}

\begin{proof}
    Immediate from \Cref{prop:ca-SES}. 
\end{proof} 

The following two propositions can be proved analogously to \cite[Proposition 3.9]{Sano:2020}. 

\begin{proposition}
    $\dd_h(\bigcirc) = \ud_h(\bigcirc) = 0$.
\end{proposition}

\begin{proposition}
\label{prop:d-positive}
    If $D$ is a positive diagram, then $\dd_h(D) = \ud_h(D) = 0$. 
\end{proposition}


In \cite{Sano:2020,Sano-Sato:2023}, for a (non-involutive) link $L$ with diagram $D$, the invariant $s_h(L)$ is defined as 
\[
    s_h(L) = 2d_h(D) + w(D) - r(D) + 1.
\]
Analogously, \Cref{prop:ca-reid,prop:ca-SES} justify the following definition. 

\begin{definition}
\label{def:equiv-ras}
    For a strongly invertible link $L$ with diagram $D$, define 
    \begin{align*}
        \ds_h(L) &= 2\dd_h(D) + w(D) - r(D) + 1, \\ 
        \us_h(L) &= 2\ud_h(D) + w(D) - r(D) + 1.
    \end{align*}
    The pair $(\ds_h(L), \us_h(L))$ is called the \textit{equivariant Rasmussen} invariant of $L$.
\end{definition}

\begin{proposition}
    \[
        \ds_h(L) \leq s_h(L) \leq \us_h(L).
    \]
\end{proposition}

\begin{proof}
    Immediate from \Cref{prop:div-ineq}.
\end{proof}

\begin{corollary}
    For a strongly invertible knot $K$, we have
    \[
        \ds_h(K) \leq s(K) \leq \us_h(K)
    \]
    where $s(K)$ is the $\FF_2$-Rasmussen invariant of $K$. 
\end{corollary}

\begin{proof}
    It is proved in \cite[Theorem 2, Proposition 4.36]{Sano-Sato:2023} that $s_h(K)$ coincides with $s(K)$ when $\fchar(R) = 2$. 
\end{proof}

\begin{proposition}
    $\ds_h(\bigcirc) = \us_h(\bigcirc) = 0$.
\end{proposition}

\begin{proposition}
\label{prop:s-positive}
    If $K$ is positive with positive diagram $D$, then 
    \[
        \ds_h(K) = \us_h(K) = n(D) - r(D) + 1
    \]
    where $n(D)$ is the number of crossings of $D$. 
\end{proposition}

\begin{corollary}
\label{cor:Tpq}
    For the positive $(p, q)$-torus knot $T_{p, q}$, 
    \[
        \ds_h(T_{p, q}) = \us_h(T_{p, q}) = (p - 1)(q - 1)
    \]
    with respect to the unique inverting involution $\tau$ of $T_{p, q}$.
\end{corollary}

\begin{example}
\label{ex:trefoil-s}
    For the left-handed trefoil $3_1$, we have $\ds_h(K) = \us_h(K) = -2$ from the computations in \Cref{ex:trefoil-divisibility}. For the right-handed trefoil $m(3_1)$, we have $\ds_h(m(3_1)) = \us_h(m(3_1)) = 2$ from \Cref{cor:Tpq}.
\end{example}

\subsection{Reduced variants}

Throughout this subsection, we assume that $D$ is a strongly invertible link diagram which is pointed and normal.
Let $O_+(D)$ denote the subset of $O(D)$ consisting of orientations $o$ whose orientation on the based component coincides with that of $D$. When  $o \in O_+(D)$, we have $-\tau(o) \in O_+(D)$, so we may define the subset
\[
    O^\tau_+(D) := \{\ o \in O_+(D) \mid o = -\tau(o) \ \}
\]
and the quotient set 
\[
    \overline{O}{}^\tau_+(D) := (O_+(D) \setminus O^\tau_+(D)) / (o \sim -\tau(o)).
\]

First, let us review the non-involutive setting, which is extensively studied in \cite{Sano-Sato:2023}. For each $o \in O_+(D)$, the cycle $\ca(D, o)$ lies in the reduced complex $\rCKh(D) \subset \CKh(D)$, which will be denoted $\ca_r(D, o)$ for the sake of distinction. The counterpart $\cb(D, o) = \ca(D, -o)$ lies in $\sigma\rCKh(D) \subset \CKh(D)$. Under the isomorphism $\sigma\rCKh(D) \isom \rCKh'(D)$, let $\cb_r(D, o)$ denote the corresponding cycle in $\rCKh'(D)$, which is obtained by replacing the label of $\cb(D, o)$ on the pointed circle from $Y$ to $1$. From \cite[Corollary 3.15]{Sano-Sato:2023}, when $h$ is invertible, $\rKh(D)$ is freely generated by the reduced Lee classes $\{\ca_r(D, o)\}$, and $\rKh'(D)$ by the coreduced Lee classes $\{\cb_r(D, o)\}$.

Now we return to the involutive setting. The \textit{(co)reduced equivariant Lee cycles} in $\rCKhI(D)$ and in $\rCKhI'(D)$ are defined as in \Cref{prop:equiv-lee-cycles}, using orientations in $O^\tau_+(D)$ and $\overline{O}{}^\tau_+(D)$. The following propositions are easy to verify. 

\begin{proposition}
    Under the maps in the short exact sequence of \Cref{prop:CKhI-ses}, for each $o \in O^\tau_+(D)$, the equivariant Lee cycles in the unreduced, reduced, and coreduced complexes correspond as
    \begin{align*}
        \dca(D, o) &\xmapsto{i} \dca_r(D, o), &
        \dcb(D, o)  &\xmapsto{p} h\dcb_r(D, o), \\
        \uca(D, o) &\xmapsto{i} \uca_r(D, o), &
        \ucb(D, o)  &\xmapsto{p} h\ucb_r(D, o).
    \end{align*}
    Similar equations hold for the cycles corresponding to $[o'] \in \overline{O}{}^\tau_+(D)$. 
\end{proposition}

\begin{proposition}
\label{prop:rkhi-structure}
    If $h \in R$ is invertible, the reduced homology $\rKhI(D)$ is generated by the reduced equivariant Lee classes, and the coreduced homology $\rKhI'(D)$ is generated by the coreduced equivariant Lee classes. More strongly, one may choose bases for these homologies as in \Cref{prop:khi-structure}. 
\end{proposition}

\begin{proposition}
\label{prop:ca-red-reideiemster}
    Suppose $h \in R$ is invertible. Let $D, D'$ be two strongly invertible link diagrams (pointed and normal) related by an involutive Reidemeister move. Under the isomorphism given in \Cref{prop:invariance-reduced}
    \[
        \rho\colon \rKhI(D) \rightarrow \rKhI(D'),
    \]
    the reduced equivariant Lee classes correspond as 
    \[
        \dca_r(D) \xmapsto{\ \rho\ } h^j \dca_r(D'),\quad
        \uca_r(D) \xmapsto{\ \rho\ } h^j \uca_r(D')
    \]
    where 
    \[ 
        j = \frac{\delta w(D, D') - \delta r(D, D')}{2}.
    \]
    Similar statements hold for the coreduced counterparts. 
\end{proposition}

Next, we state the correspondence of the unreduced and the (co)reduced equivariant Lee cycles under the splitting of \Cref{prop:CKhI-split}. First, we consider the non-involutive setting. 

\begin{proposition}
\label{prop:ca-spl}
    Under the splitting of \cite{Wigderson:2016}
    \[
        \CKh(D) \isom \rCKh'(D) \oplus  \rCKh(D),
    \]
    for each $o \in O_+(D)$, the unreduced and the (co)reduced Lee cycles correspond as 
    \[
        \ca(D, o) \mapsto (0,\ \ca_r(D, o)), \quad
        \cb(D, o) \mapsto (h\cb_r(D, o),\ \ca_r(D, o)).
    \]
\end{proposition}

\begin{proof}
    Here we only give a sketch. The isomorphism
    \[
        \CKh(D) = \Cone(f) \xlongrightarrow{\isom} \rCKh'(D) \oplus  \rCKh(D)
    \]
    is given by 
    \[
        \begin{pmatrix}
        1 & \\
        \kappa & 1
        \end{pmatrix}
    \]
    using the null-homotopy $\kappa$ described in \Cref{subsec:reduced}. From this description, it is obvious that $\ca$ maps to $(0, \ca)^T$. To see that $\cb$ maps to $(h\cb_r, \ca)^T$, put
    \[
        \cb = \underline{Y} \otimes x.
    \]
    Here the underline indicates the label corresponding to the pointed circle. With the vector notation $\cb$ is represented as $(h\underline{1} \otimes x,\ \underline{X} \otimes x)^T$. Note that $\underline{1} \otimes x$ is exactly $\cb_r$, so it remains to prove that 
    \[
        h \kappa(\underline{1} \otimes x) + \underline{X} \otimes x = \ca_r.
    \]
    If we define $\kappa_{-1}\colon \rCKh'(D) \rightarrow \rCKh(D)$ by 
    \[
        \kappa_{-1}(\underline{1} \otimes \cdots) = \underline{X} \otimes \cdots
    \]
    and
    \[
        \bar{\kappa} = \sum_{i \geq -1} h^{i + 1}\kappa_i\colon \rCKh'(D) \rightarrow \rCKh(D),
    \]
    then the aimed equation can be written as
    \[
        \bar{\kappa}(\cb_r) = \ca_r.
    \]
    This is a purely algebraic problem and can be proved by the induction on the number of the Seifert circles.
    %
    %
\end{proof}

\begin{proposition}
\label{prop:ca-spl-i}
    Under the splitting given in \Cref{prop:CKhI-split}
    \[
        \CKhI(D) \isom \rCKhI'(D) \oplus  \rCKhI(D),
    \]
    for each $o \in O^\tau_+(D)$, the equivariant Lee cycles correspond as 
    \begin{align*}
        \dca(D, o) &\mapsto (0,\ \dca_r(D, o)), &
        \dcb(D, o) &\mapsto (h\dcb_r(D, o),\ \dca_r(D, o)), \\ 
        \uca(D, o) &\mapsto (0,\ \uca_r(D, o)), &
        \ucb(D, o) &\mapsto (h\ucb_r(D, o),\ \uca_r(D, o)).
    \end{align*}
    Similar correspondences hold for the cycles corresponding to $[o'] \in \overline{O}{}^\tau_+(D)$. 
\end{proposition}

\begin{proof}
    Immediate from \Cref{prop:ca-spl}.
\end{proof}

The following lemmas will be used in the coming sections. Here, the $h$-divisibilities are considered modulo torsions. 

\begin{lemma}
\label{lem:div-ca-red}
    \begin{align*}
        d_h(\dca_r(D)) = d_h(\dcb_r(D)) &= \dd_h(D), \\
        d_h(\uca_r(D)) = d_h(\ucb_r(D)) &= \ud_h(D).
    \end{align*}
\end{lemma}

\begin{proof}
    Immediate from \Cref{prop:ca-spl-i,prop:rkhi-structure}. 
\end{proof}

\begin{lemma} 
\label{lem:div-cz}
    \begin{align*}
        d_h(\dca(D) + \dcb(D)) = \dd_h(D) + 1, \\ 
        d_h(\uca(D) + \ucb(D)) = \ud_h(D) + 1.
    \end{align*}
\end{lemma}

\begin{proof}
    Again immediate from \Cref{prop:ca-spl-i,prop:rkhi-structure}. 
\end{proof}

\Cref{lem:div-ca-red} implies that the equivariant Rasmussen invariant $(\ds_h, \us_h)$ can be equivalently defined using the $h$-divisibility of the (co)reduced equivariant Lee classes. 
    \section{Properties of the equivariant invariants}
\label{sec:4}

In this section properties of $(\ds_h, \us_h)$ stated in \Cref{sec:intro} will be proved in more generality. Throughout, we assume that $R$ is a PID and $h$ is prime.

\subsection{Mirror formula}

\begin{proposition}
\label{prop:ca-pairing}
    Consider a strongly invertible link diagram $D$ and its mirror $D^*$. With the perfect pairing of \Cref{prop:perf-pairing}, we have
    \begin{align*}
        \langle \dca(D),\ \uca(D^*) \rangle 
        &= \langle \uca(D),\ \dca(D^*) \rangle \\
        &= h^{r(D)}.
    \end{align*}
    Similarly for the reduced versions, we have 
    \begin{align*}
        \langle \dca_r(D),\ \uca_r(D^*) \rangle_r 
        &= \langle \uca_r(D),\ \dca_r(D^*) \rangle_r \\
        &= h^{r(D) - 1}.
    \end{align*}
\end{proposition}

\begin{proof}
    Obvious from the observation that the Seifert circles of $D$ and $D^*$ are identical, together with $\langle X, X \rangle = \langle Y, Y \rangle = h$.
\end{proof}

\begin{proposition}
\label{prop:d-mirror}
    For a strongly invertible knot diagram $D$ (which is pointed and normal), we have 
    \[
        \dd_h(D) + \ud_h(D^*) 
        = \ud_h(D) + \dd_h(D^*) 
        = r(D) - 1.
    \]
\end{proposition}

\begin{proof}
    The formula is proved using the reduced Lee classes. Take a generator $z$ of $\rKhI^0(D) \isom R$ and put
    \[
        \dca_r(D) = a h^{d} z
    \]
    where $d = \dd_h(D)$ and $h \nmid a$. Similarly take a generator $w$ of $\rKhI^1(D^*) \isom R[1]$ and put
    \[
        \uca_r(D^*) = b h^{d'} w
    \]
    where $d = \ud_h(D^*)$ and $h \nmid b$. The perfect pairing of \Cref{prop:perf-pairing} induces a perfect pairing
    \[
        \langle \cdot, \cdot \rangle_r\colon (\rKhI(D) / \Tor{}) \otimes (\rKhI(D^*) / \Tor{} ) \rightarrow R,
    \]
    and from \Cref{prop:ca-pairing} we have 
    \[
        \langle \dca_r(D), \uca_r(D^*) \rangle_r = ab h^{d + d'} \langle z, w \rangle_r = h^{r(D) - 1}. 
    \]
    Now $\langle z, w \rangle_r$ must be a unit of $R$, and since $h$ is assumed to be prime, we must have that $a, b \in R$ are both units and
    \[
        \dd_h(D) + \ud_h(D^*) = r(D) - 1.
    \]
    The other equation follows from a similar argument. 
\end{proof}

\begin{proposition}
\label{prop:s-mirror}
    For a strongly invertible knot $K$, 
    \[
        \ds_h(K^*) = -\us_h(K).
    \]
\end{proposition}

\begin{proof}
    Immediate from \Cref{prop:d-mirror}. 
\end{proof}

\begin{example}
    Compare \Cref{ex:trefoil-s}. 
\end{example}

The proof of \Cref{prop:d-mirror} also shows that the following two elements
\begin{align*}
    \dcx_r(D) &= h^{-\dd_h(D)} \dca_r(D),\\
    \ucx_r(D) &= h^{-\ud_h(D)} \uca_r(D)
\end{align*}
form a basis of $\rKhI(D) / \Tor{} \isom R \oplus R[1]$. Similarly,
\begin{align*}
    \dcz_r(D) &= h^{-\dd_h(D)} \dcb_r(D),\\
    \ucz_r(D) &= h^{-\ud_h(D)} \ucb_r(D)
\end{align*}
form a basis of $\rKhI'(D) / \Tor{} \isom R \oplus R[1]$. Under the identification of \Cref{prop:ca-spl-i}, the elements corresponding to $\dcx_r(D), \ucx_r(D)$ are
\begin{align*}
    \dcx(D) &= h^{-\dd_h(D)} \dca(D),\\
    \ucx(D) &= h^{-\ud_h(D)} \uca(D)
\end{align*}
and the elements corresponding to $\dcz_r(D), \ucz_r(D)$ are
\begin{align*}
    \dcz(D) &= h^{-\dd_h(D) - 1}(\dca(D) + \dcb(D)),\\
    \ucz(D) &= h^{-\ud_h(D) - 1}(\uca(D) + \ucb(D)).    
\end{align*}
Thus the four elements $\dcz(D), \dcx(D), \ucz(D), \ucx(D)$ form a basis of $\KhI(D) / \Tor{} \isom R^2 \oplus R[1]^2$. \Cref{prop:ca-reid,prop:ca-red-reideiemster} implies that all of these classes are invariant under the Reidemeister moves. 

In particular when $R$ is graded and $\deg(h) = -2$, we see that 
\begin{align*}
    \ds_h(K) &= \gr_q(\dcx_r(K)),\\
    \us_h(K) &= \gr_q(\ucx_r(K)).
\end{align*}
Thus when $(R, h) = (\FF_2[H], H)$ the definition of $(\ds, \us)$ for strongly invertible knots given in \Cref{sec:intro} coincides with \Cref{def:equiv-ras}. We summarize, 

\begin{proposition}
\label{prop:canon-classes}
    \begin{align*}
        \KhI(K) &= 
        R \langle \dcz(K), \dcx(K), \ucz(K), \ucx(K) \rangle
        \oplus
        (\Tor),\\
        \rKhI(K) &= 
        R \langle \dcx_r(K), \ucx_r(K) \rangle
        \oplus
        (\Tor),\\
        \rKhI'(K) &= 
        R \langle \dcz_r(K), \ucz_r(K) \rangle
        \oplus
        (\Tor).
    \end{align*}
\end{proposition}

\begin{example}
    For the simplest example $D = \bigcirc$, we have 
    \begin{align*}
        \dca(D) &= X, &\dcb(D) &= Y,\\
        \uca(D) &= QX, &\ucb(D) &= QY
    \end{align*}
    and $\dd_h(D) = \ud_h(D) = 0$, so
    \begin{align*}
        \dcx(D) &= X, &\dcz(D) &= 1,\\
        \ucx(D) &= QX, &\ucz(D) &= Q1
    \end{align*}
    and we obtain 
    \[
        \KhI(D) = R\langle 1, X, Q1, QX \rangle = A \oplus QA
    \]
    as expected.
\end{example}

\subsection{Connected sum formula}

Arguments in this section are inspired by \cite{Hendricks-Manolescu-Zemke:2018}, where the connected sum formula for the $\dd$-, $\ud$-invariants in involutive Heegaard Floer homology is proved. 

For strongly invertible links $L, L'$, the (equivariant) \textit{disjoint union} $L \sqcup L'$ and the (equivariant) \textit{connected sum} $L \#_b L'$ along an equivariant band $b$ are defined in the obvious ways so that the resulting links are also strongly invertible. Note that  different choices of $b$ will in general give non-equivalent links, but here we make the choice implicit and omit $b$ from the notation%
\footnote{
    For \textit{directed} strongly invertible knots $K, K'$, there is a canonical choice of the band $b$ from $K$ to $K'$, and the equivariant connected sum is defined without ambiguity. See \cite{Sakuma:86}. 
}. 
The corresponding operations for strongly invertible link diagrams $D, D'$ are also defined. When we write $D \sqcup D'$, it is assumed that $D$ and $D'$ are disjoint as diagrams. When we write $D \# D'$, it is assumed that the band is untwisted and no crossings are produced by the surgery. 

\begin{proposition}
\label{prop:CKhI-disj}
    There is a canonical isomorphism
    \begin{align*}
        &\CKhI(D \sqcup D') \\ 
        &\isom \Cone(\ \CKh(D) \otimes \CKh(D') \xrightarrow{Q(1 \otimes 1 + \tau \otimes \tau)} Q (\CKh(D) \otimes \CKh(D'))\ ).
    \end{align*}
    Under this identification, we have 
    \[
        \dca(D \sqcup D') = \ca(D) \otimes \ca(D'),\quad
        \uca(D \sqcup D') = Q(\ca(D) \otimes \ca(D')). 
    \]
\end{proposition}

\begin{proof}
    Obvious from the canonical isomorphism 
    \[
        \CKh(D \sqcup D') \isom \CKh(D) \otimes \CKh(D'),
    \]
    which holds for any $(R, h)$ (see \cite[Section 7.4]{Khovanov:2000}). 
\end{proof}

\begin{proposition}
\label{prop:disj-connsum-corresp}
    There are chain maps
    \[
\begin{tikzcd}
\CKhI(D \sqcup D') \arrow[r, "m", shift left] & \CKhI(D \# D') \arrow[l, "\Delta", shift left]
\end{tikzcd}
    \]
    corresponding to the band surgery from $D \sqcup D'$ to $D \# D'$ and its reverse. 
\end{proposition}

\begin{proof}
    The corresponding maps in the non-involutive setting are $\tau$-invariant. 
\end{proof}

\begin{proposition}
\label{prop:disj-connsum-ineq}
    For strongly invertible links $L, L'$, we have
    \[
        \ds_h(L \# L')  - 1 \leq \ds_h(L \sqcup L') \leq \ds_h(L \# L') + 1
    \]
    and similarly for $\us_h$.
\end{proposition}

\begin{proof}
    Under the maps of \Cref{prop:disj-connsum-corresp}, we have 
    \[
        m(\ca(D) \otimes \ca(D')) = h \ca(D \# D'), \quad
        \Delta \ca(D \# D') = \ca(D) \otimes \ca(D')
    \]
    hence
    \[
        \dd_h(D \# D') \leq \dd_h(D \sqcup D') \leq \dd_h(D \# D') + 1.
    \]
    This gives the desired inequality. 
\end{proof}

\begin{lemma}
\label{lem:tensor}
    Suppose $C, C'$ are $\tau$-complexes over $\FF_2$. Put 
    \[
        C_\tau = \Cone(C \xrightarrow{1 + \tau} C), \quad
        C'_\tau = \Cone(C' \xrightarrow{1 + \tau} C')
    \]
    and 
    \[
        C^\otimes_\tau = \Cone(C \otimes C' \xrightarrow{1 \otimes 1 + \tau \otimes \tau} C \otimes C').
    \]
    Let $z \in C$ and $z' \in C'$ be $\tau$-invariant cycles. In the following, $\sim$ denotes homologous.

    \begin{enumerate}[leftmargin=*]
        \item Let $x, y \in C$ and $x', y' \in C'$ be elements such that 
        \[
            \begin{pmatrix}z \\ 0\end{pmatrix} \sim \begin{pmatrix}x \\ y\end{pmatrix},\quad 
            \begin{pmatrix}z' \\ 0\end{pmatrix} \sim \begin{pmatrix}x' \\ y'\end{pmatrix}
        \]
        in $C_\tau$ and in $C'_\tau$ respectively. Then in $C^\otimes_\tau$, 
        \[
            \begin{pmatrix}z \otimes z' \\ 0\end{pmatrix} \sim
            \begin{pmatrix}x \otimes x' \\ x \otimes y' + y \otimes \tau x' \end{pmatrix}.
        \]

        \item Let $x, y \in C$ and $x', y' \in C'$ be elements such that
        \[
            \begin{pmatrix}z \\ 0\end{pmatrix} \sim \begin{pmatrix}x \\ y\end{pmatrix},\quad 
            \begin{pmatrix}0 \\ z'\end{pmatrix} \sim \begin{pmatrix}x' \\ y'\end{pmatrix}
        \]
        in $C_\tau$ and in $C'_\tau$ respectively. Then in $C^\otimes_\tau$, 
        \[
            \begin{pmatrix}0 \\ z \otimes z'\end{pmatrix} \sim 
            \begin{pmatrix}x \otimes x' \\ y \otimes x' + \tau x \otimes y' \end{pmatrix}.
        \]
    \end{enumerate}
\end{lemma}

\begin{proof}
    \begin{enumerate}[leftmargin=*]
        \item Put 
        \[
            \begin{pmatrix}z \\ 0\end{pmatrix} - \begin{pmatrix}x \\ y\end{pmatrix} = 
            \begin{pmatrix}d & \\ 1 + \tau & d\end{pmatrix} \begin{pmatrix}a \\ b\end{pmatrix}
        \]
        and
        \[
            \begin{pmatrix}z' \\ 0\end{pmatrix} - \begin{pmatrix}x' \\ y'\end{pmatrix} = 
            \begin{pmatrix}d & \\ 1 + \tau & d\end{pmatrix} \begin{pmatrix}a' \\ b'\end{pmatrix}.
        \]
        Then the boundary of
        \[
            \begin{pmatrix}
                x \otimes a' + a \otimes x' + a \otimes da' \\ 
                x \otimes b' + b \otimes \tau x' + (1 + \tau)a \otimes \tau a'
            \end{pmatrix}
        \]
        gives the desired relation. 

        \item Put 
        \[
            \begin{pmatrix}z \\ 0\end{pmatrix} - \begin{pmatrix}x \\ y\end{pmatrix} = 
            \begin{pmatrix}d & \\ 1 + \tau & d\end{pmatrix} \begin{pmatrix}a \\ b\end{pmatrix}
        \]
        and
        \[
            \begin{pmatrix}0 \\ z'\end{pmatrix} - \begin{pmatrix}x' \\ y'\end{pmatrix} = 
            \begin{pmatrix}d & \\ 1 + \tau & d\end{pmatrix} \begin{pmatrix}a' \\ b'\end{pmatrix}.
        \]
        Then the boundary of 
        \[
            \begin{pmatrix}
                x \otimes a'\\ 
                x \otimes b' + \tau a \otimes y' + (1 \otimes 1 + \tau \otimes \tau )(a \otimes a') + da \otimes b' + b \otimes da'
            \end{pmatrix}
        \]
        gives the desired relation. 
    \end{enumerate}
\end{proof}

\begin{proposition}
\label{prop:s-conn-sum}
    For strongly invertible knots $K, K'$, 
    \[
        \ds_h(K) + \ds_h(K') 
        \leq \ds_h(K \# K') 
        \leq \ds_h(K) + \us_h(K') 
        \leq \us_h(K \# K') 
        \leq \us_h(K) + \us_h(K').
    \]
\end{proposition}

\begin{proof}
    We will prove the first and the third inequalities, which is sufficient to prove \Cref{prop:s-conn-sum} from the mirror formula \Cref{prop:s-mirror}.
    Take $z = \ca(D) + \cb(D)$ in $\CKh(D)$ and $z' = \ca(D') + \cb(D')$ in $\CKh(D')$. Observe that under the chain map $m$ of \Cref{prop:disj-connsum-corresp}, we have 
    \[
        \begin{pmatrix}z \otimes z' \\ 0\end{pmatrix}
        \xmapsto{m}
        h\begin{pmatrix}\ca(D \# D') + \cb(D \# D') \\ 0\end{pmatrix} = h (\dca(D \# D') + \dcb(D \# D'))
    \]
    and 
    \[
        \begin{pmatrix}0 \\ z \otimes z'\end{pmatrix}
        \xmapsto{m}
        h\begin{pmatrix}0 \\ \ca(D \# D') + \cb(D \# D')\end{pmatrix} = h (\uca(D \# D') + \ucb(D \# D')).
    \]
    
    From \Cref{lem:div-cz}, there are elements $x, y \in \CKh(D)$ such that,  
    \[
        \begin{pmatrix}z \\ 0\end{pmatrix} \sim h^{\dd_h(D) + 1} \begin{pmatrix}x \\ y\end{pmatrix}
    \]
    in $\CKhI(D)$, modulo torsion in homology. By inverting $h$, we may assume that they are strictly homologous in $h^{-1}\CKhI(D)$. Similarly, there are elements $x', y' \in \CKh(D')$ such that 
    \[
        \begin{pmatrix}z' \\ 0\end{pmatrix} \sim h^{\dd_h(D') + 1} \begin{pmatrix}x' \\ y'\end{pmatrix}
    \]
    in $h^{-1}\CKhI(D')$. From \Cref{lem:tensor} (1), we have 
    \[
        \begin{pmatrix}z \otimes z' \\ 0\end{pmatrix} \sim
        h^{\dd_h(D) + \dd_h(D') + 2}\begin{pmatrix}x \otimes x' \\ x \otimes y' + y \otimes \tau x' \end{pmatrix}
    \]
    in $h^{-1}\CKhI(D \sqcup D')$. Under the map $m$ of \Cref{prop:disj-connsum-corresp}, the left-hand side maps to 
    \[
        h (\dca(D \# D') + \dcb(D \# D'))
    \]
    in $h^{-1}\CKhI(D \# D')$. Its homology class in $h^{-1}\KhI(D \# D')$ is the $h^{\dd_h(D \# D') + 2}$ multiple of the class $\dcz(D \# D')$ of \Cref{prop:canon-classes}. The homology class of 
    \[
        m\begin{pmatrix}x \otimes x' \\ x \otimes y' + y \otimes \tau x' \end{pmatrix}
    \]
    is also some $h^e$ multiple of $\dcz(D \# D')$ for some $e \geq 0$ (because $x, y, x'$ and $y'$ were taken before inverting $h$). Thus it follows that 
    \[
        \dd_h(D) + \dd_h(D') \leq \dd_h(D \# D').
    \]
    This implies the first inequality. Similarly for the third inequality, there are elements $x'', y'' \in \CKh(D')$ such that 
    \[
        \begin{pmatrix}
            0 \\ 
            z'
        \end{pmatrix}
        \sim 
        h^{\ud_h(D') + 1} \begin{pmatrix}
            x'' \\ 
            y''
        \end{pmatrix}
    \]
    in $h^{-1}\CKhI(D')$. From \Cref{lem:tensor} (2), we have 
    \[
        \begin{pmatrix}
            0 \\ 
            z \otimes z'
        \end{pmatrix} 
        \sim 
        h^{\dd_h(D) + \ud_h(D') + 2}
        \begin{pmatrix}
            x \otimes x'' \\ 
            y \otimes x'' + \tau x \otimes y''
        \end{pmatrix}.
    \]
    By a similar argument, we obtain 
    \[
        \dd_h(D) + \ud_h(D') \leq \ud_h(D \# D')
    \]
    which implies the third inequality.
\end{proof}

\subsection{Behavior under crossing changes}


\begin{proposition}
\label{prop:x-ch}
    Let $D^+$ be a link diagram with at least one positive crossing, and $D^-$ be a diagram obtained from $D^+$ by applying a negative crossing change to one of the positive crossings of $D^+$. There are homological grading preserving chain maps
    \[
    \begin{tikzcd}
\CKh(D^+) \arrow[r, "\Phi^-", shift left] & \CKh(D^-) \arrow[l, "\Phi^+", shift left]
    \end{tikzcd}
    \]
    such that the Lee cycles correspond as 
    \[
        \ca(D^+) \xmapsto{\Phi^-} \ca(D^-),\quad
        \ca(D^-) \xmapsto{\Phi^+} h\ca(D^+).
    \]
\end{proposition}

\begin{figure}[t]
    \centering
    \resizebox{!}{14em}{
    \tikzset{every picture/.style={line width=0.75pt}} 

\begin{tikzpicture}[x=0.75pt,y=0.75pt,yscale=-1,xscale=1]

\draw  [dash pattern={on 4.5pt off 4.5pt}] (18,112) .. controls (18,92.12) and (34.12,76) .. (54,76) .. controls (73.88,76) and (90,92.12) .. (90,112) .. controls (90,131.88) and (73.88,148) .. (54,148) .. controls (34.12,148) and (18,131.88) .. (18,112) -- cycle ;
\draw [color={rgb, 255:red, 0; green, 0; blue, 0 }  ,draw opacity=1 ][line width=1.5]    (30.19,88.12) -- (78.07,136) ;
\draw [shift={(28.07,86)}, rotate = 45] [color={rgb, 255:red, 0; green, 0; blue, 0 }  ,draw opacity=1 ][line width=1.5]    (8.53,-2.57) .. controls (5.42,-1.09) and (2.58,-0.23) .. (0,0) .. controls (2.58,0.23) and (5.42,1.09) .. (8.53,2.57)   ;
\draw  [color={rgb, 255:red, 255; green, 255; blue, 255 }  ,draw opacity=1 ][fill={rgb, 255:red, 255; green, 255; blue, 255 }  ,fill opacity=1 ] (47.71,111.71) .. controls (47.71,108.36) and (50.43,105.64) .. (53.79,105.64) .. controls (57.14,105.64) and (59.86,108.36) .. (59.86,111.71) .. controls (59.86,115.07) and (57.14,117.79) .. (53.79,117.79) .. controls (50.43,117.79) and (47.71,115.07) .. (47.71,111.71) -- cycle ;
\draw [color={rgb, 255:red, 0; green, 0; blue, 0 }  ,draw opacity=1 ][line width=1.5]    (76.29,88.11) -- (27.71,136) ;
\draw [shift={(78.43,86)}, rotate = 135.41] [color={rgb, 255:red, 0; green, 0; blue, 0 }  ,draw opacity=1 ][line width=1.5]    (8.53,-2.57) .. controls (5.42,-1.09) and (2.58,-0.23) .. (0,0) .. controls (2.58,0.23) and (5.42,1.09) .. (8.53,2.57)   ;
\draw [color={rgb, 255:red, 0; green, 0; blue, 0 }  ,draw opacity=1 ][line width=1.5]    (154.93,67.56) .. controls (166,54.31) and (165,28.31) .. (155.29,17) ;
\draw  [dash pattern={on 4.5pt off 4.5pt}] (145,42) .. controls (145,22.12) and (161.12,6) .. (181,6) .. controls (200.88,6) and (217,22.12) .. (217,42) .. controls (217,61.88) and (200.88,78) .. (181,78) .. controls (161.12,78) and (145,61.88) .. (145,42) -- cycle ;
\draw  [dash pattern={on 4.5pt off 4.5pt}] (283,110) .. controls (283,90.12) and (299.12,74) .. (319,74) .. controls (338.88,74) and (355,90.12) .. (355,110) .. controls (355,129.88) and (338.88,146) .. (319,146) .. controls (299.12,146) and (283,129.88) .. (283,110) -- cycle ;
\draw [color={rgb, 255:red, 0; green, 0; blue, 0 }  ,draw opacity=1 ][line width=1.5]    (341.29,86.11) -- (292.71,134) ;
\draw [shift={(343.43,84)}, rotate = 135.41] [color={rgb, 255:red, 0; green, 0; blue, 0 }  ,draw opacity=1 ][line width=1.5]    (8.53,-2.57) .. controls (5.42,-1.09) and (2.58,-0.23) .. (0,0) .. controls (2.58,0.23) and (5.42,1.09) .. (8.53,2.57)   ;
\draw  [color={rgb, 255:red, 255; green, 255; blue, 255 }  ,draw opacity=1 ][fill={rgb, 255:red, 255; green, 255; blue, 255 }  ,fill opacity=1 ] (312.71,109.71) .. controls (312.71,106.36) and (315.43,103.64) .. (318.79,103.64) .. controls (322.14,103.64) and (324.86,106.36) .. (324.86,109.71) .. controls (324.86,113.07) and (322.14,115.79) .. (318.79,115.79) .. controls (315.43,115.79) and (312.71,113.07) .. (312.71,109.71) -- cycle ;
\draw [color={rgb, 255:red, 0; green, 0; blue, 0 }  ,draw opacity=1 ][line width=1.5]    (295.19,86.12) -- (333.04,123.96) -- (343.07,134) ;
\draw [shift={(293.07,84)}, rotate = 45] [color={rgb, 255:red, 0; green, 0; blue, 0 }  ,draw opacity=1 ][line width=1.5]    (8.53,-2.57) .. controls (5.42,-1.09) and (2.58,-0.23) .. (0,0) .. controls (2.58,0.23) and (5.42,1.09) .. (8.53,2.57)   ;
\draw  [dash pattern={on 4.5pt off 4.5pt}] (147,176) .. controls (147,156.12) and (163.12,140) .. (183,140) .. controls (202.88,140) and (219,156.12) .. (219,176) .. controls (219,195.88) and (202.88,212) .. (183,212) .. controls (163.12,212) and (147,195.88) .. (147,176) -- cycle ;
\draw    (100,92.49) -- (135.42,64.83) ;
\draw [shift={(137,63.6)}, rotate = 142.01] [color={rgb, 255:red, 0; green, 0; blue, 0 }  ][line width=0.75]    (10.93,-3.29) .. controls (6.95,-1.4) and (3.31,-0.3) .. (0,0) .. controls (3.31,0.3) and (6.95,1.4) .. (10.93,3.29)   ;
\draw    (238.58,152.17) -- (274,124.51) ;
\draw [shift={(237,153.4)}, rotate = 322.01] [color={rgb, 255:red, 0; green, 0; blue, 0 }  ][line width=0.75]    (10.93,-3.29) .. controls (6.95,-1.4) and (3.31,-0.3) .. (0,0) .. controls (3.31,0.3) and (6.95,1.4) .. (10.93,3.29)   ;
\draw    (100,124.51) -- (135.42,152.17) ;
\draw [shift={(137,153.4)}, rotate = 217.99] [color={rgb, 255:red, 0; green, 0; blue, 0 }  ][line width=0.75]    (10.93,-3.29) .. controls (6.95,-1.4) and (3.31,-0.3) .. (0,0) .. controls (3.31,0.3) and (6.95,1.4) .. (10.93,3.29)   ;
\draw    (238.58,64.83) -- (274,92.49) ;
\draw [shift={(237,63.6)}, rotate = 37.99] [color={rgb, 255:red, 0; green, 0; blue, 0 }  ][line width=0.75]    (10.93,-3.29) .. controls (6.95,-1.4) and (3.31,-0.3) .. (0,0) .. controls (3.31,0.3) and (6.95,1.4) .. (10.93,3.29)   ;

\draw [color={rgb, 255:red, 208; green, 2; blue, 27 }  ,draw opacity=1 ][line width=2.25]    (163,42.5) -- (198,42.5) ;
\draw [color={rgb, 255:red, 0; green, 0; blue, 0 }  ,draw opacity=1 ][line width=1.5]    (207.07,67.56) .. controls (196,54.31) and (197,28.31) .. (206.71,17) ;
\draw [color={rgb, 255:red, 0; green, 0; blue, 0 }  ,draw opacity=1 ][line width=1.5]    (157.72,150.21) .. controls (170.97,161.28) and (196.97,160.28) .. (208.28,150.57) ;
\draw [color={rgb, 255:red, 208; green, 2; blue, 27 }  ,draw opacity=1 ][line width=2.25]    (182.78,158.28) -- (182.78,193.28) ;
\draw [color={rgb, 255:red, 0; green, 0; blue, 0 }  ,draw opacity=1 ][line width=1.5]    (157.72,202.35) .. controls (170.97,191.28) and (196.97,192.28) .. (208.28,201.99) ;

\draw (46,152.9) node [anchor=north west][inner sep=0.75pt]    {$D^{+}$};
\draw (313,152.9) node [anchor=north west][inner sep=0.75pt]    {$D^{-}$};
\draw (173,83.4) node [anchor=north west][inner sep=0.75pt]    {$D_{0}$};
\draw (176,217.4) node [anchor=north west][inner sep=0.75pt]    {$D_{1}$};
\draw (109,139.4) node [anchor=north west][inner sep=0.75pt]    {$1$};
\draw (108,62.4) node [anchor=north west][inner sep=0.75pt]    {$0$};
\draw (256,139.4) node [anchor=north west][inner sep=0.75pt]    {$0$};
\draw (257,62.4) node [anchor=north west][inner sep=0.75pt]    {$1$};
\draw (178,21.9) node [anchor=north west][inner sep=0.75pt]    {$\textcolor[rgb]{0.82,0.01,0.11}{e}$};
\draw (190.5,167.4) node [anchor=north west][inner sep=0.75pt]    {$\textcolor[rgb]{0.82,0.01,0.11}{e'}$};

\end{tikzpicture}
    }
    \caption{Diagrams $D^\pm$ and their resolutions.}
    \label{fig:xch-D01}
\end{figure}

\begin{proof}
    Let $x$ be the positive crossing of $D^+$ on which the crossing change is performed. Let $D_0, D_1$ be the $0$-, $1$-resolved diagram of $D^+$ at $x$ respectively. Then $\CKh(D^+)$ may be described as a cone of the surgery map
    \[
        \CKh(D_0) \xrightarrow{e} \CKh(D_1).
    \]
    Similarly, by $1$-resolving $D^-$ at the corresponding crossing, we see that $\CKh(D^-)$ can be described as a cone of 
    \[
        \CKh(D_1) \xrightarrow{e'} \CKh(D_0). 
    \]
    The setup is depicted in \Cref{fig:xch-D01} where the red arcs indicate the corresponding surgery maps. Note that $\ca(D^+)$ and $\ca(D^-)$ are identical, and that they both belong to $\CKh(D_0)$. 
    
    Now we define chain maps 
    \[
    \begin{tikzcd}
        \CKh(D^+) \arrow[r, "\Phi^-", shift left] & \CKh(D^-) \arrow[l, "\Phi^+", shift left]
    \end{tikzcd}
    \]
    so that they fit into the following commutative diagram 
    \[
    \begin{tikzcd}[row sep=3em, column sep=3em]
{\color{gray}0} \arrow[r, gray] \arrow[d, "0", leftrightarrow, gray, dashed] & \CKh(D_0) \arrow[d, "\Phi^-"', shift right, dashed] \arrow[r, "e"]   & \CKh(D_1) \arrow[d, "0", leftrightarrow, gray, dashed] \\
\CKh(D_1) \arrow[r, "e'"]      & \CKh(D_0) \arrow[u, "\Phi^+"', shift right, dashed] \arrow[r, gray] & {\color{gray}0}                    \end{tikzcd}
    \]
    giving the desired chain maps between the complexes. 
    First, define $\Phi^- = \id_{D_0}$, which is obviously a chain map satisfying 
    \[
        \ca(D^+) \xmapsto{\Phi^-} \ca(D^-).
    \]
    Next, we define $\Phi^+$ as 
    \begin{center}
        \tikzset{every picture/.style={line width=0.75pt}} 

\begin{tikzpicture}[x=0.75pt,y=0.75pt,yscale=-.7,xscale=.7]

\draw  [dash pattern={on 4.5pt off 4.5pt}] (9,45) .. controls (9,25.12) and (25.12,9) .. (45,9) .. controls (64.88,9) and (81,25.12) .. (81,45) .. controls (81,64.88) and (64.88,81) .. (45,81) .. controls (25.12,81) and (9,64.88) .. (9,45) -- cycle ;
\draw  [dash pattern={on 4.5pt off 4.5pt}] (128,45) .. controls (128,25.12) and (144.12,9) .. (164,9) .. controls (183.88,9) and (200,25.12) .. (200,45) .. controls (200,64.88) and (183.88,81) .. (164,81) .. controls (144.12,81) and (128,64.88) .. (128,45) -- cycle ;
\draw  [fill={rgb, 255:red, 0; green, 0; blue, 0 }  ,fill opacity=1 ] (177,45) .. controls (177,42.51) and (179.01,40.5) .. (181.5,40.5) .. controls (183.99,40.5) and (186,42.51) .. (186,45) .. controls (186,47.49) and (183.99,49.5) .. (181.5,49.5) .. controls (179.01,49.5) and (177,47.49) .. (177,45) -- cycle ;
\draw  [fill={rgb, 255:red, 0; green, 0; blue, 0 }  ,fill opacity=1 ] (22,45) .. controls (22,42.51) and (24.01,40.5) .. (26.5,40.5) .. controls (28.99,40.5) and (31,42.51) .. (31,45) .. controls (31,47.49) and (28.99,49.5) .. (26.5,49.5) .. controls (24.01,49.5) and (22,47.49) .. (22,45) -- cycle ;
\draw [color={rgb, 255:red, 0; green, 0; blue, 0 }  ,draw opacity=1 ][line width=1.5]    (18.93,70) .. controls (30,56.76) and (29,30.76) .. (19.29,19.44) ;
\draw [color={rgb, 255:red, 0; green, 0; blue, 0 }  ,draw opacity=1 ][line width=1.5]    (71.07,70) .. controls (60,56.76) and (61,30.76) .. (70.71,19.44) ;
\draw [color={rgb, 255:red, 0; green, 0; blue, 0 }  ,draw opacity=1 ][line width=1.5]    (137.93,70) .. controls (149,56.76) and (148,30.76) .. (138.29,19.44) ;
\draw [color={rgb, 255:red, 0; green, 0; blue, 0 }  ,draw opacity=1 ][line width=1.5]    (190.07,70) .. controls (179,56.76) and (180,30.76) .. (189.71,19.44) ;

\draw (100,35.4) node [anchor=north west][inner sep=0.75pt]    {$+$};

\end{tikzpicture}
    \end{center}
    where each dot represents the multiplication by $X$ on the circle it is drawn on. To verify that $\Phi^+$ is a chain map, it suffices to show that $e \Phi^+ = 0 = \Phi^+ e'$. The first equation can be described pictorially as
    \begin{center}
        \tikzset{every picture/.style={line width=0.75pt}} 

\begin{tikzpicture}[x=0.75pt,y=0.75pt,yscale=-.7,xscale=.7]

\draw [color={rgb, 255:red, 208; green, 2; blue, 27 }  ,draw opacity=1 ][line width=2.25]    (150.25,46) -- (187.75,46) ;
\draw [color={rgb, 255:red, 208; green, 2; blue, 27 }  ,draw opacity=1 ][line width=2.25]    (31.5,46) -- (69,46) ;
\draw  [dash pattern={on 4.5pt off 4.5pt}] (14,46) .. controls (14,26.12) and (30.12,10) .. (50,10) .. controls (69.88,10) and (86,26.12) .. (86,46) .. controls (86,65.88) and (69.88,82) .. (50,82) .. controls (30.12,82) and (14,65.88) .. (14,46) -- cycle ;
\draw  [dash pattern={on 4.5pt off 4.5pt}] (133,46) .. controls (133,26.12) and (149.12,10) .. (169,10) .. controls (188.88,10) and (205,26.12) .. (205,46) .. controls (205,65.88) and (188.88,82) .. (169,82) .. controls (149.12,82) and (133,65.88) .. (133,46) -- cycle ;
\draw  [fill={rgb, 255:red, 0; green, 0; blue, 0 }  ,fill opacity=1 ] (182,46) .. controls (182,43.51) and (184.01,41.5) .. (186.5,41.5) .. controls (188.99,41.5) and (191,43.51) .. (191,46) .. controls (191,48.49) and (188.99,50.5) .. (186.5,50.5) .. controls (184.01,50.5) and (182,48.49) .. (182,46) -- cycle ;
\draw  [fill={rgb, 255:red, 0; green, 0; blue, 0 }  ,fill opacity=1 ] (27,46) .. controls (27,43.51) and (29.01,41.5) .. (31.5,41.5) .. controls (33.99,41.5) and (36,43.51) .. (36,46) .. controls (36,48.49) and (33.99,50.5) .. (31.5,50.5) .. controls (29.01,50.5) and (27,48.49) .. (27,46) -- cycle ;
\draw [color={rgb, 255:red, 0; green, 0; blue, 0 }  ,draw opacity=1 ][line width=1.5]    (23.93,71) .. controls (35,57.76) and (34,31.76) .. (24.29,20.44) ;
\draw [color={rgb, 255:red, 0; green, 0; blue, 0 }  ,draw opacity=1 ][line width=1.5]    (76.07,71) .. controls (65,57.76) and (66,31.76) .. (75.71,20.44) ;
\draw [color={rgb, 255:red, 0; green, 0; blue, 0 }  ,draw opacity=1 ][line width=1.5]    (142.93,71) .. controls (154,57.76) and (153,31.76) .. (143.29,20.44) ;
\draw [color={rgb, 255:red, 0; green, 0; blue, 0 }  ,draw opacity=1 ][line width=1.5]    (195.07,71) .. controls (184,57.76) and (185,31.76) .. (194.71,20.44) ;

\draw (223,33.4) node [anchor=north west][inner sep=0.75pt]  [font=\large]  {$=\ 0$};
\draw (105,36.4) node [anchor=north west][inner sep=0.75pt]    {$+$};

\end{tikzpicture}
    \end{center}
    which obviously holds, since dots can move freely within their connected components. The second equation can be proved similarly. Finally, we see that 
    \[
        \ca(D^-) \xmapsto{\Phi^+} h\ca(D^+)
    \]
    from the local description of $\ca(D^+)$ together with $X^2 = hX$ and $XY = 0$. 
\end{proof}

\begin{remark}
    The crossing change maps $\Phi^\pm$ of \Cref{prop:x-ch} partially appear in \cite[Figures 3,4]{Alishahi:2017}. $\Phi^+$ also appears in \cite[Section 3]{Ito-Yoshida:2021} in the form of a morphism in the category $\Cob^3$.
\end{remark}

\begin{proposition}
\label{prop:s-xch}
    Let $L^+, L^-$ be links such that $L^-$ is obtained by applying a negative crossing change to $L^+$. Then
    \[
        s_h(L^-) \leq s_h(L^+) \leq s_h(L^-) + 2. 
    \]
\end{proposition}

\begin{proof}
    Let $D^+, D^-$ be diagrams of $L^+, L^-$ respectively, such that $D^-$ is obtained by applying a negative crossing change to a single crossing of $D^+$. From \Cref{prop:x-ch}, we have 
    \[
        d_h(D^+) \leq d_h(D^-) \leq d_h(D^+) + 1.
    \]
    Thus,
    \begin{align*}
        s_h(L^-) &= 2d_h(D^-) + w(D^-) - r(D^-) + 1\\
        &\leq 2(d_h(D^+) + 1) + (w(D^+) - 2) - r(D^+) + 1\\
        &= s_h(L^+) \\
        &\leq 2d_h(D^-) + (w(D^-) + 2) - r(D^-) + 1\\
        &= s_h(L^-) + 2.
    \end{align*}
\end{proof}

\begin{definition}
\label{def:equiv-x-ch}
    An \textit{equivariant negative crossing change} on a strongly invertible link $L$ is an operation that is either a single negative crossing change on a crossing lying on the axis of $L$, or two negative crossing changes on crossings $x$ and $\tau(x)$ lying off the axis. 
\end{definition}

\begin{proposition}
\label{prop:si-xch}
    Let $L^+, L^-$ be strongly invertible links such that $L^-$ is obtained by applying an equivariant negative crossing change to $L^+$. Then
    \[
        \ds_h(L^-) \leq \ds_h(L^+) \leq \ds_h(L^-) + 2a,
    \]
    where $a = 1$ if the move is performed on-axis, and $a = 2$ if performed off-axis. The same statement holds for $\us_h$.
\end{proposition}

\begin{proof}
    If the move is performed on-axis, the maps $\Phi^\pm$ are strictly $\tau$-equivariant and hence induce maps between the involutive complexes. If the move is performed off-axis, then we may define equivariant crossing change maps
    \[
        \Phi^\pm = \Phi^\pm_2 \Phi^\pm_1
    \]
    where $\Phi^\pm_1$ and $\Phi^\pm_2$ are the non-equivariant crossing change maps corresponding to the off-axis moves. By an argument similar to the proof of \Cref{prop:invariance} for moves IR1 -- IR3, we see that $\Phi^\pm$ is strictly $\tau$-equivariant and hence induce maps between the involutive complexes. 
\end{proof}

\subsection{Behavior under generalized crossing changes}

Next, we extend the results in the previous section to \textit{generalized crossing changes}, introduced by Cochran and Tweedy in \cite{Cochran-Tweedy:2014}.

\begin{definition}
\label{def:gen-x-ch}
    For $n \geq 1$, a \textit{$2n$-strand generalized negative (resp. positive) crossing change} on a link $L$ is a modification of $L$ by adding a \textit{positive} (resp. \textit{negative}) full twist on $2n$ parallel strands of $L$, where $n$ strands are oriented one way and the other are oriented the other (see \Cref{fig:gen-x-ch}). 
\end{definition}

\begin{figure}[t]
    \centering
    \tikzset{every picture/.style={line width=0.75pt}} 

\begin{tikzpicture}[x=0.75pt,y=0.75pt,yscale=-1,xscale=1]

\draw    (20,10) -- (20,99) ;
\draw [shift={(20,59.5)}, rotate = 270] [fill={rgb, 255:red, 0; green, 0; blue, 0 }  ][line width=0.08]  [draw opacity=0] (8.93,-4.29) -- (0,0) -- (8.93,4.29) -- cycle    ;
\draw    (30,10) -- (30,99) ;
\draw [shift={(30,48)}, rotate = 90] [fill={rgb, 255:red, 0; green, 0; blue, 0 }  ][line width=0.08]  [draw opacity=0] (8.93,-4.29) -- (0,0) -- (8.93,4.29) -- cycle    ;
\draw    (60,10) -- (60,99) ;
\draw [shift={(60,59.5)}, rotate = 270] [fill={rgb, 255:red, 0; green, 0; blue, 0 }  ][line width=0.08]  [draw opacity=0] (8.93,-4.29) -- (0,0) -- (8.93,4.29) -- cycle    ;
\draw    (70,10) -- (70,99) ;
\draw [shift={(70,48)}, rotate = 90] [fill={rgb, 255:red, 0; green, 0; blue, 0 }  ][line width=0.08]  [draw opacity=0] (8.93,-4.29) -- (0,0) -- (8.93,4.29) -- cycle    ;
\draw   (15,106) .. controls (15,110.67) and (17.33,113) .. (22,113) -- (35.5,113) .. controls (42.17,113) and (45.5,115.33) .. (45.5,120) .. controls (45.5,115.33) and (48.83,113) .. (55.5,113)(52.5,113) -- (69,113) .. controls (73.67,113) and (76,110.67) .. (76,106) ;
\draw   (98.5,48.69) -- (109.5,48.69) -- (109.5,42.79) -- (121.72,54.59) -- (109.5,66.39) -- (109.5,60.49) -- (98.5,60.49) -- cycle ;
\draw    (145,43) -- (145,99) ;
\draw [shift={(145,76)}, rotate = 270] [fill={rgb, 255:red, 0; green, 0; blue, 0 }  ][line width=0.08]  [draw opacity=0] (8.93,-4.29) -- (0,0) -- (8.93,4.29) -- cycle    ;
\draw    (155,43) -- (155,99) ;
\draw [shift={(155,64.5)}, rotate = 90] [fill={rgb, 255:red, 0; green, 0; blue, 0 }  ][line width=0.08]  [draw opacity=0] (8.93,-4.29) -- (0,0) -- (8.93,4.29) -- cycle    ;
\draw    (185,43) -- (185,99) ;
\draw [shift={(185,76)}, rotate = 270] [fill={rgb, 255:red, 0; green, 0; blue, 0 }  ][line width=0.08]  [draw opacity=0] (8.93,-4.29) -- (0,0) -- (8.93,4.29) -- cycle    ;
\draw    (195,43) -- (195,99) ;
\draw [shift={(195,64.5)}, rotate = 90] [fill={rgb, 255:red, 0; green, 0; blue, 0 }  ][line width=0.08]  [draw opacity=0] (8.93,-4.29) -- (0,0) -- (8.93,4.29) -- cycle    ;

\draw    (145,11) -- (145,41) ;
\draw    (155,11) -- (155,41) ;
\draw    (185,11) -- (185,41) ;
\draw    (195,11) -- (195,41) ;

\draw  [fill={rgb, 255:red, 255; green, 255; blue, 255 }  ,fill opacity=1 ] (137,27) -- (204,27) -- (204,51) -- (137,51) -- cycle ;

\draw (36,46.9) node [anchor=north west][inner sep=0.75pt]    {$\cdots $};
\draw (38,121.4) node [anchor=north west][inner sep=0.75pt]    {$2n$};
\draw (161,63.4) node [anchor=north west][inner sep=0.75pt]    {$\cdots $};
\draw (160,33.4) node [anchor=north west][inner sep=0.75pt]    {$+1$};

\end{tikzpicture}
    \caption{A $2n$-strand generalized negative crossing change}
    \label{fig:gen-x-ch}
    \vspace{2em}
    \resizebox{\textwidth}{!}{
        \centering
        \tikzset{
    every picture/.style={line width=0.75pt},
    every node/.style={font=\large}
}

\begin{tikzpicture}[x=1pt,y=1pt,yscale=-1,xscale=1]

\draw    (50,71.5) -- (50,101) ;
\draw [shift={(50,91.25)}, rotate = 270] [fill={rgb, 255:red, 0; green, 0; blue, 0 }  ][line width=0.08]  [draw opacity=0] (8.93,-4.29) -- (0,0) -- (8.93,4.29) -- cycle    ;
\draw    (65,71.5) -- (65,101) ;
\draw [shift={(65,79.75)}, rotate = 90] [fill={rgb, 255:red, 0; green, 0; blue, 0 }  ][line width=0.08]  [draw opacity=0] (8.93,-4.29) -- (0,0) -- (8.93,4.29) -- cycle    ;
\draw    (95,71.5) -- (95,101) ;
\draw [shift={(95,91.25)}, rotate = 270] [fill={rgb, 255:red, 0; green, 0; blue, 0 }  ][line width=0.08]  [draw opacity=0] (8.93,-4.29) -- (0,0) -- (8.93,4.29) -- cycle    ;
\draw    (109,71.5) -- (109,101) ;
\draw [shift={(109,79.75)}, rotate = 90] [fill={rgb, 255:red, 0; green, 0; blue, 0 }  ][line width=0.08]  [draw opacity=0] (8.93,-4.29) -- (0,0) -- (8.93,4.29) -- cycle    ;
\draw    (50,10) -- (50,40) ;
\draw    (65,10) -- (65,40) ;
\draw    (95,10) -- (95,40) ;
\draw    (109,10) -- (109,40) ;
\draw  [fill={rgb, 255:red, 255; green, 255; blue, 255 }  ,fill opacity=1 ] (42,21.5) -- (117,21.5) -- (117,71.5) -- (42,71.5) -- cycle ;
\draw    (291.5,72) -- (291.5,101.5) ;
\draw [shift={(291.5,91.75)}, rotate = 270] [fill={rgb, 255:red, 0; green, 0; blue, 0 }  ][line width=0.08]  [draw opacity=0] (8.93,-4.29) -- (0,0) -- (8.93,4.29) -- cycle    ;
\draw    (306.5,72) -- (306.5,101.5) ;
\draw [shift={(306.5,80.25)}, rotate = 90] [fill={rgb, 255:red, 0; green, 0; blue, 0 }  ][line width=0.08]  [draw opacity=0] (8.93,-4.29) -- (0,0) -- (8.93,4.29) -- cycle    ;
\draw    (336.5,72) -- (336.5,101.5) ;
\draw [shift={(336.5,91.75)}, rotate = 270] [fill={rgb, 255:red, 0; green, 0; blue, 0 }  ][line width=0.08]  [draw opacity=0] (8.93,-4.29) -- (0,0) -- (8.93,4.29) -- cycle    ;
\draw    (350.5,72) -- (350.5,101.5) ;
\draw [shift={(350.5,80.25)}, rotate = 90] [fill={rgb, 255:red, 0; green, 0; blue, 0 }  ][line width=0.08]  [draw opacity=0] (8.93,-4.29) -- (0,0) -- (8.93,4.29) -- cycle    ;

\draw    (291.5,10.5) -- (291.5,40.5) ;
\draw    (306.5,10.5) -- (306.5,40.5) ;
\draw    (336.5,10.5) -- (336.5,40.5) ;
\draw    (350.5,10.5) -- (350.5,40.5) ;
\draw  [fill={rgb, 255:red, 255; green, 255; blue, 255 }  ,fill opacity=1 ] (283.5,22) -- (358.5,22) -- (358.5,72) -- (283.5,72) -- cycle ;
\draw    (150,71) -- (150,100.5) ;
\draw [shift={(150,90.75)}, rotate = 270] [fill={rgb, 255:red, 0; green, 0; blue, 0 }  ][line width=0.08]  [draw opacity=0] (8.93,-4.29) -- (0,0) -- (8.93,4.29) -- cycle    ;
\draw    (164,71) -- (164,100.5) ;
\draw [shift={(164,79.25)}, rotate = 90] [fill={rgb, 255:red, 0; green, 0; blue, 0 }  ][line width=0.08]  [draw opacity=0] (8.93,-4.29) -- (0,0) -- (8.93,4.29) -- cycle    ;
\draw    (194,71) -- (194,100.5) ;
\draw [shift={(194,90.75)}, rotate = 270] [fill={rgb, 255:red, 0; green, 0; blue, 0 }  ][line width=0.08]  [draw opacity=0] (8.93,-4.29) -- (0,0) -- (8.93,4.29) -- cycle    ;
\draw    (209,71) -- (209,100.5) ;
\draw [shift={(209,79.25)}, rotate = 90] [fill={rgb, 255:red, 0; green, 0; blue, 0 }  ][line width=0.08]  [draw opacity=0] (8.93,-4.29) -- (0,0) -- (8.93,4.29) -- cycle    ;

\begin{knot}[
  clip width=2.5pt,
  end tolerance=1pt
]

\strand    (209,71) .. controls (209,58.7) and (182.91,62.09) .. (165.58,48.4) .. controls (148.24,34.71) and (149.83,30.77) .. (150,11) ;
\strand    (194,71) .. controls (194.18,60.52) and (194.5,49) .. (194.5,41) .. controls (194.5,33) and (166,31) .. (165,12) ;
\strand    (164,71) .. controls (164.51,60.93) and (164.5,49.5) .. (165,41) .. controls (165.5,32.5) and (195,30) .. (194.5,12.5) ;
\strand    (150,71) .. controls (150,58.21) and (178.5,59.5) .. (193.22,49.81) .. controls (207.93,40.11) and (209.5,32.3) .. (209.5,12) ;

\end{knot}

\draw    (393,71.5) -- (393,101) ;
\draw [shift={(393,91.25)}, rotate = 270] [fill={rgb, 255:red, 0; green, 0; blue, 0 }  ][line width=0.08]  [draw opacity=0] (8.93,-4.29) -- (0,0) -- (8.93,4.29) -- cycle    ;
\draw    (407,71.5) -- (407,101) ;
\draw [shift={(407,79.75)}, rotate = 90] [fill={rgb, 255:red, 0; green, 0; blue, 0 }  ][line width=0.08]  [draw opacity=0] (8.93,-4.29) -- (0,0) -- (8.93,4.29) -- cycle    ;
\draw    (437,71.5) -- (437,101) ;
\draw [shift={(437,91.25)}, rotate = 270] [fill={rgb, 255:red, 0; green, 0; blue, 0 }  ][line width=0.08]  [draw opacity=0] (8.93,-4.29) -- (0,0) -- (8.93,4.29) -- cycle    ;
\draw    (452,71.5) -- (452,101) ;
\draw [shift={(452,79.75)}, rotate = 90] [fill={rgb, 255:red, 0; green, 0; blue, 0 }  ][line width=0.08]  [draw opacity=0] (8.93,-4.29) -- (0,0) -- (8.93,4.29) -- cycle    ;

\begin{knot}[
  clip width=2.5pt,
  end tolerance=1pt
]

\strand    (393,71.5) .. controls (393,58.71) and (421.5,60) .. (436.22,50.31) .. controls (450.93,40.61) and (452.5,32.8) .. (452.5,12.5) ;
\strand    (407,71.5) .. controls (407.51,61.43) and (407.5,50) .. (408,41.5) .. controls (408.5,33) and (438,30.5) .. (437.5,13) ;
\strand    (437,71.5) .. controls (437.18,61.02) and (437.5,49.5) .. (437.5,41.5) .. controls (437.5,33.5) and (409,31.5) .. (408,12.5) ;
\strand    (452,71.5) .. controls (452,59.2) and (425.91,62.59) .. (408.58,48.9) .. controls (391.24,35.21) and (392.83,31.27) .. (393,11.5) ;

\end{knot}

\draw (63.5,40.31) node [anchor=north west][inner sep=0.75pt]    {$-1/2$};
\draw (12.5,42.9) node [anchor=north west][inner sep=0.75pt]    {$D^+$};
\draw (124.5,41.9) node [anchor=north west][inner sep=0.75pt]    {$=$};
\draw (304,40.81) node [anchor=north west][inner sep=0.75pt]    {$+1/2$};
\draw (250,39.9) node [anchor=north west][inner sep=0.75pt]    {$D^{-}$};
\draw (366,42.9) node [anchor=north west][inner sep=0.75pt]    {$=$};

\end{tikzpicture}
    }
    \caption{Diagrams $D^+, D^-$}
    \label{fig:gxch-half}
\end{figure}

Note that the case $n = 1$ gives the ordinary negative (resp. positive) crossing change. Hereafter we only consider the case $n = 2$. 

\begin{proposition}
\label{prop:gx-ch}
    Let $D^+, D^-$ be diagrams that differ locally as in \Cref{fig:gxch-half}. Then there are homological grading preserving chain maps
    \[
    \begin{tikzcd}
\CKh(D^+) \arrow[r, "\Phi^-", shift left] & \CKh(D^-) \arrow[l, "\Phi^+", shift left]
    \end{tikzcd}
    \]
    such that the Lee classes modulo torsions correspond as
    \[
        \ca(D^+) \xmapsto{\Phi^-} \ca(D^-),\quad
        \ca(D^-) \xmapsto{\Phi^+} h^2\ca(D^+).
    \]
\end{proposition}

The idea of the proof is similar to that of \Cref{prop:x-ch} but we need some preparations. First recall from \cite[Section 4]{BarNatan:2004} the definition of \textit{strong deformation retracts}.

\begin{definition}
    A chain map $r\colon C \rightarrow C'$ between chain complexes (in any additive category) is a \textit{strong deformation retract}\footnote{%
        Conditions (iii) -- (v) are called the \textit{side conditions}. In \cite[Definition 4.3]{BarNatan:2004} only conditions (i) -- (iii) are imposed, but we can easily check that the remaining two conditions also hold for the R1 and the R2 maps. 
    }
    if there is a chain map $i\colon C' \rightarrow C$ (called the \textit{inclusion}) and a homotopy $h$ on $C$, satisfying (i) $ri = 1$, (ii) $ir - 1 = dh + hd$, (iii) $hi = 0$, (iv) $rh = 0$ and (v) $h^2 = 0$. 
\end{definition}

The following lemma is a generalization of \cite[Lemma 4.5]{BarNatan:2004}.

\begin{lemma}
\label{lem:cone-rtr}
    Suppose $X, Y, Z, W$ are chain complexes (in any additive category), and there are maps $f, g, k, l$ (not necessarily chain maps) such that the following square of complexes
    \[
\begin{tikzcd}[row sep=3em, column sep=3em]
X \arrow[r, "f"] \arrow[d, "k"'] & Y \arrow[d, "g"] \\
Z \arrow[r, "l"']                & W               
\end{tikzcd}
    \]
    form a chain complex $C$. Furthermore suppose $Y$ has a strong deformation retract
    $r: Y \rightarrow Y'$ with inclusion $i$ and homotopy $h$. Then the following square of complexes:
    \[
\begin{tikzcd}[row sep=3em, column sep=3em]
X \arrow[r, "rf"] \arrow[d, "k"'] \arrow[rd, "ghf"] & Y' \arrow[d, "gi"] \\
Z \arrow[r, "l"']                & W               
\end{tikzcd}
    \]
    forms a chain complex $C'$ and is a strong deformation retract of $C$. 
\end{lemma}

\begin{proof}
    The differentials of $C$ and $C'$ may be expressed by matrices
    \[
        D = \begin{pmatrix}
            d_X & & & \\
            f & d_Y & & \\
            k & & d_Z & \\
            & g & l & d_W
        \end{pmatrix},
        \
        D' = \begin{pmatrix}
            d_X & & & \\
            rf & d_{Y'} & & \\
            k & & d_Z & \\
            ghf & gi & l & d_W
        \end{pmatrix}
    \]
    and one can see that $D^2 = 0$ implies $(D')^2 = 0$. Furthermore, one can check that 
    \[
        R = \begin{pmatrix}
            1 & & & \\
            & r & & \\
            & & 1 & \\
            & gh & & 1
        \end{pmatrix},\ 
        I = \begin{pmatrix}
            1 & & & \\
            hf & i & & \\
            & & 1 & \\
            & & & 1
        \end{pmatrix},\ 
        H = \begin{pmatrix}
            0 & & & \\
            & h & & \\
            & & 0 & \\
            & & & 0
        \end{pmatrix}
    \]
    gives a strong deformation retract $R: C \rightarrow C'$ with inclusion $I$ and homotopy $H$. 
\end{proof}

\begin{figure}[t]
    \centering
    \begin{subfigure}{0.48\textwidth}
        \centering
        \tikzset{every picture/.style={line width=0.75pt}} 

\begin{tikzpicture}[x=0.75pt,y=0.75pt,yscale=-1,xscale=1]

\draw [color={rgb, 255:red, 160; green, 160; blue, 160 }][line width=1.5]    (42.58,52.9) .. controls (65.66,49.93) and (73.15,33.89) .. (77.05,22.78) ;
\draw [shift={(78,20)}, rotate = 108.43] [color={rgb, 255:red, 160; green, 160; blue, 160 }][line width=1.5]    (14.21,-4.28) .. controls (9.04,-1.82) and (4.3,-0.39) .. (0,0) .. controls (4.3,0.39) and (9.04,1.82) .. (14.21,4.28)   ;
\draw [color={rgb, 255:red, 160; green, 160; blue, 160 }][line width=1.5]    (69.22,53.31) .. controls (46.13,50.33) and (38.64,34.29) .. (34.74,23.18) ;
\draw [shift={(33.79,20.4)}, rotate = 71.57] [color={rgb, 255:red, 160; green, 160; blue, 160 }][line width=1.5]    (14.21,-4.28) .. controls (9.04,-1.82) and (4.3,-0.39) .. (0,0) .. controls (4.3,0.39) and (9.04,1.82) .. (14.21,4.28)   ;

\draw    (26,75.84) -- (26,100) ;
\draw    (40,75.84) -- (40,100) ;
\draw    (70,75.84) -- (70,100) ;
\draw    (85,75.84) -- (85,100) ;
\draw    (40,75.84) .. controls (40.33,68.05) and (26.33,68.05) .. (26,75.84) ;
\draw    (85,75.84) .. controls (85.33,69.69) and (70.33,68.05) .. (70,75.84) ;

\draw    (70.33,56) .. controls (70.33,62) and (68.33,61) .. (63.33,59) .. controls (58.33,57) and (50.92,55.28) .. (41.58,47.9) .. controls (32.23,40.52) and (25.83,30.27) .. (26,10.5) ;
\draw    (70.33,56) .. controls (70.51,45.52) and (70.5,48.5) .. (70.5,40.5) .. controls (70.5,32.5) and (42,30.5) .. (41,11.5) ;
\draw    (40.34,58.24) .. controls (40.47,52.69) and (40.77,43.41) .. (41,39.5) .. controls (41.23,35.59) and (71,28.5) .. (70.5,11) ;
\draw    (40.34,58.24) .. controls (40.67,60.67) and (42.17,60.67) .. (45.67,59.17) .. controls (49.17,57.67) and (61.37,53.47) .. (69.22,48.31) .. controls (77.06,43.14) and (85.5,30.8) .. (85.5,10.5) ;
\draw [color={rgb, 255:red, 208; green, 2; blue, 27 }  ,draw opacity=1 ][line width=1.5]    (36,70.67) -- (41.34,60.24) ;
\draw [color={rgb, 255:red, 208; green, 2; blue, 27 }  ,draw opacity=1 ][line width=1.5]    (76.8,70.29) -- (69.8,60.29) ;

\draw (28,57) node [anchor=north west][inner sep=0.75pt]    {$\textcolor[rgb]{0.82,0.01,0.11}{e}$};
\draw (78,52.9) node [anchor=north west][inner sep=0.75pt]    {$\textcolor[rgb]{0.82,0.01,0.11}{e'}$};
\draw (85,18.9) node [anchor=north west][inner sep=0.75pt]    {$\textcolor[rgb]{0.5,0.5,0.5}{G}$};
\draw (8,18.4) node [anchor=north west][inner sep=0.75pt]    {$\textcolor[rgb]{0.5,0.5,0.5}{G'}$};

\end{tikzpicture}
        \caption{$D^+$}
        \label{subfig:simplify-A}
        \vspace{2em}
    \end{subfigure}
    \begin{subfigure}{0.48\textwidth}
        \centering
        \tikzset{every picture/.style={line width=0.75pt}} 

\begin{tikzpicture}[x=0.75pt,y=0.75pt,yscale=-1,xscale=1]

\draw [color={rgb, 255:red, 160; green, 160; blue, 160 }][line width=1.5]    (42.58,52.9) .. controls (65.66,49.93) and (73.15,33.89) .. (77.05,22.78) ;
\draw [shift={(78,20)}, rotate = 108.43] [color={rgb, 255:red, 160; green, 160; blue, 160 }][line width=1.5]    (14.21,-4.28) .. controls (9.04,-1.82) and (4.3,-0.39) .. (0,0) .. controls (4.3,0.39) and (9.04,1.82) .. (14.21,4.28)   ;
\draw [color={rgb, 255:red, 160; green, 160; blue, 160 }][line width=1.5]    (69.22,53.31) .. controls (46.13,50.33) and (38.64,34.29) .. (34.74,23.18) ;
\draw [shift={(33.79,20.4)}, rotate = 71.57] [color={rgb, 255:red, 160; green, 160; blue, 160 }][line width=1.5]    (14.21,-4.28) .. controls (9.04,-1.82) and (4.3,-0.39) .. (0,0) .. controls (4.3,0.39) and (9.04,1.82) .. (14.21,4.28)   ;

\draw    (26,75.84) -- (26,100) ;
\draw    (40,75.84) -- (40,100) ;
\draw    (70,75.84) -- (70,100) ;
\draw    (85,75.84) -- (85,100) ;
\draw    (40.34,58.24) .. controls (40.02,61.44) and (26.33,68.05) .. (26,75.84) ;
\draw    (85,75.84) .. controls (85.33,69.69) and (72.02,63.44) .. (70.33,56) ;
\draw    (70,75.84) .. controls (70.52,69.94) and (68.33,61) .. (63.33,59) .. controls (58.33,57) and (50.92,55.28) .. (41.58,47.9) .. controls (32.23,40.52) and (25.83,30.27) .. (26,10.5) ;
\draw    (70.33,56) .. controls (70.51,45.52) and (70.5,48.5) .. (70.5,40.5) .. controls (70.5,32.5) and (42,30.5) .. (41,11.5) ;
\draw    (40.34,58.24) .. controls (40.47,52.69) and (40.77,43.41) .. (41,39.5) .. controls (41.23,35.59) and (71,28.5) .. (70.5,11) ;
\draw    (40,75.84) .. controls (40.02,70.94) and (39.87,68.71) .. (44.52,64.44) .. controls (49.17,60.17) and (61.37,53.47) .. (69.22,48.31) .. controls (77.06,43.14) and (85.5,30.8) .. (85.5,10.5) ;
\draw [color={rgb, 255:red, 208; green, 2; blue, 27 }  ,draw opacity=1 ][line width=1.5]    (35.5,62.67) -- (41.52,67.94) ;
\draw [color={rgb, 255:red, 208; green, 2; blue, 27 }  ,draw opacity=1 ][line width=1.5]    (75.52,63.94) -- (69.3,68.29) ;

\draw (27.5,54) node [anchor=north west][inner sep=0.75pt]    {$\textcolor[rgb]{0.82,0.01,0.11}{e}$};
\draw (79,50.4) node [anchor=north west][inner sep=0.75pt]    {$\textcolor[rgb]{0.82,0.01,0.11}{e'}$};
\draw (85,18.9) node [anchor=north west][inner sep=0.75pt]    {$\textcolor[rgb]{0.5,0.5,0.5}{G}$};
\draw (8,18.4) node [anchor=north west][inner sep=0.75pt]    {$\textcolor[rgb]{0.5,0.5,0.5}{G'}$};

\end{tikzpicture}
        \caption{$D^-$}
        \label{subfig:simplify-B}
        \vspace{2em}
    \end{subfigure}

    \begin{subfigure}{0.48\textwidth}
        \centering
        \tikzset{every picture/.style={line width=0.75pt}} 

\begin{tikzpicture}[x=0.75pt,y=0.75pt,yscale=-1,xscale=1]

\draw    (78.75,11.39) .. controls (79,23.39) and (70,23.89) .. (70,11.64) ;
\draw    (42.01,49.61) .. controls (41.76,37.61) and (50.76,37.11) .. (50.76,49.36) ;
\draw    (42.01,11.17) .. controls (41.76,23.17) and (50.76,23.67) .. (50.76,11.42) ;
\draw    (78.75,49.39) .. controls (79,37.39) and (70,36.89) .. (70,49.14) ;

\draw    (197.5,11.13) -- (168.5,49.98) ;
\draw    (188.5,10.8) -- (159.5,49.65) ;
\draw    (160.25,11.23) .. controls (160,23.23) and (169,23.73) .. (169,11.48) ;
\draw    (196.99,49.45) .. controls (197.24,37.45) and (188.24,36.95) .. (188.24,49.2) ;

\draw    (41.5,121.03) -- (70.5,159.88) ;
\draw    (50.5,120.7) -- (79.5,159.55) ;
\draw    (78.75,121.13) .. controls (79,133.13) and (70,133.63) .. (70,121.38) ;
\draw    (42.01,159.35) .. controls (41.76,147.35) and (50.76,146.85) .. (50.76,159.1) ;

\draw    (197.5,121.53) -- (168.5,160.38) ;
\draw    (188.5,121.2) -- (159.5,160.05) ;
\draw  [color={rgb, 255:red, 255; green, 255; blue, 255 }  ,draw opacity=1 ][fill={rgb, 255:red, 255; green, 255; blue, 255 }  ,fill opacity=1 ] (176.5,133.38) .. controls (176.5,131.72) and (177.84,130.38) .. (179.5,130.38) .. controls (181.16,130.38) and (182.5,131.72) .. (182.5,133.38) .. controls (182.5,135.03) and (181.16,136.38) .. (179.5,136.38) .. controls (177.84,136.38) and (176.5,135.03) .. (176.5,133.38) -- cycle ;
\draw  [color={rgb, 255:red, 255; green, 255; blue, 255 }  ,draw opacity=1 ][fill={rgb, 255:red, 255; green, 255; blue, 255 }  ,fill opacity=1 ] (171,140.63) .. controls (171,138.97) and (172.34,137.63) .. (174,137.63) .. controls (175.66,137.63) and (177,138.97) .. (177,140.63) .. controls (177,142.28) and (175.66,143.63) .. (174,143.63) .. controls (172.34,143.63) and (171,142.28) .. (171,140.63) -- cycle ;
\draw  [color={rgb, 255:red, 255; green, 255; blue, 255 }  ,draw opacity=1 ][fill={rgb, 255:red, 255; green, 255; blue, 255 }  ,fill opacity=1 ] (181,140.3) .. controls (181,138.64) and (182.34,137.3) .. (184,137.3) .. controls (185.66,137.3) and (187,138.64) .. (187,140.3) .. controls (187,141.96) and (185.66,143.3) .. (184,143.3) .. controls (182.34,143.3) and (181,141.96) .. (181,140.3) -- cycle ;
\draw  [color={rgb, 255:red, 255; green, 255; blue, 255 }  ,draw opacity=1 ][fill={rgb, 255:red, 255; green, 255; blue, 255 }  ,fill opacity=1 ] (176,146.38) .. controls (176,144.72) and (177.34,143.38) .. (179,143.38) .. controls (180.66,143.38) and (182,144.72) .. (182,146.38) .. controls (182,148.03) and (180.66,149.38) .. (179,149.38) .. controls (177.34,149.38) and (176,148.03) .. (176,146.38) -- cycle ;
\draw    (160,120.53) -- (189,159.38) ;
\draw    (169,120.2) -- (198,159.05) ;
\draw    (93.5,41) -- (144.5,41) ;
\draw [shift={(146.5,41)}, rotate = 180] [color={rgb, 255:red, 0; green, 0; blue, 0 }  ][line width=0.75]    (10.93,-4.9) .. controls (6.95,-2.3) and (3.31,-0.67) .. (0,0) .. controls (3.31,0.67) and (6.95,2.3) .. (10.93,4.9)   ;
\draw [color={rgb, 255:red, 128; green, 128; blue, 128 }  ,draw opacity=1 ]   (93.5,141) -- (144.5,141) ;
\draw [shift={(146.5,141)}, rotate = 180] [color={rgb, 255:red, 128; green, 128; blue, 128 }  ,draw opacity=1 ][line width=0.75]    (10.93,-4.9) .. controls (6.95,-2.3) and (3.31,-0.67) .. (0,0) .. controls (3.31,0.67) and (6.95,2.3) .. (10.93,4.9)   ;
\draw    (61.5,58) -- (61.5,107) ;
\draw [shift={(61.5,109)}, rotate = 270] [color={rgb, 255:red, 0; green, 0; blue, 0 }  ][line width=0.75]    (10.93,-4.9) .. controls (6.95,-2.3) and (3.31,-0.67) .. (0,0) .. controls (3.31,0.67) and (6.95,2.3) .. (10.93,4.9)   ;
\draw [color={rgb, 255:red, 128; green, 128; blue, 128 }  ,draw opacity=1 ]   (186,58) -- (186,105.75) -- (186,107) ;
\draw [shift={(186,109)}, rotate = 270] [color={rgb, 255:red, 128; green, 128; blue, 128 }  ,draw opacity=1 ][line width=0.75]    (10.93,-4.9) .. controls (6.95,-2.3) and (3.31,-0.67) .. (0,0) .. controls (3.31,0.67) and (6.95,2.3) .. (10.93,4.9)   ;
\draw    (128.91,11.31) .. controls (129.07,18.98) and (123.32,19.3) .. (123.32,11.47) ;
\draw    (105.42,35.75) .. controls (105.26,28.08) and (111.01,27.76) .. (111.01,35.59) ;
\draw    (105.42,11.17) .. controls (105.26,18.84) and (111.01,19.16) .. (111.01,11.33) ;
\draw    (128.91,35.61) .. controls (129.07,27.94) and (123.32,27.62) .. (123.32,35.45) ;

\draw [color={rgb, 255:red, 208; green, 2; blue, 27 }  ,draw opacity=1 ][line width=1.5]    (110.21,30.88) -- (125.11,15.97) ;

\draw    (54.43,67.8) .. controls (54.59,75.17) and (49.07,75.47) .. (49.07,67.96) ;
\draw    (31.89,91.25) .. controls (31.74,83.89) and (37.26,83.58) .. (37.26,91.1) ;
\draw    (31.89,67.67) .. controls (31.74,75.03) and (37.26,75.34) .. (37.26,67.82) ;
\draw    (54.43,91.11) .. controls (54.59,83.75) and (49.07,83.45) .. (49.07,90.96) ;

\draw [color={rgb, 255:red, 208; green, 2; blue, 27 }  ,draw opacity=1 ][line width=1.5]    (35.33,72.85) -- (49.22,86.73) ;

\draw    (86.5,57) -- (150.34,120.84) ;
\draw [shift={(151.75,122.25)}, rotate = 225] [color={rgb, 255:red, 0; green, 0; blue, 0 }  ][line width=0.75]    (10.93,-4.9) .. controls (6.95,-2.3) and (3.31,-0.67) .. (0,0) .. controls (3.31,0.67) and (6.95,2.3) .. (10.93,4.9)   ;
\draw  [color={rgb, 255:red, 255; green, 255; blue, 255 }  ,draw opacity=1 ][fill={rgb, 255:red, 255; green, 255; blue, 255 }  ,fill opacity=1 ] (90.38,70) -- (156.62,70) -- (156.62,100.75) -- (90.38,100.75) -- cycle ;
\draw    (115.61,72.87) .. controls (115.77,80.54) and (110.01,80.86) .. (110.01,73.03) ;
\draw    (92.11,97.31) .. controls (91.95,89.64) and (97.7,89.32) .. (97.7,97.15) ;
\draw    (92.11,72.73) .. controls (91.95,80.4) and (97.7,80.72) .. (97.7,72.89) ;
\draw    (115.61,97.17) .. controls (115.77,89.49) and (110.01,89.17) .. (110.01,97.01) ;

\draw    (153.98,72.87) .. controls (154.14,80.54) and (148.38,80.86) .. (148.38,73.03) ;
\draw    (130.48,97.31) .. controls (130.32,89.64) and (136.07,89.32) .. (136.07,97.15) ;
\draw    (130.48,72.73) .. controls (130.32,80.4) and (136.07,80.72) .. (136.07,72.89) ;
\draw    (153.98,97.17) .. controls (154.14,89.49) and (148.38,89.17) .. (148.38,97.01) ;

\draw [color={rgb, 255:red, 208; green, 2; blue, 27 }  ,draw opacity=1 ][line width=1.5]    (94.98,78.37) .. controls (101.7,83.97) and (108.41,82.69) .. (112.57,78.85) ;
\draw [color={rgb, 255:red, 208; green, 2; blue, 27 }  ,draw opacity=1 ][line width=1.5]    (133.35,91.16) .. controls (141.03,86.53) and (144.23,86.21) .. (150.94,91.64) ;

\draw (116.58,75.35) node [anchor=north west][inner sep=0.75pt]  [font=\large]  {$+$};

\end{tikzpicture}
        \caption{$E^+$}
        \label{subfig:simplify-A-cpx}
        \vspace{1em}
    \end{subfigure}
    \begin{subfigure}{0.48\textwidth}
        \centering
        \tikzset{every picture/.style={line width=0.75pt}} 

\begin{tikzpicture}[x=0.75pt,y=0.75pt,yscale=-1,xscale=1]

\draw    (198.25,120.39) .. controls (198.5,132.39) and (189.5,132.89) .. (189.5,120.64) ;
\draw    (161.51,158.61) .. controls (161.26,146.61) and (170.26,146.11) .. (170.26,158.36) ;
\draw    (161.51,120.17) .. controls (161.26,132.17) and (170.26,132.67) .. (170.26,120.42) ;
\draw    (198.25,158.39) .. controls (198.5,146.39) and (189.5,145.89) .. (189.5,158.14) ;

\draw    (80,120.63) -- (51,159.48) ;
\draw    (71,120.3) -- (42,159.15) ;
\draw    (42.75,120.73) .. controls (42.5,132.73) and (51.5,133.23) .. (51.5,120.98) ;
\draw    (79.49,158.95) .. controls (79.74,146.95) and (70.74,146.45) .. (70.74,158.7) ;

\draw    (161.5,11.53) -- (190.5,50.38) ;
\draw    (170.5,11.2) -- (199.5,50.05) ;
\draw    (198.75,11.63) .. controls (199,23.63) and (190,24.13) .. (190,11.88) ;
\draw    (162.01,49.85) .. controls (161.76,37.85) and (170.76,37.35) .. (170.76,49.6) ;

\draw    (78.5,11.53) -- (49.5,50.38) ;
\draw    (69.5,11.2) -- (40.5,50.05) ;
\draw  [color={rgb, 255:red, 255; green, 255; blue, 255 }  ,draw opacity=1 ][fill={rgb, 255:red, 255; green, 255; blue, 255 }  ,fill opacity=1 ] (57.5,23.38) .. controls (57.5,21.72) and (58.84,20.38) .. (60.5,20.38) .. controls (62.16,20.38) and (63.5,21.72) .. (63.5,23.38) .. controls (63.5,25.03) and (62.16,26.38) .. (60.5,26.38) .. controls (58.84,26.38) and (57.5,25.03) .. (57.5,23.38) -- cycle ;
\draw  [color={rgb, 255:red, 255; green, 255; blue, 255 }  ,draw opacity=1 ][fill={rgb, 255:red, 255; green, 255; blue, 255 }  ,fill opacity=1 ] (52,30.63) .. controls (52,28.97) and (53.34,27.63) .. (55,27.63) .. controls (56.66,27.63) and (58,28.97) .. (58,30.63) .. controls (58,32.28) and (56.66,33.63) .. (55,33.63) .. controls (53.34,33.63) and (52,32.28) .. (52,30.63) -- cycle ;
\draw  [color={rgb, 255:red, 255; green, 255; blue, 255 }  ,draw opacity=1 ][fill={rgb, 255:red, 255; green, 255; blue, 255 }  ,fill opacity=1 ] (62,30.3) .. controls (62,28.64) and (63.34,27.3) .. (65,27.3) .. controls (66.66,27.3) and (68,28.64) .. (68,30.3) .. controls (68,31.96) and (66.66,33.3) .. (65,33.3) .. controls (63.34,33.3) and (62,31.96) .. (62,30.3) -- cycle ;
\draw  [color={rgb, 255:red, 255; green, 255; blue, 255 }  ,draw opacity=1 ][fill={rgb, 255:red, 255; green, 255; blue, 255 }  ,fill opacity=1 ] (57,36.38) .. controls (57,34.72) and (58.34,33.38) .. (60,33.38) .. controls (61.66,33.38) and (63,34.72) .. (63,36.38) .. controls (63,38.03) and (61.66,39.38) .. (60,39.38) .. controls (58.34,39.38) and (57,38.03) .. (57,36.38) -- cycle ;
\draw    (41,10.53) -- (70,49.38) ;
\draw    (50,10.2) -- (79,49.05) ;
\draw [color={rgb, 255:red, 128; green, 128; blue, 128 }  ,draw opacity=1 ]   (93.5,41) -- (144.5,41) ;
\draw [shift={(146.5,41)}, rotate = 180] [color={rgb, 255:red, 128; green, 128; blue, 128 }  ,draw opacity=1 ][line width=0.75]    (10.93,-4.9) .. controls (6.95,-2.3) and (3.31,-0.67) .. (0,0) .. controls (3.31,0.67) and (6.95,2.3) .. (10.93,4.9)   ;
\draw [color={rgb, 255:red, 0; green, 0; blue, 0 }  ,draw opacity=1 ]   (95.5,133.5) -- (146.5,133.5) ;
\draw [shift={(148.5,133.5)}, rotate = 180] [color={rgb, 255:red, 0; green, 0; blue, 0 }  ,draw opacity=1 ][line width=0.75]    (10.93,-4.9) .. controls (6.95,-2.3) and (3.31,-0.67) .. (0,0) .. controls (3.31,0.67) and (6.95,2.3) .. (10.93,4.9)   ;
\draw [color={rgb, 255:red, 128; green, 128; blue, 128 }  ,draw opacity=1 ]   (61.5,58) -- (61.5,107) ;
\draw [shift={(61.5,109)}, rotate = 270] [color={rgb, 255:red, 128; green, 128; blue, 128 }  ,draw opacity=1 ][line width=0.75]    (10.93,-4.9) .. controls (6.95,-2.3) and (3.31,-0.67) .. (0,0) .. controls (3.31,0.67) and (6.95,2.3) .. (10.93,4.9)   ;
\draw [color={rgb, 255:red, 0; green, 0; blue, 0 }  ,draw opacity=1 ]   (172,58) -- (172,105.75) -- (172,107) ;
\draw [shift={(172,109)}, rotate = 270] [color={rgb, 255:red, 0; green, 0; blue, 0 }  ,draw opacity=1 ][line width=0.75]    (10.93,-4.9) .. controls (6.95,-2.3) and (3.31,-0.67) .. (0,0) .. controls (3.31,0.67) and (6.95,2.3) .. (10.93,4.9)   ;
\draw [color={rgb, 255:red, 208; green, 2; blue, 27 }  ,draw opacity=1 ][line width=1.5]    (189.29,79.81) -- (192.5,76.6) ;
\draw [color={rgb, 255:red, 208; green, 2; blue, 27 }  ,draw opacity=1 ][line width=1.5]    (118.38,150.19) -- (120.94,152.75) ;
\draw    (86.5,57) -- (150.34,120.84) ;
\draw [shift={(151.75,122.25)}, rotate = 225] [color={rgb, 255:red, 0; green, 0; blue, 0 }  ][line width=0.75]    (10.93,-4.9) .. controls (6.95,-2.3) and (3.31,-0.67) .. (0,0) .. controls (3.31,0.67) and (6.95,2.3) .. (10.93,4.9)   ;
\draw  [color={rgb, 255:red, 255; green, 255; blue, 255 }  ,draw opacity=1 ][fill={rgb, 255:red, 255; green, 255; blue, 255 }  ,fill opacity=1 ] (88.38,71.5) -- (154.62,71.5) -- (154.62,102.25) -- (88.38,102.25) -- cycle ;
\draw    (113.61,74.37) .. controls (113.77,82.04) and (108.01,82.36) .. (108.01,74.53) ;
\draw    (90.11,98.81) .. controls (93,87.25) and (111,87.75) .. (113.61,98.67) ;
\draw    (90.11,74.23) .. controls (89.95,81.9) and (95.7,82.22) .. (95.7,74.39) ;
\draw    (108.5,98.75) .. controls (106.5,94.75) and (99.5,93.25) .. (95.51,98.51) ;
\draw    (146.5,75.25) .. controls (147.5,81.75) and (134.5,82.25) .. (133.88,74.53) ;
\draw    (128.48,98.81) .. controls (128.32,91.14) and (134.07,90.82) .. (134.07,98.65) ;
\draw    (128.48,74.23) .. controls (128.5,86.75) and (152,88.25) .. (151.98,74.37) ;
\draw    (151.98,98.67) .. controls (152.14,90.99) and (146.38,90.67) .. (146.38,98.51) ;
\draw [color={rgb, 255:red, 208; green, 2; blue, 27 }  ,draw opacity=1 ][line width=1.5]    (102,91) -- (102,95.25) ;
\draw [color={rgb, 255:red, 208; green, 2; blue, 27 }  ,draw opacity=1 ][line width=1.5]    (140.5,79.5) -- (140.5,83.75) ;
\draw    (181,68.38) -- (196.58,89.25) ;
\draw    (185.84,68.2) -- (201.42,89.07) ;
\draw    (201.01,68.43) .. controls (201.15,74.88) and (196.31,75.15) .. (196.31,68.56) ;
\draw    (181.27,88.97) .. controls (181.14,82.52) and (185.97,82.25) .. (185.97,88.83) ;

\draw    (129.35,141.47) -- (114.58,161.25) ;
\draw    (124.77,141.3) -- (110,161.08) ;
\draw    (110.38,141.52) .. controls (110.25,147.63) and (114.84,147.88) .. (114.84,141.65) ;
\draw    (129.09,160.98) .. controls (129.22,154.87) and (124.64,154.62) .. (124.64,160.85) ;

\draw (114.58,74.85) node [anchor=north west][inner sep=0.75pt]  [font=\large]  {$+$};

\end{tikzpicture}
        \caption{$E^-$}
        \label{subfig:simplify-B-cpx}
        \vspace{1em}
    \end{subfigure}
    \caption{Simplifying $\CKh(D^+)$ and $\CKh(D^-)$}
    \label{fig:cone-simplify}
\end{figure}

\begin{lemma}
\label{lem:simplify-D}
    Let $D^+, D^-$ be diagrams that differ locally as in \Cref{fig:gxch-half}. The complex $\CKh(D^+)$ strongly deformation retracts onto the complex $E$ described in \Cref{subfig:simplify-A-cpx}. Similarly, the complex $\CKh(D^-)$ strongly deformation retracts onto the complex $E^-$ described in \Cref{subfig:simplify-B-cpx}. (Descriptions for some arrows are omitted, since we will not use them.)
\end{lemma}

\begin{proof}
    By considering the $0$ and $1$ resolutions of the bottom left and the bottom right crossings of $D^+$, we see that the complex $\CKh(D^+)$ can be expressed as a square of complexes
    \[  
\begin{tikzcd}
\CKh(D^+_{00}) \arrow[r, "e"] \arrow[d, "e'"] & \CKh(D^+_{10}) \arrow[d, "e'"] \\
\CKh(D^+_{01}) \arrow[r, "e"] & \CKh(D^+_{11})
\end{tikzcd}
    \]
    where $e, e'$ denotes the surgery map corresponding to the bottom left and the bottom right crossings respectively. Observe that the resolved diagrams $D^+_{00}$, $D^+_{10}$, $D^+_{01}$ can be simplified by performing R2 moves (see \Cref{subfig:simplify-A}). 
    Since the map $G$ for the R2-move defined in \cite[Section 4.2]{BarNatan:2004} is a strong deformation retract, we may apply \Cref{lem:cone-rtr} repeatedly and obtain a strong deformation retract $E^+$ of $\CKh(D^+)$ of the form
    \begin{center}
\begin{tikzcd}[row sep=4em, column sep=4em]
E^+_{00} \arrow[r, "G'eF"] \arrow[d, "e'"] \arrow[rd, "e'H'eF"] & E^+_{10} \arrow[d, "e'F'"] \\
E^+_{01} \arrow[r, "eF"] & E^+_{11}
\end{tikzcd}
    \end{center}
    Here, $G, G'$ are the maps for the R2-moves, $F, F'$ and $H, H'$ are the corresponding inclusions and homotopies. By unraveling the explicit maps, one can check that the maps are given as in \Cref{subfig:simplify-A-cpx}. Similarly $\CKh(D^-)$ can be expressed as a square of complexes
    \[  
\begin{tikzcd}
\CKh(D^-_{00}) \arrow[r, "e"] \arrow[d, "e'"] & \CKh(D^-_{10}) \arrow[d, "e'"] \\
\CKh(D^-_{01}) \arrow[r, "e"] & \CKh(D^-_{11})
\end{tikzcd}
    \]
    and the resolved diagrams $D^-_{10}$, $D^-_{01}$, $D^-_{11}$ can be simplified by performing R2 moves. Thus we obtain a strong deformation retract $E^-$ of $\CKh(D^-)$ of the form
    \begin{center}
\begin{tikzcd}[row sep=4em, column sep=4em]
E^-_{00} \arrow[r, "Ge"] \arrow[d, "G'e'"] \arrow[rd, "GeH'e'"] & E^-_{10} \arrow[d, "e'"] \\
E^-_{01} \arrow[r, "GeF'"] & E^-_{11}
\end{tikzcd}
    \end{center}
    and again one can check that the maps are given as in \Cref{subfig:simplify-B-cpx}.
\end{proof}

\begin{proof}[Proof of \Cref{prop:gx-ch}]
    First we define chain maps
    \[
\begin{tikzcd}
E^+ \arrow[r, "\phi^-", shift left] & E^- \arrow[l, "\phi^+", shift left]
\end{tikzcd}
    \]
    between the two simplified complexes $E^+, E^-$ of \Cref{lem:simplify-D}. For $\phi^-$, since $E^+_{00}$ and $E^-_{11}$ are identical, we may define $\phi^- = \id$ on $E^+_{00}$ and $0$ elsewhere. Obviously this is a chain map, since we have $\phi^-d^+ = 0 = d^-\phi^-$. Next, $\phi^+$ is defined by the sum of the following four maps $\phi^+_1, \phi^+_2, \phi^+_3, \phi^+_4$:
\[
\begin{tikzcd}[row sep=2em, column sep=2em]
 & & \underline{E^+_{00}} \arrow[r] \arrow[rd] \arrow[rrd] & E^+_{10} \arrow[rd] & \\
 & & & E^+_{01} \arrow[r] & E^+_{11} \\
E^-_{00} \arrow[rd] \arrow[r] \arrow[rrd] \arrow[rrrru, "\phi^+_4", pos=.4, dashed] \arrow[rruu, "\phi^+_3", dashed] & E^-_{10} \arrow[rd] & & & \\
 & E^-_{01} \arrow[r] & \underline{E^-_{11}} \arrow[uuu, "\phi^+_1"', pos=.4, dashed] \arrow[rruu, "\phi^+_2", dashed] &                   &       
\end{tikzcd}
\]
    with explicit descriptions:
    \begin{center}
        \vspace{1em}
        \input{tikzpictures/gxch-map}
        \vspace{1em}
    \end{center}
    Note that $E^+_{00}$ and $E^-_{11}$ only have homological grading $0$ (within the displayed area), whereas $E^+_{11}$ has homological grading in range $[0, 4]$ and $E^-_{00}$ has homological grading in range $[-4, 0]$. One can see that each $\phi^+_i$ has domain and codomain in the homological grading $0$ parts. 
    
    \begin{figure}[t]
        \centering
        \begin{tabular}{ccc}
\begin{tikzcd}
E^-_{10}(-1) \arrow[r, "d"] \arrow[d] & E^-_{11}(0) \arrow[d, "\phi^+_1"] \\
0 \arrow[r] & E^+_{00}(0)                   
\end{tikzcd}
        & 
\begin{tikzcd}
E^-_{01}(-1) \arrow[r, "d"] \arrow[d] & E^-_{11}(0) \arrow[d, "\phi^+_1"] \\
0 \arrow[r] & E^+_{00}(0)                   
\end{tikzcd}
        &
\begin{tikzcd}
E^-_{00}(-1) \arrow[r, "d"] \arrow[d, "d"] & E^-_{11}(0) \arrow[d, "\phi^+_1"] \\
E^-_{00}(0) \arrow[r, "\phi^+_3"] & E^+_{00}(0)                   
\end{tikzcd}
        \vspace{1em}
        \\
\begin{tikzcd}
E^-_{00}(-1) \arrow[r, "d"] \arrow[d, "d"] & E^-_{11}(0) \arrow[d, "\phi^+_2"] \\
E^-_{00}(0) \arrow[r, "\phi^+_4"] & E^+_{11}(0)                   
\end{tikzcd}
        &
\begin{tikzcd}
E^-_{00}(0) \arrow[r, "\phi^+_3"] \arrow[d] & E^+_{00}(0) \arrow[d, "d"] \\
0 \arrow[r] & E^+_{10}(1)                   
\end{tikzcd} 
        &
\begin{tikzcd}
E^-_{00}(0) \arrow[r, "\phi^+_3"] \arrow[d] & E^+_{00}(0) \arrow[d, "d"] \\
0 \arrow[r] & E^+_{01}(1)                   
\end{tikzcd} 
        \vspace{1em}
        \\
\begin{tikzcd}
E^-_{00}(0) \arrow[r, "\phi^+_3"] \arrow[d, "\phi^+_4"] & E^+_{00}(0) \arrow[d, "d"] \\
E^+_{11}(0) \arrow[r, "d"] & E^+_{11}(1)                   
\end{tikzcd} 
        &
\begin{tikzcd}
E^-_{11}(0) \arrow[r, "\phi^+_1"] \arrow[d] & E^+_{00}(0) \arrow[d, "d"] \\
0 \arrow[r] & E^+_{10}(1)                   
\end{tikzcd} 
        &
\begin{tikzcd}
E^-_{11}(0) \arrow[r, "\phi^+_1"] \arrow[d] & E^+_{00}(0) \arrow[d, "d"] \\
0 \arrow[r] & E^+_{01}(1)                   
\end{tikzcd} 
        \vspace{1em}
        \\
\begin{tikzcd}
E^-_{11}(0) \arrow[r, "\phi^+_1"] \arrow[d, "\phi^+_2"] & E^+_{00}(0) \arrow[d, "d"] \\
E^+_{11}(0) \arrow[r, "d"] & E^+_{11}(1)                   
\end{tikzcd} 
    \end{tabular}
    \caption{}
    \label{fig:phi-comm}
\end{figure}

    To show that $\phi^+ = \sum_i \phi^+_i$ actually defines a chain map, it suffices to verify that all the diagrams displayed in \Cref{fig:phi-comm} commute. Here each number in the parenthesis indicates the homological grading. This can be checked by the explicit descriptions of the maps $\phi^+_i$ and the differentials of $E^+, E^-$ given in \Cref{fig:cone-simplify}. For example, the commutativity of the third diagram is equivalent to
    \begin{center}
        \vspace{1em}
        \tikzset{every picture/.style={line width=0.75pt}} 

\begin{tikzpicture}[x=0.75pt,y=0.75pt,yscale=-1,xscale=1]

\draw    (39.43,10.88) .. controls (39.62,19.92) and (32.84,20.3) .. (32.84,11.07) ;
\draw    (11.76,39.67) .. controls (15.17,26.06) and (36.36,26.64) .. (39.43,39.5) ;
\draw    (11.76,10.72) .. controls (11.57,19.76) and (18.35,20.13) .. (18.35,10.91) ;
\draw    (33.42,39.6) .. controls (31.06,34.89) and (22.82,33.12) .. (18.12,39.31) ;
\draw [color={rgb, 255:red, 208; green, 2; blue, 27 }  ,draw opacity=1 ][line width=1.5]    (25.76,29.77) -- (25.76,34.77) ;

\draw    (100.43,9.88) .. controls (100.62,18.92) and (93.84,19.3) .. (93.84,10.07) ;
\draw    (72.76,38.67) .. controls (76.17,25.06) and (97.36,25.64) .. (100.43,38.5) ;
\draw    (72.76,9.72) .. controls (72.57,18.76) and (79.35,19.13) .. (79.35,9.91) ;
\draw    (94.42,38.6) .. controls (92.06,33.89) and (83.82,32.12) .. (79.12,38.31) ;
\draw [color={rgb, 255:red, 208; green, 2; blue, 27 }  ,draw opacity=1 ][line width=1.5]    (86.76,28.77) -- (86.76,33.77) ;

\draw    (158.43,10.88) .. controls (158.62,19.92) and (151.84,20.3) .. (151.84,11.07) ;
\draw    (130.76,39.67) .. controls (134.17,26.06) and (155.36,26.64) .. (158.43,39.5) ;
\draw    (130.76,10.72) .. controls (130.57,19.76) and (137.35,20.13) .. (137.35,10.91) ;
\draw    (152.42,39.6) .. controls (150.06,34.89) and (141.82,33.12) .. (137.12,39.31) ;
\draw [color={rgb, 255:red, 208; green, 2; blue, 27 }  ,draw opacity=1 ][line width=1.5]    (144.76,29.77) -- (144.76,34.77) ;

\draw    (217.43,10.88) .. controls (217.62,19.92) and (210.84,20.3) .. (210.84,11.07) ;
\draw    (189.76,39.67) .. controls (193.17,26.06) and (214.36,26.64) .. (217.43,39.5) ;
\draw    (189.76,10.72) .. controls (189.57,19.76) and (196.35,20.13) .. (196.35,10.91) ;
\draw    (211.42,39.6) .. controls (209.06,34.89) and (200.82,33.12) .. (196.12,39.31) ;
\draw [color={rgb, 255:red, 208; green, 2; blue, 27 }  ,draw opacity=1 ][line width=1.5]    (203.76,29.77) -- (203.76,34.77) ;

\draw  [fill={rgb, 255:red, 0; green, 0; blue, 0 }  ,fill opacity=1 ] (16,30.63) .. controls (16,29.31) and (17.06,28.25) .. (18.38,28.25) .. controls (19.69,28.25) and (20.75,29.31) .. (20.75,30.63) .. controls (20.75,31.94) and (19.69,33) .. (18.38,33) .. controls (17.06,33) and (16,31.94) .. (16,30.63) -- cycle ;
\draw  [fill={rgb, 255:red, 0; green, 0; blue, 0 }  ,fill opacity=1 ] (31.25,30.63) .. controls (31.25,29.31) and (32.31,28.25) .. (33.63,28.25) .. controls (34.94,28.25) and (36,29.31) .. (36,30.63) .. controls (36,31.94) and (34.94,33) .. (33.63,33) .. controls (32.31,33) and (31.25,31.94) .. (31.25,30.63) -- cycle ;
\draw  [fill={rgb, 255:red, 0; green, 0; blue, 0 }  ,fill opacity=1 ] (91.7,30.13) .. controls (91.7,28.81) and (92.76,27.75) .. (94.07,27.75) .. controls (95.39,27.75) and (96.45,28.81) .. (96.45,30.13) .. controls (96.45,31.44) and (95.39,32.5) .. (94.07,32.5) .. controls (92.76,32.5) and (91.7,31.44) .. (91.7,30.13) -- cycle ;
\draw  [fill={rgb, 255:red, 0; green, 0; blue, 0 }  ,fill opacity=1 ] (94.7,17.63) .. controls (94.7,16.31) and (95.76,15.25) .. (97.07,15.25) .. controls (98.39,15.25) and (99.45,16.31) .. (99.45,17.63) .. controls (99.45,18.94) and (98.39,20) .. (97.07,20) .. controls (95.76,20) and (94.7,18.94) .. (94.7,17.63) -- cycle ;
\draw  [fill={rgb, 255:red, 0; green, 0; blue, 0 }  ,fill opacity=1 ] (131.25,18.38) .. controls (131.25,17.06) and (132.31,16) .. (133.63,16) .. controls (134.94,16) and (136,17.06) .. (136,18.38) .. controls (136,19.69) and (134.94,20.75) .. (133.63,20.75) .. controls (132.31,20.75) and (131.25,19.69) .. (131.25,18.38) -- cycle ;
\draw  [fill={rgb, 255:red, 0; green, 0; blue, 0 }  ,fill opacity=1 ] (135.25,30.88) .. controls (135.25,29.56) and (136.31,28.5) .. (137.63,28.5) .. controls (138.94,28.5) and (140,29.56) .. (140,30.88) .. controls (140,32.19) and (138.94,33.25) .. (137.63,33.25) .. controls (136.31,33.25) and (135.25,32.19) .. (135.25,30.88) -- cycle ;
\draw  [fill={rgb, 255:red, 0; green, 0; blue, 0 }  ,fill opacity=1 ] (190.5,17.88) .. controls (190.5,16.56) and (191.56,15.5) .. (192.88,15.5) .. controls (194.19,15.5) and (195.25,16.56) .. (195.25,17.88) .. controls (195.25,19.19) and (194.19,20.25) .. (192.88,20.25) .. controls (191.56,20.25) and (190.5,19.19) .. (190.5,17.88) -- cycle ;
\draw  [fill={rgb, 255:red, 0; green, 0; blue, 0 }  ,fill opacity=1 ] (211.75,17.88) .. controls (211.75,16.56) and (212.81,15.5) .. (214.13,15.5) .. controls (215.44,15.5) and (216.5,16.56) .. (216.5,17.88) .. controls (216.5,19.19) and (215.44,20.25) .. (214.13,20.25) .. controls (212.81,20.25) and (211.75,19.19) .. (211.75,17.88) -- cycle ;
\draw    (308.33,10.08) .. controls (308.53,19.85) and (301.21,20.25) .. (301.21,10.28) ;
\draw    (278.43,9.9) .. controls (278.23,19.67) and (285.55,20.07) .. (285.55,10.11) ;
\draw [color={rgb, 255:red, 208; green, 2; blue, 27 }  ,draw opacity=1 ][line width=0.75]    (282.09,17.09) .. controls (290.63,24.21) and (299.18,22.58) .. (304.46,17.7) ;
\draw [color={rgb, 255:red, 208; green, 2; blue, 27 }  ,draw opacity=1 ][line width=0.75]    (283.67,15.67) .. controls (291,21.33) and (296.67,21.33) .. (303,16.33) ;
\draw    (279.43,40) .. controls (282.83,26.39) and (304.03,26.98) .. (307.1,39.83) ;
\draw    (301.08,39.93) .. controls (298.73,35.22) and (290.49,33.45) .. (285.79,39.65) ;
\draw [color={rgb, 255:red, 208; green, 2; blue, 27 }  ,draw opacity=1 ][line width=1.5]    (293.43,30.1) -- (293.43,35.1) ;
\draw    (367.66,9.75) .. controls (367.87,19.51) and (360.54,19.92) .. (360.54,9.95) ;
\draw    (337.76,9.57) .. controls (337.56,19.33) and (344.88,19.74) .. (344.88,9.77) ;
\draw [color={rgb, 255:red, 208; green, 2; blue, 27 }  ,draw opacity=1 ][line width=0.75]    (341.42,16.75) .. controls (349.97,23.87) and (358.51,22.25) .. (363.8,17.36) ;
\draw [color={rgb, 255:red, 208; green, 2; blue, 27 }  ,draw opacity=1 ][line width=0.75]    (343,15.33) .. controls (350.33,21) and (356,21) .. (362.33,16) ;
\draw  [fill={rgb, 255:red, 0; green, 0; blue, 0 }  ,fill opacity=1 ] (350.73,20.46) .. controls (350.73,19.15) and (351.8,18.08) .. (353.11,18.08) .. controls (354.42,18.08) and (355.48,19.15) .. (355.48,20.46) .. controls (355.48,21.77) and (354.42,22.83) .. (353.11,22.83) .. controls (351.8,22.83) and (350.73,21.77) .. (350.73,20.46) -- cycle ;
\draw    (338.76,39.67) .. controls (342.17,26.06) and (363.36,26.64) .. (366.43,39.5) ;
\draw    (360.42,39.6) .. controls (358.06,34.89) and (349.82,33.12) .. (345.12,39.31) ;
\draw [color={rgb, 255:red, 208; green, 2; blue, 27 }  ,draw opacity=1 ][line width=1.5]    (352.76,29.77) -- (352.76,34.77) ;
\draw  [fill={rgb, 255:red, 0; green, 0; blue, 0 }  ,fill opacity=1 ] (291.06,29.1) .. controls (291.06,27.79) and (292.12,26.72) .. (293.43,26.72) .. controls (294.74,26.72) and (295.81,27.79) .. (295.81,29.1) .. controls (295.81,30.41) and (294.74,31.47) .. (293.43,31.47) .. controls (292.12,31.47) and (291.06,30.41) .. (291.06,29.1) -- cycle ;

\draw (48.75,16.4) node [anchor=north west][inner sep=0.75pt]    {$+$};
\draw (109.75,15.9) node [anchor=north west][inner sep=0.75pt]    {$+$};
\draw (170.25,15.4) node [anchor=north west][inner sep=0.75pt]    {$+$};
\draw (237.92,15.57) node [anchor=north west][inner sep=0.75pt]    {$=$};
\draw (316.25,18.4) node [anchor=north west][inner sep=0.75pt]    {$+$};

\end{tikzpicture}
    \end{center}
    Here, a doubled arc in the right hand side represents a handle attachment (or a saddle move applied twice), and the equation can be checked using the \textit{neck cutting relation}:
    \begin{center}
        \tikzset{every picture/.style={line width=0.75pt}} 

\begin{tikzpicture}[x=0.75pt,y=0.75pt,yscale=-.7,xscale=.7]

\draw [color={rgb, 255:red, 208; green, 2; blue, 27 }  ,draw opacity=1 ][line width=1.5]    (18,33) -- (54,33) ;
\draw [color={rgb, 255:red, 208; green, 2; blue, 27 }  ,draw opacity=1 ][line width=1.5]    (18,39) -- (54,39) ;
\draw  [fill={rgb, 255:red, 0; green, 0; blue, 0 }  ,fill opacity=1 ] (109.7,37.9) .. controls (109.7,35.06) and (112.01,32.75) .. (114.85,32.75) .. controls (117.69,32.75) and (120,35.06) .. (120,37.9) .. controls (120,40.74) and (117.69,43.05) .. (114.85,43.05) .. controls (112.01,43.05) and (109.7,40.74) .. (109.7,37.9) -- cycle ;
\draw  [fill={rgb, 255:red, 0; green, 0; blue, 0 }  ,fill opacity=1 ] (231.7,37.9) .. controls (231.7,35.06) and (234.01,32.75) .. (236.85,32.75) .. controls (239.69,32.75) and (242,35.06) .. (242,37.9) .. controls (242,40.74) and (239.69,43.05) .. (236.85,43.05) .. controls (234.01,43.05) and (231.7,40.74) .. (231.7,37.9) -- cycle ;
\draw [color={rgb, 255:red, 0; green, 0; blue, 0 }  ,draw opacity=1 ][line width=1.5]    (10.53,62.28) .. controls (21.6,49.03) and (20.6,23.03) .. (10.89,11.72) ;
\draw [color={rgb, 255:red, 0; green, 0; blue, 0 }  ,draw opacity=1 ][line width=1.5]    (62.67,62.28) .. controls (51.6,49.03) and (52.6,23.03) .. (62.31,11.72) ;

\draw [color={rgb, 255:red, 0; green, 0; blue, 0 }  ,draw opacity=1 ][line width=1.5]    (108.53,62.28) .. controls (119.6,49.03) and (118.6,23.03) .. (108.89,11.72) ;
\draw [color={rgb, 255:red, 0; green, 0; blue, 0 }  ,draw opacity=1 ][line width=1.5]    (160.67,62.28) .. controls (149.6,49.03) and (150.6,23.03) .. (160.31,11.72) ;

\draw [color={rgb, 255:red, 0; green, 0; blue, 0 }  ,draw opacity=1 ][line width=1.5]    (194.53,62.28) .. controls (205.6,49.03) and (204.6,23.03) .. (194.89,11.72) ;
\draw [color={rgb, 255:red, 0; green, 0; blue, 0 }  ,draw opacity=1 ][line width=1.5]    (246.67,62.28) .. controls (235.6,49.03) and (236.6,23.03) .. (246.31,11.72) ;

\draw [color={rgb, 255:red, 0; green, 0; blue, 0 }  ,draw opacity=1 ][line width=1.5]    (288.53,62.28) .. controls (299.6,49.03) and (298.6,23.03) .. (288.89,11.72) ;
\draw [color={rgb, 255:red, 0; green, 0; blue, 0 }  ,draw opacity=1 ][line width=1.5]    (340.67,62.28) .. controls (329.6,49.03) and (330.6,23.03) .. (340.31,11.72) ;

\draw (77,27.4) node [anchor=north west][inner sep=0.75pt]    {$=$};
\draw (168,27.4) node [anchor=north west][inner sep=0.75pt]    {$+$};
\draw (252,27.4) node [anchor=north west][inner sep=0.75pt]    {$+$};
\draw (275,30.4) node [anchor=north west][inner sep=0.75pt]    {$h$};

\end{tikzpicture}
    \end{center}
    Verifications are left to the reader. Now the desired maps $\Phi^\pm$ between $\CKh(D^+)$ and $\CKh(D^-)$ are defined as 
    \[
    \begin{tikzcd}
\CKh(D^+) 
    \arrow[r, "R^+", shift left] & 
E^+ 
    \arrow[r, "\phi^-", shift left] 
    \arrow[l, "I^+", shift left] & 
E^- 
    \arrow[r, "I^-", shift left]
    \arrow[l, "\phi^+", shift left] & 
\CKh(D^-).
    \arrow[l, "R^-", shift left]
    \end{tikzcd}    
    \]
    where $R^\pm$ and $I^\pm$ are the retractions and inclusions respectively. 

    Finally, the correspondence between the Lee classes can be checked by comparing those images in $E^+$ and $E^-$ under the retractions $R^\pm$. The retractions are given by the following matrices
    \[
        R^+ = \begin{pmatrix}
            G & & & \\
            G'eH & G' & & \\
            & & G & \\
            e'H'eH & e'H' & eH & 1
        \end{pmatrix},\ 
        R^- = \begin{pmatrix}
            1 & & & \\
            & G & & \\
            & & G' & \\
            & & GeH' & G
        \end{pmatrix}.
    \]
    The image of $\ca(D^+) \in \CKh(D^+_{00})$ can be directly computed as 
    \begin{center}
        \tikzset{every picture/.style={line width=0.75pt}} 

\begin{tikzpicture}[x=0.75pt,y=0.75pt,yscale=-1,xscale=1]

\draw [color={rgb, 255:red, 208; green, 2; blue, 27 }  ,draw opacity=1 ][line width=1.5]    (82.5,86.5) .. controls (82.33,77.13) and (82.75,80.5) .. (83.25,71.75) .. controls (83.75,63) and (95,62) .. (96.5,62.5) .. controls (98,63) and (97.25,63.75) .. (97.25,68) .. controls (97.25,72.25) and (97.58,76.52) .. (97.5,84.5) ;
\draw [color={rgb, 255:red, 208; green, 2; blue, 27 }  ,draw opacity=1 ][line width=1.5]    (144.5,87.5) .. controls (144.67,78.13) and (144.75,80.5) .. (144.25,71.75) .. controls (143.75,63) and (132.5,62) .. (131,62.5) .. controls (129.5,63) and (130.25,63.75) .. (130.25,68) .. controls (130.25,72.25) and (130.42,77.52) .. (130.5,85.5) ;
\draw [color={rgb, 255:red, 74; green, 144; blue, 226 }  ,draw opacity=1 ][line width=1.5]    (83.25,10.75) .. controls (83.03,16.58) and (84.04,29.82) .. (84.5,32.5) .. controls (84.96,35.18) and (92.63,41.72) .. (96.56,47.11) .. controls (100.5,52.5) and (97.14,60.37) .. (99.25,59.5) .. controls (101.36,58.63) and (104.75,55.25) .. (112.75,55.75) ;
\draw [color={rgb, 255:red, 208; green, 2; blue, 27 }  ,draw opacity=1 ][line width=1.5]    (114.75,32.25) .. controls (110,32.25) and (100.02,38.16) .. (99.75,43) .. controls (99.48,47.84) and (107.85,53.4) .. (114,53) ;
\draw [color={rgb, 255:red, 74; green, 144; blue, 226 }  ,draw opacity=1 ][line width=1.5]    (98.72,11.75) .. controls (98.72,21.75) and (111.47,30) .. (114.22,30) ;
\draw [color={rgb, 255:red, 74; green, 144; blue, 226 }  ,draw opacity=1 ][line width=1.5]    (144.72,10.75) .. controls (144.94,16.58) and (143.93,29.82) .. (143.47,32.5) .. controls (143.01,35.18) and (135.34,41.72) .. (131.4,47.11) .. controls (127.47,52.5) and (130.83,60.37) .. (128.72,59.5) .. controls (126.6,58.63) and (120.75,55.25) .. (112.75,55.75) ;
\draw [color={rgb, 255:red, 208; green, 2; blue, 27 }  ,draw opacity=1 ][line width=1.5]    (114.75,32.25) .. controls (119.5,32.25) and (128.98,38.16) .. (129.25,43) .. controls (129.52,47.84) and (120.15,53.4) .. (114,53) ;
\draw [color={rgb, 255:red, 74; green, 144; blue, 226 }  ,draw opacity=1 ][line width=1.5]    (129.72,11.75) .. controls (129.72,21.75) and (116.97,30) .. (114.22,30) ;
\draw    (162,48) -- (202.5,48) ;
\draw [shift={(204.5,48)}, rotate = 180] [color={rgb, 255:red, 0; green, 0; blue, 0 }  ][line width=0.75]    (10.93,-4.9) .. controls (6.95,-2.3) and (3.31,-0.67) .. (0,0) .. controls (3.31,0.67) and (6.95,2.3) .. (10.93,4.9)   ;
\draw    (162.5,44) -- (162.5,52) ;

\draw [color={rgb, 255:red, 74; green, 144; blue, 226 }  ,draw opacity=1 ][line width=1.5]    (298.77,14.5) .. controls (300.19,36.67) and (283.21,39.97) .. (284.83,15.21) ;
\draw [color={rgb, 255:red, 74; green, 144; blue, 226 }  ,draw opacity=1 ][line width=1.5]    (255.59,15.21) .. controls (256.33,34.31) and (242.18,39.02) .. (240.79,14.5) ;
\draw [color={rgb, 255:red, 208; green, 2; blue, 27 }  ,draw opacity=1 ][line width=1.5]    (299.24,81) .. controls (300.66,58.98) and (283.68,55.71) .. (285.3,80.3) ;
\draw [color={rgb, 255:red, 208; green, 2; blue, 27 }  ,draw opacity=1 ][line width=1.5]    (256.06,80.3) .. controls (256.8,61.33) and (242.65,56.64) .. (241.26,81) ;

\draw [color={rgb, 255:red, 74; green, 144; blue, 226 }  ,draw opacity=1 ][line width=1.5]    (406.98,15.65) .. controls (400.8,29.45) and (376.06,29.93) .. (368.65,13.99) ;
\draw [color={rgb, 255:red, 74; green, 144; blue, 226 }  ,draw opacity=1 ][line width=1.5]    (416.02,15.18) .. controls (416.02,41.82) and (358.93,39.92) .. (357.53,15.18) ;
\draw [color={rgb, 255:red, 208; green, 2; blue, 27 }  ,draw opacity=1 ][line width=1.5]    (406.98,80.84) .. controls (404.13,67.99) and (376.06,64.18) .. (368.18,81.08) ;
\draw [color={rgb, 255:red, 208; green, 2; blue, 27 }  ,draw opacity=1 ][line width=1.5]    (416.5,82.26) .. controls (415.55,55.62) and (359.4,57.69) .. (358.01,82.26) ;

\draw (311,36.4) node [anchor=north west][inner sep=0.75pt]    {$+$};
\draw (220.5,40.9) node [anchor=north west][inner sep=0.75pt]    {$h^{\epsilon }$};
\draw (335.5,38.9) node [anchor=north west][inner sep=0.75pt]    {$h^{\epsilon '}$};
\draw (8.5,40.4) node [anchor=north west][inner sep=0.75pt]    {$\alpha (D^+) \ =\ $};
\draw (176,27.4) node [anchor=north west][inner sep=0.75pt]    {$R^+$};

\end{tikzpicture}
    \end{center}
    Here $\epsilon, \epsilon' \in \{0, 1\}$ are determined by how the arcs are connected outside the displayed area. Similarly, the image of $\ca(D^-) \in \CKh(D^-_{11})$ is
    \begin{center}
        \tikzset{every picture/.style={line width=0.75pt}} 

\begin{tikzpicture}[x=0.75pt,y=0.75pt,yscale=-1,xscale=1]

\draw [color={rgb, 255:red, 208; green, 2; blue, 27 }  ,draw opacity=1 ][line width=1.5]    (82.5,86.5) .. controls (82.33,77.13) and (82.75,80.5) .. (83.25,71.75) .. controls (83.75,63) and (95,62) .. (96.5,62.5) .. controls (98,63) and (97.25,63.75) .. (97.25,68) .. controls (97.25,72.25) and (97.58,76.52) .. (97.5,84.5) ;
\draw [color={rgb, 255:red, 208; green, 2; blue, 27 }  ,draw opacity=1 ][line width=1.5]    (144.5,87.5) .. controls (144.67,78.13) and (144.75,80.5) .. (144.25,71.75) .. controls (143.75,63) and (132.5,62) .. (131,62.5) .. controls (129.5,63) and (130.25,63.75) .. (130.25,68) .. controls (130.25,72.25) and (130.42,77.52) .. (130.5,85.5) ;
\draw [color={rgb, 255:red, 74; green, 144; blue, 226 }  ,draw opacity=1 ][line width=1.5]    (83.25,10.75) .. controls (83.03,16.58) and (84.04,29.82) .. (84.5,32.5) .. controls (84.96,35.18) and (92.63,41.72) .. (96.56,47.11) .. controls (100.5,52.5) and (97.14,60.37) .. (99.25,59.5) .. controls (101.36,58.63) and (104.75,55.25) .. (112.75,55.75) ;
\draw [color={rgb, 255:red, 208; green, 2; blue, 27 }  ,draw opacity=1 ][line width=1.5]    (114.75,32.25) .. controls (110,32.25) and (100.02,38.16) .. (99.75,43) .. controls (99.48,47.84) and (107.85,53.4) .. (114,53) ;
\draw [color={rgb, 255:red, 74; green, 144; blue, 226 }  ,draw opacity=1 ][line width=1.5]    (98.72,11.75) .. controls (98.72,21.75) and (111.47,30) .. (114.22,30) ;
\draw [color={rgb, 255:red, 74; green, 144; blue, 226 }  ,draw opacity=1 ][line width=1.5]    (144.72,10.75) .. controls (144.94,16.58) and (143.93,29.82) .. (143.47,32.5) .. controls (143.01,35.18) and (135.34,41.72) .. (131.4,47.11) .. controls (127.47,52.5) and (130.83,60.37) .. (128.72,59.5) .. controls (126.6,58.63) and (120.75,55.25) .. (112.75,55.75) ;
\draw [color={rgb, 255:red, 208; green, 2; blue, 27 }  ,draw opacity=1 ][line width=1.5]    (114.75,32.25) .. controls (119.5,32.25) and (128.98,38.16) .. (129.25,43) .. controls (129.52,47.84) and (120.15,53.4) .. (114,53) ;
\draw [color={rgb, 255:red, 74; green, 144; blue, 226 }  ,draw opacity=1 ][line width=1.5]    (129.72,11.75) .. controls (129.72,21.75) and (116.97,30) .. (114.22,30) ;
\draw    (162,48) -- (202.5,48) ;
\draw [shift={(204.5,48)}, rotate = 180] [color={rgb, 255:red, 0; green, 0; blue, 0 }  ][line width=0.75]    (10.93,-4.9) .. controls (6.95,-2.3) and (3.31,-0.67) .. (0,0) .. controls (3.31,0.67) and (6.95,2.3) .. (10.93,4.9)   ;
\draw    (162.5,44) -- (162.5,52) ;

\draw [color={rgb, 255:red, 74; green, 144; blue, 226 }  ,draw opacity=1 ][line width=1.5]    (298.77,14.5) .. controls (300.19,36.67) and (283.21,39.97) .. (284.83,15.21) ;
\draw [color={rgb, 255:red, 74; green, 144; blue, 226 }  ,draw opacity=1 ][line width=1.5]    (255.59,15.21) .. controls (256.33,34.31) and (242.18,39.02) .. (240.79,14.5) ;
\draw [color={rgb, 255:red, 208; green, 2; blue, 27 }  ,draw opacity=1 ][line width=1.5]    (299.24,81) .. controls (300.66,58.98) and (283.68,55.71) .. (285.3,80.3) ;
\draw [color={rgb, 255:red, 208; green, 2; blue, 27 }  ,draw opacity=1 ][line width=1.5]    (256.06,80.3) .. controls (256.8,61.33) and (242.65,56.64) .. (241.26,81) ;

\draw (220.5,40.9) node [anchor=north west][inner sep=0.75pt]    {$h^{\epsilon }$};
\draw (5.5,40.4) node [anchor=north west][inner sep=0.75pt]    {$\alpha \left( D^{-}\right) \ =\ $};
\draw (176,27.4) node [anchor=north west][inner sep=0.75pt]    {$R^{-}$};

\end{tikzpicture}
    \end{center}
    First,
    \[
        \phi^-R^+(\ca(D^+)) = R^-(\ca(D^-))
    \]
    is obvious by definition. Next, from the explicit description of $\phi^+_1$ and $\phi^+_2$, 
    \begin{center}
        \tikzset{every picture/.style={line width=0.75pt}} 

\begin{tikzpicture}[x=0.75pt,y=0.75pt,yscale=-1,xscale=1]

\draw [color={rgb, 255:red, 74; green, 144; blue, 226 }  ,draw opacity=1 ][line width=1.5]    (55.59,16.5) .. controls (56.63,32.86) and (44.1,35.3) .. (45.3,17.02) ;
\draw [color={rgb, 255:red, 74; green, 144; blue, 226 }  ,draw opacity=1 ][line width=1.5]    (23.71,17.02) .. controls (24.26,31.12) and (13.81,34.6) .. (12.79,16.5) ;
\draw [color={rgb, 255:red, 208; green, 2; blue, 27 }  ,draw opacity=1 ][line width=1.5]    (55.94,65.59) .. controls (56.98,49.34) and (44.45,46.92) .. (45.65,65.07) ;
\draw [color={rgb, 255:red, 208; green, 2; blue, 27 }  ,draw opacity=1 ][line width=1.5]    (24.06,65.07) .. controls (24.61,51.06) and (14.16,47.61) .. (13.14,65.59) ;

\draw    (71,40) -- (111.5,40) ;
\draw [shift={(113.5,40)}, rotate = 180] [color={rgb, 255:red, 0; green, 0; blue, 0 }  ][line width=0.75]    (10.93,-4.9) .. controls (6.95,-2.3) and (3.31,-0.67) .. (0,0) .. controls (3.31,0.67) and (6.95,2.3) .. (10.93,4.9)   ;
\draw    (71.5,36) -- (71.5,44) ;

\draw [color={rgb, 255:red, 74; green, 144; blue, 226 }  ,draw opacity=1 ][line width=1.5]    (190.59,16.5) .. controls (191.63,32.86) and (179.1,35.3) .. (180.3,17.02) ;
\draw [color={rgb, 255:red, 74; green, 144; blue, 226 }  ,draw opacity=1 ][line width=1.5]    (158.71,17.02) .. controls (159.26,31.12) and (148.81,34.6) .. (147.79,16.5) ;
\draw [color={rgb, 255:red, 208; green, 2; blue, 27 }  ,draw opacity=1 ][line width=1.5]    (190.94,65.59) .. controls (191.98,49.34) and (179.45,46.92) .. (180.65,65.07) ;
\draw [color={rgb, 255:red, 208; green, 2; blue, 27 }  ,draw opacity=1 ][line width=1.5]    (159.06,65.07) .. controls (159.61,51.06) and (149.16,47.61) .. (148.14,65.59) ;

\draw (81,17.4) node [anchor=north west][inner sep=0.75pt]    {$f_{1}^{+}$};
\draw (123.5,30.9) node [anchor=north west][inner sep=0.75pt]    {$h^{2}$};

\end{tikzpicture}
    \end{center}
    and 
    \begin{center}
        \tikzset{every picture/.style={line width=0.75pt}} 

\begin{tikzpicture}[x=0.75pt,y=0.75pt,yscale=-1,xscale=1]

\draw [color={rgb, 255:red, 74; green, 144; blue, 226 }  ,draw opacity=1 ][line width=1.5]    (57.59,14.5) .. controls (58.63,30.86) and (46.1,33.3) .. (47.3,15.02) ;
\draw [color={rgb, 255:red, 74; green, 144; blue, 226 }  ,draw opacity=1 ][line width=1.5]    (25.71,15.02) .. controls (26.26,29.12) and (15.81,32.6) .. (14.79,14.5) ;
\draw [color={rgb, 255:red, 208; green, 2; blue, 27 }  ,draw opacity=1 ][line width=1.5]    (57.94,63.59) .. controls (58.98,47.34) and (46.45,44.92) .. (47.65,63.07) ;
\draw [color={rgb, 255:red, 208; green, 2; blue, 27 }  ,draw opacity=1 ][line width=1.5]    (26.06,63.07) .. controls (26.61,49.06) and (16.16,45.61) .. (15.14,63.59) ;

\draw    (73,38) -- (113.5,38) ;
\draw [shift={(115.5,38)}, rotate = 180] [color={rgb, 255:red, 0; green, 0; blue, 0 }  ][line width=0.75]    (10.93,-4.9) .. controls (6.95,-2.3) and (3.31,-0.67) .. (0,0) .. controls (3.31,0.67) and (6.95,2.3) .. (10.93,4.9)   ;
\draw    (73.5,34) -- (73.5,42) ;

\draw [color={rgb, 255:red, 74; green, 144; blue, 226 }  ,draw opacity=1 ][line width=1.5]    (248.82,13.98) .. controls (244.16,24.38) and (225.5,24.74) .. (219.92,12.72) ;
\draw [color={rgb, 255:red, 74; green, 144; blue, 226 }  ,draw opacity=1 ][line width=1.5]    (255.64,13.62) .. controls (255.64,33.71) and (212.59,32.27) .. (211.53,13.62) ;
\draw [color={rgb, 255:red, 208; green, 2; blue, 27 }  ,draw opacity=1 ][line width=1.5]    (248.82,63.13) .. controls (246.67,53.44) and (225.5,50.57) .. (219.56,63.31) ;
\draw [color={rgb, 255:red, 208; green, 2; blue, 27 }  ,draw opacity=1 ][line width=1.5]    (256,64.21) .. controls (255.28,44.12) and (212.94,45.68) .. (211.89,64.21) ;

\draw (83,15.4) node [anchor=north west][inner sep=0.75pt]    {$f_{2}^{+}$};
\draw (130.5,28.9) node [anchor=north west][inner sep=0.75pt]    {$h^{( 1 - \epsilon ) + \epsilon' + 1}$};

\end{tikzpicture}
    \end{center}
    Thus
    \[
        \phi^+R^-(\ca(D^-)) = h^2 R^+(\ca(D^+))
    \]
    and the proof is complete. 
\end{proof}

\begin{proposition}
\label{prop:s-gxch}
    Let $L^+, L^-$ be links such that $L^-$ is obtained from $L^+$ by applying a $4$-strand generalized negative crossing change on $L^+$. Then
    \[
        s_h(L^-) \leq s_h(L^+) \leq s_h(L^-) + 4. 
    \]
\end{proposition}

\begin{proof}
    Note that a generalized negative crossing change can be realized by first applying Reidemeister moves on the four parallel strands and then modifying the negative half-twist part to a positive half-twist. Thus the result follows from \Cref{prop:gx-ch} by an argument similar to the proof of \Cref{prop:s-xch}. 
\end{proof}

\begin{remark}
    In \cite[Theorem 1.11]{MMSSW:2023}, the lower bound
    \[
        s^\QQ(L^-) \leq s^\QQ(L^+)
    \]
    is given for the $\QQ$-Rasmussen invariant (with no restrictions on the number of strands). It is obtained by realizing the move as a connected genus-$0$ cobordism from $L^+$ to $L^-$ in $\overline{\CC P^2} \setminus (B^4 \sqcup B^4)$. We expect that our map $\Phi^+$ gives (a part of) the combinatorial description of this geometric map. 
\end{remark}

Now we prove the result for the equivariant case.

\begin{definition}
\label{def:equiv-gen-x-ch}
    An \textit{equivariant generalized negative crossing change} on a strongly invertible link $L$ is an operation that is either a generalized negative crossing change on a set of strands lying on the axis, or two generalized negative crossing changes on two sets of strands that correspond by $\tau$, lying off the axis. 
\end{definition}

\begin{proposition}
\label{prop:si-gxch}
    Let $L^+, L^-$ be strongly invertible links such that $L^-$ is obtained by applying an equivariant $4$-strand generalized negative crossing change to $L^+$. Then
    \[
        \ds_h(L^-) \leq \ds_h(L^+) \leq \ds_h(L^-) + 4a 
    \]
    where $a = 1$ if the move is performed on-axis, and $a = 2$ if performed off-axis. The same statement holds for $\us_h$. 
\end{proposition}

\begin{proof}
    Let $D^+, D^-$ be diagrams of $L^+, L^-$ respectively such that $D^+$ and $D^-$ are locally related by an equivariant generalized crossing change. First suppose the generalized negative crossing change is performed on-axis. Recall the definition of $\Phi^\pm$
    \[
    \begin{tikzcd}
\CKh(D^+) 
    \arrow[r, "R^+", shift left] & 
E^+ 
    \arrow[r, "\phi^-", shift left] 
    \arrow[l, "I^+", shift left] & 
E^- 
    \arrow[r, "I^-", shift left]
    \arrow[l, "\phi^+", shift left] & 
\CKh(D^-).
    \arrow[l, "R^-", shift left]
    \end{tikzcd}    
    \]
    The middle maps $\phi^\pm$ are strictly $\tau$-equivariant. Moreover, since the tangle diagrams appearing in the moves are $\Kh$-simple, from \Cref{lem:htpy-1} it follows that the chain homotopy equivalences $R^\pm$ and $I^\pm$ are homotopy $\tau$-equivariant. Thus the composite maps $\Phi^\pm$ are homotopy $\tau$-equivariant and hence induce maps between the involutive complexes. 
    
    The case when the move is performed off-axis can be proved similarly as in the proof of \Cref{prop:si-xch}.
\end{proof}

\subsection{Behavior under cobordisms}

First recall that in the non-involutive setting, given an oriented cobordism $S$ between links $L, L'$ in $\RR^3$, there is a map between the corresponding complexes
\[
    \phi_S\colon \CKh(D^+) \rightarrow \CKh(D'),
\]
where $D, D'$ are diagrams of $L, L'$ under a fixed projection $\RR^3 \rightarrow \RR^2$. The map $\phi_S$ is obtained by decomposing $S$ into elementary cobordisms and composing the corresponding maps between the intermediate complexes. It is proved in \cite[Theorem 4]{BarNatan:2004} that $\phi_S$ is invariant up to chain homotopy under isotopies of $S$ rel boundary. 
We prove an analogous statement in the involutive setting.

\begin{definition}   
\label{def:simple-eq-cob}
    Let $L, L'$ be two involutive links sharing the same involution $\tau$ of $S^3$. A \textit{simple equivariant cobordism} between $L$ and $L'$ is an oriented cobordism $S$ in $S^3 \times I$ between $L, L'$ satisfying $(\tau \times \id)(S) = S$. A \textit{simple isotopy-equivariant cobordism} between $L, L'$ is defined similarly, except that $(\tau \times \id)(S)$ is only required to be isotopic to $S$ rel boundary. 
\end{definition}

Hereafter we simply write $\tau S$ for $(\tau \times \id)(S)$. 

\begin{lemma}
\label{prop:tau-S}
    Let $\tau$ be an involution on $S^3$ and $S$ be a cobordism between links $L, L'$. Then the following diagram commutes up to homotopy
    \[
\begin{tikzcd}
\CKh(D) \arrow[d, "\tau"] \arrow[r, "\phi_S"] & \CKh(D') \arrow[d, "\tau"] \\
\CKh(\tau D) \arrow[r, "\phi_{\tau S}"]       & \CKh(\tau D').             
\end{tikzcd}
    \]
\end{lemma}

\begin{proof}
    First, assume that $S$ is an elementary cobordism. For the Reidemeister moves R1, R2 and the three Morse moves M1 -- M3, one can check from the explicit descriptions that the square strictly commutes. For the R3 move, although it does not commute strictly, it follows from $\Kh$-simplicity that $(\phi_S)^\tau = \tau \phi_S \tau$ and $\phi_{\tau S}$ are chain homotopic, so the square commutes up to homotopy.

    For a general cobordism $S$, decompose it into elementary cobordisms as $S = S_1 \cup S_2 \cup \cdots \cup S_N$. Then $\tau S_1 \cup \tau S_2 \cup \cdots \cup \tau S_N$ gives an elementary cobordism decomposition of $\tau S$, and we obtain a homotopy commutative diagram
    \[
    \begin{tikzcd}
C(D) \arrow[d, "\tau"] \arrow[r, "\phi_{S_1}"] & C(D_1) \arrow[d, "\tau"] \arrow[r] & \cdots \arrow[r] & C(D_{N-1}) \arrow[d, "\tau"] \arrow[r, "\phi_{S_N}"] & C(D') \arrow[d, "\tau"] \\
C(\tau D) \arrow[r, "\phi_{\tau S_1}"]         & C(\tau D_1) \arrow[r]                & \cdots \arrow[r] & C(\tau D_{N-1}) \arrow[r, "\phi_{\tau S_N}"]          & C(\tau D').           \end{tikzcd}
    \]
\end{proof}

\begin{proposition}
\label{prop:cob-map}
    Let $S$ be a simple isotopy-equivariant cobordism between involutive links $L$ and $L'$. Then there is a cobordism map
    \[
        \phi_S\colon \CKhI(D) \rightarrow \CKhI(D')
    \]
    such that the following diagram commutes
    \[ 
\begin{tikzcd}
{Q\CKh(D)[1]} \arrow[r, hook] \arrow[d, "\phi_S"] & \CKhI(D) \arrow[r, two heads] \arrow[d, "\phi_S"] & \CKh(D) \arrow[d, "\phi_S"] \\
{Q\CKh(D')[1]} \arrow[r, hook]                   & \CKhI(D') \arrow[r, two heads]                            & \CKh(D').
\end{tikzcd} 
    \]
\end{proposition}

\begin{proof}
    By definition $S$ and $\tau S$ are isotopic rel boundary, so from \cite[Theorem 4]{BarNatan:2004} there is a homotopy $\phi_S \htpy \phi_{\tau S}$. Together with \Cref{prop:tau-S}, we have 
    \[
        \phi_S \htpy \phi_{\tau S} \htpy (\phi_S)^\tau.
    \]
    Thus the result follows from \Cref{lem:induced-ch-map}. 
\end{proof}

Now we restrict to strongly invertible links, and prove the behavior of the equivariant Lee classes under simple isotopy-equivariant cobordisms. 

\begin{proposition}
\label{prop:ca-cob}
    Suppose $S$ is a simple isotopy-equivariant cobordism between strongly invertible links $L, L'$, such that every component of $S$ has boundary in $L$. Then under the map $\phi_S$ of \Cref{prop:cob-map}, the equivariant Lee classes modulo torsions for the given orientations correspond as
    \begin{align*}
        \dca(D) &\xmapsto{\ \phi_S\ } h^j \dca(D'),&
        \dcb(D) &\xmapsto{\ \phi_S\ } h^j \dcb(D'),\\
        \uca(D) &\xmapsto{\ \phi_S\ } h^j \uca(D'),&
        \ucb(D) &\xmapsto{\ \phi_S\ } h^j \ucb(D')\\
    \end{align*}
    where 
    \[ 
        j = \frac{\delta w(D, D') - \delta r(D, D') - \chi(S)}{2}.
    \]
\end{proposition}

\begin{proof}
    By an argument similar to the proof of \Cref{prop:ca-reid}, the assertions are immediate from \cite[Proposition 3.4]{Sano:2020}.
\end{proof}

\begin{proposition}
    Under the assumption of \Cref{prop:ca-cob}, we have 
    \begin{align*}
        \ds_h(L) &\leq \ds_h(L') - \chi(S), \\ 
        \us_h(L) &\leq \us_h(L') - \chi(S).
    \end{align*}
    Moreover if every component of $S$ has boundary in both $L$ and $L'$, then we have 
    \begin{align*}
        |\ds_h(L) - \ds_h(L')| \leq -\chi(S), \\ 
        |\us_h(L) - \us_h(L')| \leq -\chi(S).
    \end{align*}
    In particular, both $\ds$ and $\us$ are invariant under \textit{simple isotopy-equivariant link concordances}.
\end{proposition}

\begin{proof}
    Immediate from \Cref{prop:ca-cob}.
\end{proof}

\begin{corollary}
\label{cor:s-bound-g}
    For a strongly invertible knot $K$, both $|\ds_h(K)|$ and $|\us_h(K)|$ bounds $2\widetilde{sig}_4(K)$ from below. 
\end{corollary}

\begin{corollary}
    If either $|\ds_h(K)|$ or $|\us_h(K)|$ is greater than $2g_4(K)$, then no slice surface $S$ of $K$ realizing $g(S) = g_4(K)$ are simple isotopy-equivariant.
\end{corollary}
    \section{The main theorem}
\label{sec:5}

In this section specialize to $(R, h) = (\FF_2[H], H)$. The invariants $\ds_H, \us_H$ will be denoted $\ds, \us$. 


\begin{table}[t]
    \centering
    \small
    
    \begin{tabular}{r|llllllll}
$12$ & $.$ & $.$ & $.$ & $.$ & $.$ & $.$ & $\FF$ & $\FF$ \\
$10$ & $.$ & $.$ & $.$ & $.$ & $.$ & $.$ & $.$ & $.$ \\
$8$ & $.$ & $.$ & $.$ & $.$ & $\FF$ & $\FF$ & $.$ & $.$ \\
$6$ & $.$ & $.$ & $.$ & $\FF$ & $\FF$ & $.$ & $.$ & $.$ \\
$4$ & $.$ & $.$ & $.$ & $.$ & $.$ & $.$ & $.$ & $.$ \\
${\color{orange} 2}$ & $.$ & ${\color{orange} \FF[H]}$ & $\FF$ & $.$ & $.$ & $.$ & $.$ & $.$ \\
${\color{orange} 0}$ & ${\color{orange} \FF[H]}$ & $.$ & $.$ & $.$ & $.$ & $.$ & $.$ & $.$ \\
\hline
$ $ & ${\color{orange} 0}$ & ${\color{orange} 1}$ & $2$ & $3$ & $4$ & $5$ & $6$ & $7$ \\
\end{tabular}
    \caption{$\rBNI(m(9_{46}))$}
    \label{tab:m9_46}
    
    \vspace{1em}
    
    \begin{tabular}{r|llllllllllll}
$18$ & $.$ & $.$ & $.$ & $.$ & $.$ & $.$ & $.$ & $.$ & $.$ & $.$ & $\FF$ & $\FF$ \\
$16$ & $.$ & $.$ & $.$ & $.$ & $.$ & $.$ & $.$ & $.$ & $.$ & $.$ & $.$ & $.$ \\
$14$ & $.$ & $.$ & $.$ & $.$ & $.$ & $.$ & $.$ & $.$ & $\FF$ & $\FF$ & $.$ & $.$ \\
$12$ & $.$ & $.$ & $.$ & $.$ & $.$ & $.$ & $\FF$ & $\FF^{2}$ & $\FF$ & $.$ & $.$ & $.$ \\
$10$ & $.$ & $.$ & $.$ & $.$ & $.$ & $.$ & $.$ & $.$ & $.$ & $.$ & $.$ & $.$ \\
$8$ & $.$ & $.$ & $.$ & $.$ & $\FF$ & $\FF^{2}$ & $\FF$ & $.$ & $.$ & $.$ & $.$ & $.$ \\
$6$ & $.$ & $.$ & $.$ & $\FF$ & $\FF$ & $.$ & $.$ & $.$ & $.$ & $.$ & $.$ & $.$ \\
$4$ & $.$ & $.$ & $.$ & $.$ & $.$ & $.$ & $.$ & $.$ & $.$ & $.$ & $.$ & $.$ \\
${\color{orange} 2}$ & $.$ & ${\color{orange} \FF[H]}$ & $\FF$ & $.$ & $.$ & $.$ & $.$ & $.$ & $.$ & $.$ & $.$ & $.$ \\
${\color{orange} 0}$ & ${\color{orange} \FF[H]}$ & $.$ & $.$ & $.$ & $.$ & $.$ & $.$ & $.$ & $.$ & $.$ & $.$ & $.$ \\
\hline
$ $ & ${\color{orange} 0}$ & ${\color{orange} 1}$ & $2$ & $3$ & $4$ & $5$ & $6$ & $7$ & $8$ & $9$ & $10$ & $11$ \\
\end{tabular}
    \caption{$\rBNI(15n_{103488})$}
    \label{tab:15n}
\end{table}

\begin{restate-proposition}[prop:calc]
    The three strongly invertible slice knots $K = m(9_{46})$, $15n_{103488}$, $17nh_{73}$ have 
    \[
        \ds(K) = 0 < 2 = \us(K).
    \]
\end{restate-proposition}

\begin{proof}
    Proved by direct computations of $\rBNI$ for the three knots. The result for $J_0 = 17nh_{73}$ is given in \Cref{tab:knotJ}. The results for $m(9_{46})$ and $15n_{103488}$ are given in \Cref{tab:m9_46,tab:15n}. 
\end{proof}

\begin{restate-theorem}[thm:1]
    The strongly invertible knot $J_n$ of \Cref{fig:knotJn} has
    \[
        \ds(J_n) = 0 < 2 \leq \us(J_n).
    \]
\end{restate-theorem}

\begin{proof}
    Immediate from \Cref{thm:s-inv} and \Cref{prop:calc,prop:intro-neg-gxch}.
\end{proof}

\begin{figure}[t]
    \begin{subfigure}{0.3\textwidth}
        \centering
        \tikzset{every picture/.style={line width=0.75pt}} 

\begin{tikzpicture}[x=0.75pt,y=0.75pt,yscale=-.85,xscale=.85]

\begin{knot}[
  clip width=3pt,
  end tolerance=1pt,
]

\strand    (39.05,30.83) .. controls (38.61,45.83) and (19.06,44.83) .. (19.33,60.16) .. controls (19.61,75.49) and (39.33,76.16) .. (39.33,90.83) .. controls (39.33,105.49) and (19.33,104.83) .. (18.67,120.16) ;
\strand    (20.24,30.16) .. controls (20.64,45.16) and (39.59,44.83) .. (39.33,60.16) .. controls (39.08,75.49) and (19.98,75.49) .. (19.98,90.16) .. controls (19.98,104.83) and (38.05,105.49) .. (38.67,120.83) ;

\strand    (80.05,30.83) .. controls (79.6,45.83) and (60.06,44.83) .. (60.33,60.16) .. controls (60.61,75.49) and (80.33,76.16) .. (80.33,90.83) .. controls (80.33,105.49) and (60.33,104.83) .. (59.66,120.16) ;
\strand    (61.23,30.16) .. controls (61.64,45.16) and (80.59,44.83) .. (80.33,60.16) .. controls (80.08,75.49) and (60.98,75.49) .. (60.98,90.16) .. controls (60.98,104.83) and (79.05,105.49) .. (79.66,120.83) ;
\strand    (121.05,30.83) .. controls (120.61,45.83) and (101.06,44.83) .. (101.33,60.16) .. controls (101.61,75.49) and (121.33,76.16) .. (121.33,90.83) .. controls (121.33,105.49) and (101.33,104.83) .. (100.67,120.16) ;
\strand    (102.24,30.16) .. controls (102.64,45.16) and (121.59,44.83) .. (121.33,60.16) .. controls (121.08,75.49) and (101.98,75.49) .. (101.98,90.16) .. controls (101.98,104.83) and (120.05,105.49) .. (120.67,120.83) ;

\flipcrossings{1,3,5,7,9}

\end{knot}

\draw    (20.24,30.16) .. controls (20.36,22.23) and (20.2,6.6) .. (69.8,6.6) .. controls (119.4,6.6) and (122.31,18.8) .. (121.05,30.83) ;
\draw    (120.67,120.83) .. controls (120.54,128.75) and (121,143.4) .. (71.4,143.4) .. controls (21.8,143.4) and (18.6,131) .. (18.67,120.16) ;
\draw    (39.05,30.83) .. controls (38.6,22.2) and (61.8,21.8) .. (61.23,30.16) ;
\draw    (80.05,30.83) .. controls (79.6,22.2) and (102.2,22.2) .. (102.24,30.16) ;
\draw    (38.67,120.83) .. controls (38.21,129.45) and (59.8,130.6) .. (59.66,120.16) ;
\draw    (79.66,120.83) .. controls (79.21,129.45) and (103.4,131) .. (100.67,120.16) ;

\end{tikzpicture}
        \caption{}
        \label{fig:m9_46}
    \end{subfigure}
    \begin{subfigure}{0.3\textwidth}
        \centering
        \tikzset{every picture/.style={line width=0.75pt}} 

\begin{tikzpicture}[x=0.75pt,y=0.75pt,yscale=-.85,xscale=.85]

\draw [color={rgb, 255:red, 208; green, 2; blue, 27 }  ,draw opacity=1 ][line width=1.5]    (13.44,30.16) .. controls (13.56,22.23) and (13.4,6.6) .. (63,6.6) .. controls (112.6,6.6) and (115.51,18.8) .. (114.25,30.83) ;
\draw [color={rgb, 255:red, 208; green, 2; blue, 27 }  ,draw opacity=1 ][line width=1.5]    (13.44,30.16) .. controls (13.8,36.67) and (19.8,43.6) .. (23.4,43.2) .. controls (27,42.8) and (30.2,39.2) .. (31.4,34) .. controls (32.6,28.8) and (32.6,24) .. (43.8,24) .. controls (55,24) and (53.8,32) .. (55.8,36.4) .. controls (57.8,40.8) and (61.26,43.18) .. (63.8,42.8) .. controls (66.34,42.42) and (71.4,42.4) .. (73,35.2) .. controls (74.6,28) and (73,24.4) .. (85,24) .. controls (97,23.6) and (95.4,32.8) .. (97.4,36.8) .. controls (99.4,40.8) and (101.8,43.2) .. (104.6,43.2) .. controls (107.4,43.2) and (114.2,38) .. (114.25,30.83) ;
\draw [color={rgb, 255:red, 74; green, 144; blue, 226 }  ,draw opacity=1 ][line width=1.5]    (12.44,120.56) .. controls (12.56,128.49) and (12.4,144.12) .. (62,144.12) .. controls (111.6,144.12) and (114.51,131.92) .. (113.25,119.9) ;
\draw [color={rgb, 255:red, 74; green, 144; blue, 226 }  ,draw opacity=1 ][line width=1.5]    (12.44,120.56) .. controls (12.8,114.06) and (18.8,107.12) .. (22.4,107.52) .. controls (26,107.92) and (29.2,111.52) .. (30.4,116.72) .. controls (31.6,121.92) and (31.6,126.72) .. (42.8,126.72) .. controls (54,126.72) and (52.8,118.72) .. (54.8,114.32) .. controls (56.8,109.92) and (59.2,106.82) .. (63.4,107.2) .. controls (67.6,107.58) and (70.4,108.32) .. (72,115.52) .. controls (73.6,122.72) and (72,126.32) .. (84,126.72) .. controls (96,127.12) and (94.4,117.92) .. (96.4,113.92) .. controls (98.4,109.92) and (99.6,107.12) .. (104.6,107.2) .. controls (109.6,107.28) and (113.2,112.72) .. (113.25,119.9) ;
\draw  [color={rgb, 255:red, 74; green, 144; blue, 226 }  ,draw opacity=1 ][line width=1.5]  (13.19,60.16) .. controls (13.19,52.56) and (17.51,46.4) .. (22.83,46.4) .. controls (28.16,46.4) and (32.47,52.56) .. (32.47,60.16) .. controls (32.47,67.76) and (28.16,73.92) .. (22.83,73.92) .. controls (17.51,73.92) and (13.19,67.76) .. (13.19,60.16) -- cycle ;
\draw  [color={rgb, 255:red, 74; green, 144; blue, 226 }  ,draw opacity=1 ][line width=1.5]  (54.25,60.16) .. controls (54.25,52.56) and (58.56,46.4) .. (63.89,46.4) .. controls (69.21,46.4) and (73.53,52.56) .. (73.53,60.16) .. controls (73.53,67.76) and (69.21,73.92) .. (63.89,73.92) .. controls (58.56,73.92) and (54.25,67.76) .. (54.25,60.16) -- cycle ;
\draw  [color={rgb, 255:red, 74; green, 144; blue, 226 }  ,draw opacity=1 ][line width=1.5]  (94.53,60.16) .. controls (94.53,52.56) and (98.85,46.4) .. (104.17,46.4) .. controls (109.5,46.4) and (113.82,52.56) .. (113.82,60.16) .. controls (113.82,67.76) and (109.5,73.92) .. (104.17,73.92) .. controls (98.85,73.92) and (94.53,67.76) .. (94.53,60.16) -- cycle ;
\draw  [color={rgb, 255:red, 208; green, 2; blue, 27 }  ,draw opacity=1 ][line width=1.5]  (12.39,90.56) .. controls (12.39,82.96) and (16.71,76.8) .. (22.03,76.8) .. controls (27.36,76.8) and (31.67,82.96) .. (31.67,90.56) .. controls (31.67,98.16) and (27.36,104.32) .. (22.03,104.32) .. controls (16.71,104.32) and (12.39,98.16) .. (12.39,90.56) -- cycle ;
\draw  [color={rgb, 255:red, 208; green, 2; blue, 27 }  ,draw opacity=1 ][line width=1.5]  (53.45,90.56) .. controls (53.45,82.96) and (57.76,76.8) .. (63.09,76.8) .. controls (68.41,76.8) and (72.73,82.96) .. (72.73,90.56) .. controls (72.73,98.16) and (68.41,104.32) .. (63.09,104.32) .. controls (57.76,104.32) and (53.45,98.16) .. (53.45,90.56) -- cycle ;
\draw  [color={rgb, 255:red, 208; green, 2; blue, 27 }  ,draw opacity=1 ][line width=1.5]  (93.73,90.56) .. controls (93.73,82.96) and (98.05,76.8) .. (103.37,76.8) .. controls (108.7,76.8) and (113.02,82.96) .. (113.02,90.56) .. controls (113.02,98.16) and (108.7,104.32) .. (103.37,104.32) .. controls (98.05,104.32) and (93.73,98.16) .. (93.73,90.56) -- cycle ;

\draw (57.2,12) node [anchor=north west][inner sep=0.75pt]    {$\textcolor[rgb]{0.82,0.01,0.11}{X}$};
\draw (14.8,81.4) node [anchor=north west][inner sep=0.75pt]    {$\textcolor[rgb]{0.82,0.01,0.11}{X}$};
\draw (56,81.8) node [anchor=north west][inner sep=0.75pt]    {$\textcolor[rgb]{0.82,0.01,0.11}{X}$};
\draw (95.6,82.2) node [anchor=north west][inner sep=0.75pt]    {$\textcolor[rgb]{0.82,0.01,0.11}{X}$};
\draw (16.6,51.8) node [anchor=north west][inner sep=0.75pt]    {$\textcolor[rgb]{0.29,0.56,0.89}{Y}$};
\draw (57.2,52.2) node [anchor=north west][inner sep=0.75pt]    {$\textcolor[rgb]{0.29,0.56,0.89}{Y}$};
\draw (97.6,52.2) node [anchor=north west][inner sep=0.75pt]    {$\textcolor[rgb]{0.29,0.56,0.89}{Y}$};
\draw (56.8,117.72) node [anchor=north west][inner sep=0.75pt]    {$\textcolor[rgb]{0.29,0.56,0.89}{Y}$};
\draw (5.6,39.6) node [anchor=north west][inner sep=0.75pt]  [font=\scriptsize,color={rgb, 255:red, 74; green, 74; blue, 74 }  ,opacity=1 ]  {$0$};
\draw (5.6,71) node [anchor=north west][inner sep=0.75pt]  [font=\scriptsize,color={rgb, 255:red, 74; green, 74; blue, 74 }  ,opacity=1 ]  {$0$};
\draw (5.6,102.4) node [anchor=north west][inner sep=0.75pt]  [font=\scriptsize,color={rgb, 255:red, 74; green, 74; blue, 74 }  ,opacity=1 ]  {$0$};
\draw (116,40.8) node [anchor=north west][inner sep=0.75pt]  [font=\scriptsize,color={rgb, 255:red, 74; green, 74; blue, 74 }  ,opacity=1 ]  {$0$};
\draw (116,72.2) node [anchor=north west][inner sep=0.75pt]  [font=\scriptsize,color={rgb, 255:red, 74; green, 74; blue, 74 }  ,opacity=1 ]  {$0$};
\draw (116,103.6) node [anchor=north west][inner sep=0.75pt]  [font=\scriptsize,color={rgb, 255:red, 74; green, 74; blue, 74 }  ,opacity=1 ]  {$0$};
\draw (74,40) node [anchor=north west][inner sep=0.75pt]  [font=\scriptsize,color={rgb, 255:red, 74; green, 74; blue, 74 }  ,opacity=1 ]  {$1$};
\draw (74.4,71.4) node [anchor=north west][inner sep=0.75pt]  [font=\scriptsize,color={rgb, 255:red, 74; green, 74; blue, 74 }  ,opacity=1 ]  {$1$};
\draw (74.4,102.8) node [anchor=north west][inner sep=0.75pt]  [font=\scriptsize,color={rgb, 255:red, 74; green, 74; blue, 74 }  ,opacity=1 ]  {$1$};

\end{tikzpicture}
        \caption{}
        \label{fig:m9_46-lee}
    \end{subfigure}
    \begin{subfigure}{0.3\textwidth}
        \centering
        \tikzset{every picture/.style={line width=0.75pt}} 

\begin{tikzpicture}[x=0.75pt,y=0.75pt,yscale=-.85,xscale=.85]

\draw [color={rgb, 255:red, 208; green, 2; blue, 27 }  ,draw opacity=1 ][line width=1.5]    (13.44,30.16) .. controls (13.56,22.23) and (13.4,6.6) .. (63,6.6) .. controls (112.6,6.6) and (115.51,18.8) .. (114.25,30.83) ;
\draw [color={rgb, 255:red, 208; green, 2; blue, 27 }  ,draw opacity=1 ][line width=1.5]    (13.44,30.16) .. controls (13.8,36.67) and (19.8,43.6) .. (23.4,43.2) .. controls (27,42.8) and (30.2,39.2) .. (31.4,34) .. controls (32.6,28.8) and (32.6,24) .. (43.8,24) .. controls (55,24) and (53.8,32) .. (55.8,36.4) .. controls (57.8,40.8) and (61.26,43.18) .. (63.8,42.8) .. controls (66.34,42.42) and (71.4,42.4) .. (73,35.2) .. controls (74.6,28) and (73,24.4) .. (85,24) .. controls (97,23.6) and (95.4,32.8) .. (97.4,36.8) .. controls (99.4,40.8) and (101.8,43.2) .. (104.6,43.2) .. controls (107.4,43.2) and (114.2,38) .. (114.25,30.83) ;
\draw [color={rgb, 255:red, 74; green, 144; blue, 226 }  ,draw opacity=1 ][line width=1.5]    (12.44,120.56) .. controls (12.56,128.49) and (12.4,144.12) .. (62,144.12) .. controls (111.6,144.12) and (114.51,131.92) .. (113.25,119.9) ;
\draw [color={rgb, 255:red, 74; green, 144; blue, 226 }  ,draw opacity=1 ][line width=1.5]    (12.44,120.56) .. controls (12.8,114.06) and (18.8,107.12) .. (22.4,107.52) .. controls (26,107.92) and (29.2,111.52) .. (30.4,116.72) .. controls (31.6,121.92) and (31.6,126.72) .. (42.8,126.72) .. controls (54,126.72) and (52.8,118.72) .. (54.8,114.32) .. controls (56.8,109.92) and (59.2,106.82) .. (63.4,107.2) .. controls (67.6,107.58) and (70.4,108.32) .. (72,115.52) .. controls (73.6,122.72) and (72,126.32) .. (84,126.72) .. controls (96,127.12) and (94.4,117.92) .. (96.4,113.92) .. controls (98.4,109.92) and (99.6,107.12) .. (104.6,107.2) .. controls (109.6,107.28) and (113.2,112.72) .. (113.25,119.9) ;
\draw  [color={rgb, 255:red, 74; green, 144; blue, 226 }  ,draw opacity=1 ][line width=1.5]  (13.19,60.16) .. controls (13.19,52.56) and (17.51,46.4) .. (22.83,46.4) .. controls (28.16,46.4) and (32.47,52.56) .. (32.47,60.16) .. controls (32.47,67.76) and (28.16,73.92) .. (22.83,73.92) .. controls (17.51,73.92) and (13.19,67.76) .. (13.19,60.16) -- cycle ;
\draw  [color={rgb, 255:red, 155; green, 155; blue, 155 }  ,draw opacity=1 ][line width=1.5]  (54.25,60.16) .. controls (54.25,52.56) and (58.56,46.4) .. (63.89,46.4) .. controls (69.21,46.4) and (73.53,52.56) .. (73.53,60.16) .. controls (73.53,67.76) and (69.21,73.92) .. (63.89,73.92) .. controls (58.56,73.92) and (54.25,67.76) .. (54.25,60.16) -- cycle ;
\draw  [color={rgb, 255:red, 74; green, 144; blue, 226 }  ,draw opacity=1 ][line width=1.5]  (94.53,60.16) .. controls (94.53,52.56) and (98.85,46.4) .. (104.17,46.4) .. controls (109.5,46.4) and (113.82,52.56) .. (113.82,60.16) .. controls (113.82,67.76) and (109.5,73.92) .. (104.17,73.92) .. controls (98.85,73.92) and (94.53,67.76) .. (94.53,60.16) -- cycle ;
\draw  [color={rgb, 255:red, 208; green, 2; blue, 27 }  ,draw opacity=1 ][line width=1.5]  (12.39,90.56) .. controls (12.39,82.96) and (16.71,76.8) .. (22.03,76.8) .. controls (27.36,76.8) and (31.67,82.96) .. (31.67,90.56) .. controls (31.67,98.16) and (27.36,104.32) .. (22.03,104.32) .. controls (16.71,104.32) and (12.39,98.16) .. (12.39,90.56) -- cycle ;
\draw  [color={rgb, 255:red, 155; green, 155; blue, 155 }  ,draw opacity=1 ][line width=1.5]  (53.45,90.56) .. controls (53.45,82.96) and (57.76,76.8) .. (63.09,76.8) .. controls (68.41,76.8) and (72.73,82.96) .. (72.73,90.56) .. controls (72.73,98.16) and (68.41,104.32) .. (63.09,104.32) .. controls (57.76,104.32) and (53.45,98.16) .. (53.45,90.56) -- cycle ;
\draw  [color={rgb, 255:red, 208; green, 2; blue, 27 }  ,draw opacity=1 ][line width=1.5]  (93.73,90.56) .. controls (93.73,82.96) and (98.05,76.8) .. (103.37,76.8) .. controls (108.7,76.8) and (113.02,82.96) .. (113.02,90.56) .. controls (113.02,98.16) and (108.7,104.32) .. (103.37,104.32) .. controls (98.05,104.32) and (93.73,98.16) .. (93.73,90.56) -- cycle ;

\draw (57.2,12) node [anchor=north west][inner sep=0.75pt]    {$\textcolor[rgb]{0.82,0.01,0.11}{X}$};
\draw (14.8,81.4) node [anchor=north west][inner sep=0.75pt]    {$\textcolor[rgb]{0.82,0.01,0.11}{X}$};
\draw (59,82.8) node [anchor=north west][inner sep=0.75pt]  [color={rgb, 255:red, 155; green, 155; blue, 155 }  ,opacity=1 ]  {$1$};
\draw (95.6,82.2) node [anchor=north west][inner sep=0.75pt]    {$\textcolor[rgb]{0.82,0.01,0.11}{X}$};
\draw (16.6,51.8) node [anchor=north west][inner sep=0.75pt]    {$\textcolor[rgb]{0.29,0.56,0.89}{Y}$};
\draw (60.2,51.2) node [anchor=north west][inner sep=0.75pt]  [color={rgb, 255:red, 155; green, 155; blue, 155 }  ,opacity=1 ]  {$1$};
\draw (97.6,52.2) node [anchor=north west][inner sep=0.75pt]    {$\textcolor[rgb]{0.29,0.56,0.89}{Y}$};
\draw (56.8,117.72) node [anchor=north west][inner sep=0.75pt]    {$\textcolor[rgb]{0.29,0.56,0.89}{Y}$};
\draw (5.6,39.6) node [anchor=north west][inner sep=0.75pt]  [font=\scriptsize,color={rgb, 255:red, 74; green, 74; blue, 74 }  ,opacity=1 ]  {$0$};
\draw (5.6,71) node [anchor=north west][inner sep=0.75pt]  [font=\scriptsize,color={rgb, 255:red, 74; green, 74; blue, 74 }  ,opacity=1 ]  {$0$};
\draw (5.6,102.4) node [anchor=north west][inner sep=0.75pt]  [font=\scriptsize,color={rgb, 255:red, 74; green, 74; blue, 74 }  ,opacity=1 ]  {$0$};
\draw (116,40.8) node [anchor=north west][inner sep=0.75pt]  [font=\scriptsize,color={rgb, 255:red, 74; green, 74; blue, 74 }  ,opacity=1 ]  {$0$};
\draw (116,72.2) node [anchor=north west][inner sep=0.75pt]  [font=\scriptsize,color={rgb, 255:red, 74; green, 74; blue, 74 }  ,opacity=1 ]  {$0$};
\draw (116,103.6) node [anchor=north west][inner sep=0.75pt]  [font=\scriptsize,color={rgb, 255:red, 74; green, 74; blue, 74 }  ,opacity=1 ]  {$0$};
\draw (74,40) node [anchor=north west][inner sep=0.75pt]  [font=\scriptsize,color={rgb, 255:red, 74; green, 74; blue, 74 }  ,opacity=1 ]  {$1$};
\draw (74.4,71.4) node [anchor=north west][inner sep=0.75pt]  [font=\scriptsize,color={rgb, 255:red, 74; green, 74; blue, 74 }  ,opacity=1 ]  {$1$};
\draw (74.4,102.8) node [anchor=north west][inner sep=0.75pt]  [font=\scriptsize,color={rgb, 255:red, 74; green, 74; blue, 74 }  ,opacity=1 ]  {$1$};

\end{tikzpicture}
        \caption{}
        \label{fig:m9_46-lee-red}
    \end{subfigure}    \centering
    \caption{$m(9_{46})$}
\end{figure}

To give some intuition for the inequality $\ds < \us$, we give a human readable proof of a weaker inequality
\[
    \ds(K) = 0 < 2 \leq \us(K)
\]
specifically for $K = m(9_{46})$, inspired by the arguments given in \cite{Abe:2012}. \Cref{fig:m9_46} depicts a diagram $D$ of $m(9_{46})$ in the standard form of the pretzel knot $P(-3, 3, -3)$. Given $w(D) = 3$, $r(D) = 8$, it suffices to show that 
\[
    \dd_h(D) = 2 < 3 \leq \ud_h(D).
\]
\Cref{fig:m9_46-lee} depicts the Lee cycle $\ca(D)$ of $D$ and \Cref{fig:m9_46-lee-red} depicts another $\tau$-invariant cycle $z$, which is obtained from $\ca(D)$ by altering the labels of the two center circles to $1$. By an argument similar to \Cref{ex:trefoil-divisibility}, we may find some $\tau$-invariant element $x \in \CKh^{-1}(D)$ that gives
\[
    \ca(D) \sim h^2 z
\]
where $\sim$ denotes homologous. Using $x \in \CKhI^{-1}(D)$ and $Qx \in \CKhI^{0}(D)$, we obtain
\[
    \dca(D) \sim h^2 z,\ \uca(D) \sim h^2 Qz
\]
and hence
\[
    \dd_h(D),\ \ud_h(D) \geq 2.
\]
Here, it is necessary that $x$ is $\tau$-invariant, otherwise $Q(1 + \tau)x$ will produce some non-zero term in $Q\CKh(D)$. 

Since $K$ is slice, we have $s_h(K) = 0$ and hence $d_h(D) = 2$. Thus from \Cref{prop:div-ineq} we obtain $\dd_h(D) = 2$. It remains to show that the cycle $Qz \in \CKhI(D)$ is at least once more $h$-divisible. This is equivalent to $Qz$ being null-homologous \textit{modulo $h$} in $\CKhI(D)$. Thus it suffices to set $h = 0$ and prove that there exists elements $a, b \in \CKh(D)$ such that
\[
\begin{tikzcd}
 & b \arrow[d, "Q(1 + \tau)", maps to] \arrow[r, "d", maps to] & 0 \\
Qa \arrow[r, "d", maps to] & Qz. &  
\end{tikzcd}    
\]
\[
    d(Qa) + Q(1 + \tau)b = Qz.
\]
Here it is necessary that $a \in \CKh(D)$ is \textit{not} $\tau$-invariant, otherwise we would also have $z \sim 0$ modulo $h$ in $\CKhI(D)$ and $\dd_h(D) \geq 3$. We leave it to the reader to find explicit elements $a, b$ satisfying the above relation. Hence we conclude that $\ud_h(D) \geq 3$. 

This observation demonstrates the following facts:
\begin{itemize}
    \item A cycle $z \in \CKh(D)$ must be symmetric (i.e.\ $\tau$-invariant) to be a cycle in $\CKhI(D)$, whereas $Qz \in \CKhI(D)$ is always a cycle. 
    \item A non-trivial boundary $dx \in \CKh(D)$ must be symmetric to be a boundary in $\CKhI(D)$, whereas $Q(dx) \in \CKhI(D)$ is always a boundary. 
    \item A non-symmetric cycle $z \in \CKh(D)$ gives a non-trivial boundary $Q(1 + \tau)z$ in $Q\CKh(D) \subset \CKhI(D)$.
\end{itemize}
Therefore there are `less cycles' in $\CKh(D) \subset \CKhI(D)$ and `more boundaries' in $Q\CKh(D) \subset \CKhI(D)$, which allow $\ds < \us$ to happen. A more simplified proof of \Cref{prop:calc} using reduction techniques is to be presented in a future paper. 
    \section{On $2$-periodic links}
\label{sec:6}

Analogous constructions for $2$-periodic links are possible by using $\rho = \sigma\tau$ instead of $\tau$, where $\sigma$ is the involution given in \Cref{def:sigma}. Since $\sigma$ is induced from a Frobenius algebra isomorphism, it commutes with any map coming from Bar-Natan's category $\Cob^3_{/l}(B)$, and many of the arguments in \Cref{sec:2,sec:3} run in parallel for $2$-periodic knots and links. The exceptions are (i) equivariant connected sums are not defined, and (ii) the reduced complexes cannot be defined for $2$-periodic links. Here we only state some of the basic definitions and results. 

\begin{definition}
\label{def:ckhi-sigma}
    Given an involutive link $(D, \tau)$, define
    \[
        \CKhI(D, \rho) = \Cone(\ \CKh(D) \xrightarrow{Q(1 + \rho)} Q\CKh(D)\ ).
    \]
    The homology of $\CKhI(D, \rho)$ is denoted $\KhI(D, \rho)$.
\end{definition}

\begin{proposition}
    The chain homotopy type of $\CKhI(D, \rho)$ is an invariant of the involutive link. 
\end{proposition}

Hereafter we assume that $(D, \tau)$ is $2$-periodic. Recall that $O(D)$ denotes the set of all orientations on the underlying unoriented diagram of $D$. Define
\begin{align*}
    O^\rho(D) 
        &:= \{\ o \in O(D) \mid \tau(o) = o \ \},\\
    \overline{O}{}^\rho(D) 
        &:= (O(D) \setminus O^\rho(D)) / (\tau(o) \sim o).
\end{align*}

\begin{proposition}
\label{prop:lee-sigma-tau-invariant}
    $\ca(D, o)$ is $\rho$-invariant for any $o \in O^\rho(D)$, and $(1 + \rho)\ca(D, o')$ is (non-zero and) $\rho$-invariant for any $[o'] \in \overline{O}{}^\rho(D)$. 
\end{proposition}

\begin{proof}
    Immediate from \Cref{lem:ca-tau-comm}.
\end{proof}

\begin{definition}
    The cycles $\dca(D, o) = \ca(D, o)$ for $o \in O^\rho(D)$, $(1 + \rho)\dca(D, o')$ for $[o'] \in \overline{O}{}^\rho(D)$, and $\uca(D, o) = Q\ca(D, o)$ for $o \in O(D)$, regarded in $\CKhI(D, \rho)$, are called the \textit{$\rho$-equivariant Lee cycles} of $D$ (cf.\ \Cref{prop:equiv-lee-cycles}).
\end{definition}
    
\begin{proposition}
    If $h \in R$ is invertible, the involutive Khovanov homology $\KhI(D, \rho)$ is generated by the $\rho$-equivariant Lee classes. More strongly, one may choose an explicit basis as in \Cref{prop:khi-structure}.
\end{proposition}

Hereafter we additionally assume that $R$ is a PID and $h$ is prime. 

\begin{proposition}
    Let $\dd_h(D)$ and $\ud_h(D)$ be the $h$-divisibility (modulo torsion) of the $\rho$-equivariant Lee classes $\dca(D), \uca(D)$ in $\KhI(D, \rho)$. Then the following quantities
    \begin{align*}
        \ds_h(L) &= 2\dd_h(D) + w(D) - r(D) + 1, \\ 
        \us_h(L) &= 2\ud_h(D) + w(D) - r(D) + 1
    \end{align*}
    are invariants of the corresponding $2$-periodic link $L$, satisfying
    \[
        \ds_h(L) \leq s_h(L) \leq \us_h(L).
    \]
    In particular when $K$ is a $2$-periodic knot, then 
    \[
        \ds_h(K) \leq s(K) \leq \us_h(K)
    \]
    where $s(K)$ is the $\FF_2$-Rasmussen invariant. 
\end{proposition}

\begin{proposition}
    Let $S$ be a simple isotopy-equivariant cobordism between $2$-periodic links $L$ and $L'$. Then we have 
    \begin{align*}
        \ds_h(L) &\leq \ds_h(L') - \chi(S), \\ 
        \us_h(L) &\leq \us_h(L') - \chi(S).
    \end{align*}
    Moreover if every component of $S$ has boundary in both $L$ and $L'$, then we have 
    \begin{align*}
        |\ds_h(L) - \ds_h(L')| \leq -\chi(S), \\ 
        |\us_h(L) - \us_h(L')| \leq -\chi(S).
    \end{align*}
    In particular, both $\ds$ and $\us$ are invariant under \textit{simple isotopy-equivariant link concordances}.
\end{proposition}

\begin{corollary}
    For a $2$-periodic knot $K$, both $|\ds_h(K)|$ and $|\us_h(K)|$ bound $2\widetilde{sig}_4(K)$ from below.
\end{corollary}

    \clearpage
    \appendix
\section{Computations for small prime knots}
\label{appendix}

Here, the reduced involutive Khovanov homologies for prime knots with up to $7$ crossings are given, computed by a program\footnote{
    The program is available at \url{https://github.com/taketo1024/yui}. 
}
developed by the author, with the diagrams given in \cite[Section 6.5]{Lobb-Watson:2021} as the input data. 

Before describing the results, we review some of the basic facts on strongly invertible knots. See \cite[Section 3]{Sakuma:86} for the references. 

\begin{proposition}
\label{prop:basic-fact}
    $ $
    \begin{enumerate}
        \item Every torus knot admits exactly one inverting involution.
        \item An invertible hyperbolic knot admits exactly one or two inverting involutions; two when it has (cyclic or free) period $2$, or one otherwise. 
        \item For a fully amphichiral hyperbolic knot $K$, 
        \begin{enumerate}
            \item if $K$ admits a unique inverting involution $\tau$, then $(K, \tau) \isom (K, \tau)^*$,
            \item if $K$ admits two inverting involutions $\tau, \tau'$, then $(K, \tau) \isom (K, \tau')^*$. 
        \end{enumerate}
    \end{enumerate}
\end{proposition}

Thus for a torus knot (T) and a fully amphichiral hyperbolic knot (HA), it suffices to perform the computation for only one of its involutions. For a non-amphichiral hyperbolic knot (HN), there are possibly two non-equivalent involutions. In fact, all non-amphichiral hyperbolic knots with up to $7$ crossings have exactly two non-equivalent involutions, and are distinguished by Sakuma's $\eta$-polynomial (see \cite[Appendix]{Sakuma:86}). 

The following list shows the computation results for strongly invertible prime knots with up to $7$ crossings. Each item displays the name of the strongly invertible knot $K$, together with its type: (T), (HA) or (HN). For a (HN) type knot, the two distinct strongly invertible knots are distinguished by suffixes $a, b$, such as $5_{2a}$ and $5_{2b}$. Its reduced involutive Khovanov homology $(R, h) = (\FF_2, 0)$ is displayed on the left and its reduced involutive Bar-Natan homology $(R, h) = (\FF_2[H], H)$ on the right. Note that both theories are bigraded, and the bigrading of each summand can be extracted from its computed generator. For the sake of readability, $\FF_2$ is simply written as $\FF$ and also $\FF_2[H]/(H)$ is replaced with $\FF$. From the latter table, one can read off the values of the equivariant Rasmussen invariants $(\ds, \us)$, however all knots in this list have $\ds = \us = s^{\FF_2}$. One can observe that $\rBNI$ is generally not a direct sum of two shifted copies of $\rBN$, as can be seen in $7_{4b}$ and $7_{7b}$. It is also notable that these two have order two $H$-torsions in $\rBNI$.

For the purpose of distinguishing non-equivalent involutions on the same knot, we see that the only successful ones are $7_4$ and $7_7$, so our invariants are not as strong as Sakuma's $\eta$-polynomial and Lobb--Watson's triply graded invariant $\CKh_\tau$ (\cite{Lobb-Watson:2021}). Some additional structures, such as filtrations, might be given on $\KhI$ to further strengthen the invariant. 

\newgeometry{top=2cm, bottom=2.5cm, left=1.5cm, right=1.5cm} 
\begin{itemize}[leftmargin=*]
\setlength\itemsep{3em}

\item $3_1$ (T)\\\\
\noindent
\begin{minipage}{0.5\textwidth}
\begin{tabular}{r|lllll}
$8$ & $.$ & $.$ & $.$ & $\FF$ & $\FF$ \\
$6$ & $.$ & $.$ & $\FF$ & $\FF$ & $.$ \\
$4$ & $.$ & $.$ & $.$ & $.$ & $.$ \\
$2$ & $\FF$ & $\FF$ & $.$ & $.$ & $.$ \\
\hline
$ $ & $0$ & $1$ & $2$ & $3$ & $4$ \\
\end{tabular}
\end{minipage}
\begin{minipage}{0.5\textwidth}
\begin{tabular}{r|lllll}
$8$ & $.$ & $.$ & $.$ & $\FF$ & $\FF$ \\
$6$ & $.$ & $.$ & $.$ & $.$ & $.$ \\
$4$ & $.$ & $.$ & $.$ & $.$ & $.$ \\
$2$ & $\FF[H]$ & $\FF[H]$ & $.$ & $.$ & $.$ \\
\hline
$ $ & $0$ & $1$ & $2$ & $3$ & $4$ \\
\end{tabular}
\end{minipage}

\item $4_1$ (HA) \\\\
\begin{minipage}{0.5\textwidth}
\begin{tabular}{r|llllll}
$4$ & $.$ & $.$ & $.$ & $.$ & $\FF$ & $\FF$ \\
$2$ & $.$ & $.$ & $.$ & $\FF$ & $\FF$ & $.$ \\
$0$ & $.$ & $.$ & $\FF$ & $\FF$ & $.$ & $.$ \\
$-2$ & $.$ & $\FF$ & $\FF$ & $.$ & $.$ & $.$ \\
$-4$ & $\FF$ & $\FF$ & $.$ & $.$ & $.$ & $.$ \\
\hline
$ $ & $-2$ & $-1$ & $0$ & $1$ & $2$ & $3$ \\
\end{tabular}
\end{minipage}
\begin{minipage}{0.5\textwidth}
\begin{tabular}{r|llllll}
$4$ & $.$ & $.$ & $.$ & $.$ & $\FF$ & $\FF$ \\
$2$ & $.$ & $.$ & $.$ & $.$ & $.$ & $.$ \\
$0$ & $.$ & $.$ & $\FF[H]$ & $\FF[H]$ & $.$ & $.$ \\
$-2$ & $.$ & $\FF$ & $\FF$ & $.$ & $.$ & $.$ \\
$-4$ & $.$ & $.$ & $.$ & $.$ & $.$ & $.$ \\
\hline
$ $ & $-2$ & $-1$ & $0$ & $1$ & $2$ & $3$ \\
\end{tabular}
\end{minipage}

\item $5_1$ (T)\\\\
\noindent
\begin{minipage}{0.5\textwidth}
\begin{tabular}{r|lllllll}
$14$ & $.$ & $.$ & $.$ & $.$ & $.$ & $\FF$ & $\FF$ \\
$12$ & $.$ & $.$ & $.$ & $.$ & $\FF$ & $\FF$ & $.$ \\
$10$ & $.$ & $.$ & $.$ & $\FF$ & $\FF$ & $.$ & $.$ \\
$8$ & $.$ & $.$ & $\FF$ & $\FF$ & $.$ & $.$ & $.$ \\
$6$ & $.$ & $.$ & $.$ & $.$ & $.$ & $.$ & $.$ \\
$4$ & $\FF$ & $\FF$ & $.$ & $.$ & $.$ & $.$ & $.$ \\
\hline
$ $ & $0$ & $1$ & $2$ & $3$ & $4$ & $5$ & $6$ \\
\end{tabular}
\end{minipage}
\begin{minipage}{0.5\textwidth}
\begin{tabular}{r|lllllll}
$14$ & $.$ & $.$ & $.$ & $.$ & $.$ & $\FF$ & $\FF$ \\
$12$ & $.$ & $.$ & $.$ & $.$ & $.$ & $.$ & $.$ \\
$10$ & $.$ & $.$ & $.$ & $\FF$ & $\FF$ & $.$ & $.$ \\
$8$ & $.$ & $.$ & $.$ & $.$ & $.$ & $.$ & $.$ \\
$6$ & $.$ & $.$ & $.$ & $.$ & $.$ & $.$ & $.$ \\
$4$ & $\FF[H]$ & $\FF[H]$ & $.$ & $.$ & $.$ & $.$ & $.$ \\
\hline
$ $ & $0$ & $1$ & $2$ & $3$ & $4$ & $5$ & $6$ \\
\end{tabular}
\end{minipage}

\item $5_{2a}$, $5_{2b}$ (HN)\\\\
\noindent
\begin{minipage}{0.5\textwidth}
\begin{tabular}{r|lllllll}
$12$ & $.$ & $.$ & $.$ & $.$ & $.$ & $\FF$ & $\FF$ \\
$10$ & $.$ & $.$ & $.$ & $.$ & $\FF$ & $\FF$ & $.$ \\
$8$ & $.$ & $.$ & $.$ & $\FF$ & $\FF$ & $.$ & $.$ \\
$6$ & $.$ & $.$ & $\FF^{2}$ & $\FF^{2}$ & $.$ & $.$ & $.$ \\
$4$ & $.$ & $\FF$ & $\FF$ & $.$ & $.$ & $.$ & $.$ \\
$2$ & $\FF$ & $\FF$ & $.$ & $.$ & $.$ & $.$ & $.$ \\
\hline
$ $ & $0$ & $1$ & $2$ & $3$ & $4$ & $5$ & $6$ \\
\end{tabular}

\end{minipage}
\begin{minipage}{0.5\textwidth}
\begin{tabular}{r|lllllll}
$12$ & $.$ & $.$ & $.$ & $.$ & $.$ & $\FF$ & $\FF$ \\
$10$ & $.$ & $.$ & $.$ & $.$ & $.$ & $.$ & $.$ \\
$8$ & $.$ & $.$ & $.$ & $\FF$ & $\FF$ & $.$ & $.$ \\
$6$ & $.$ & $.$ & $\FF$ & $\FF$ & $.$ & $.$ & $.$ \\
$4$ & $.$ & $.$ & $.$ & $.$ & $.$ & $.$ & $.$ \\
$2$ & $\FF[H]$ & $\FF[H]$ & $.$ & $.$ & $.$ & $.$ & $.$ \\
\hline
$ $ & $0$ & $1$ & $2$ & $3$ & $4$ & $5$ & $6$ \\
\end{tabular}
\end{minipage}

\item $6_{1a}, 6_{1b}$ (HN)\\\\
\noindent
\begin{minipage}{0.5\textwidth}
\begin{tabular}{r|llllllll}
$8$ & $.$ & $.$ & $.$ & $.$ & $.$ & $.$ & $\FF$ & $\FF$ \\
$6$ & $.$ & $.$ & $.$ & $.$ & $.$ & $\FF$ & $\FF$ & $.$ \\
$4$ & $.$ & $.$ & $.$ & $.$ & $\FF$ & $\FF$ & $.$ & $.$ \\
$2$ & $.$ & $.$ & $.$ & $\FF^{2}$ & $\FF^{2}$ & $.$ & $.$ & $.$ \\
$0$ & $.$ & $.$ & $\FF^{2}$ & $\FF^{2}$ & $.$ & $.$ & $.$ & $.$ \\
$-2$ & $.$ & $\FF$ & $\FF$ & $.$ & $.$ & $.$ & $.$ & $.$ \\
$-4$ & $\FF$ & $\FF$ & $.$ & $.$ & $.$ & $.$ & $.$ & $.$ \\
\hline
$ $ & $-2$ & $-1$ & $0$ & $1$ & $2$ & $3$ & $4$ & $5$ \\
\end{tabular}
\end{minipage}
\begin{minipage}{0.5\textwidth}
\begin{tabular}{r|llllllll}
$8$ & $.$ & $.$ & $.$ & $.$ & $.$ & $.$ & $\FF$ & $\FF$ \\
$6$ & $.$ & $.$ & $.$ & $.$ & $.$ & $.$ & $.$ & $.$ \\
$4$ & $.$ & $.$ & $.$ & $.$ & $\FF$ & $\FF$ & $.$ & $.$ \\
$2$ & $.$ & $.$ & $.$ & $\FF$ & $\FF$ & $.$ & $.$ & $.$ \\
$0$ & $.$ & $.$ & $\FF[H]$ & $\FF[H]$ & $.$ & $.$ & $.$ & $.$ \\
$-2$ & $.$ & $\FF$ & $\FF$ & $.$ & $.$ & $.$ & $.$ & $.$ \\
$-4$ & $.$ & $.$ & $.$ & $.$ & $.$ & $.$ & $.$ & $.$ \\
\hline
$ $ & $-2$ & $-1$ & $0$ & $1$ & $2$ & $3$ & $4$ & $5$ \\
\end{tabular}
\end{minipage}

\pagebreak

\item $6_{2a}, 6_{2b}$ (HN)\\\\
\noindent
\begin{minipage}{0.5\textwidth}
\begin{tabular}{r|llllllll}
$10$ & $.$ & $.$ & $.$ & $.$ & $.$ & $.$ & $\FF$ & $\FF$ \\
$8$ & $.$ & $.$ & $.$ & $.$ & $.$ & $\FF^{2}$ & $\FF^{2}$ & $.$ \\
$6$ & $.$ & $.$ & $.$ & $.$ & $\FF^{2}$ & $\FF^{2}$ & $.$ & $.$ \\
$4$ & $.$ & $.$ & $.$ & $\FF^{2}$ & $\FF^{2}$ & $.$ & $.$ & $.$ \\
$2$ & $.$ & $.$ & $\FF^{2}$ & $\FF^{2}$ & $.$ & $.$ & $.$ & $.$ \\
$0$ & $.$ & $\FF$ & $\FF$ & $.$ & $.$ & $.$ & $.$ & $.$ \\
$-2$ & $\FF$ & $\FF$ & $.$ & $.$ & $.$ & $.$ & $.$ & $.$ \\
\hline
$ $ & $-2$ & $-1$ & $0$ & $1$ & $2$ & $3$ & $4$ & $5$ \\
\end{tabular}
\end{minipage}
\begin{minipage}{0.5\textwidth}
\begin{tabular}{r|llllllll}
$10$ & $.$ & $.$ & $.$ & $.$ & $.$ & $.$ & $\FF$ & $\FF$ \\
$8$ & $.$ & $.$ & $.$ & $.$ & $.$ & $\FF$ & $\FF$ & $.$ \\
$6$ & $.$ & $.$ & $.$ & $.$ & $\FF$ & $\FF$ & $.$ & $.$ \\
$4$ & $.$ & $.$ & $.$ & $\FF$ & $\FF$ & $.$ & $.$ & $.$ \\
$2$ & $.$ & $.$ & $\FF[H]$ & $\FF[H]$ & $.$ & $.$ & $.$ & $.$ \\
$0$ & $.$ & $\FF$ & $\FF$ & $.$ & $.$ & $.$ & $.$ & $.$ \\
$-2$ & $.$ & $.$ & $.$ & $.$ & $.$ & $.$ & $.$ & $.$ \\
\hline
$ $ & $-2$ & $-1$ & $0$ & $1$ & $2$ & $3$ & $4$ & $5$ \\
\end{tabular}
\end{minipage}

\item $6_3$ (HA)\\\\
\noindent
\begin{minipage}{0.5\textwidth}
\small
\begin{tabular}{r|llllllll}
$6$ & $.$ & $.$ & $.$ & $.$ & $.$ & $.$ & $\FF$ & $\FF$ \\
$4$ & $.$ & $.$ & $.$ & $.$ & $.$ & $\FF^{2}$ & $\FF^{2}$ & $.$ \\
$2$ & $.$ & $.$ & $.$ & $.$ & $\FF^{2}$ & $\FF^{2}$ & $.$ & $.$ \\
$0$ & $.$ & $.$ & $.$ & $\FF^{3}$ & $\FF^{3}$ & $.$ & $.$ & $.$ \\
$-2$ & $.$ & $.$ & $\FF^{2}$ & $\FF^{2}$ & $.$ & $.$ & $.$ & $.$ \\
$-4$ & $.$ & $\FF^{2}$ & $\FF^{2}$ & $.$ & $.$ & $.$ & $.$ & $.$ \\
$-6$ & $\FF$ & $\FF$ & $.$ & $.$ & $.$ & $.$ & $.$ & $.$ \\
\hline
$ $ & $-3$ & $-2$ & $-1$ & $0$ & $1$ & $2$ & $3$ & $4$ \\
\end{tabular}
\end{minipage}
\begin{minipage}{0.5\textwidth}
\small
\begin{tabular}{r|llllllll}
$6$ & $.$ & $.$ & $.$ & $.$ & $.$ & $.$ & $\FF$ & $\FF$ \\
$4$ & $.$ & $.$ & $.$ & $.$ & $.$ & $\FF$ & $\FF$ & $.$ \\
$2$ & $.$ & $.$ & $.$ & $.$ & $\FF$ & $\FF$ & $.$ & $.$ \\
$0$ & $.$ & $.$ & $.$ & $\FF[H] \oplus \FF$ & $\FF[H] \oplus \FF$ & $.$ & $.$ & $.$ \\
$-2$ & $.$ & $.$ & $\FF$ & $\FF$ & $.$ & $.$ & $.$ & $.$ \\
$-4$ & $.$ & $\FF$ & $\FF$ & $.$ & $.$ & $.$ & $.$ & $.$ \\
$-6$ & $.$ & $.$ & $.$ & $.$ & $.$ & $.$ & $.$ & $.$ \\
\hline
$ $ & $-3$ & $-2$ & $-1$ & $0$ & $1$ & $2$ & $3$ & $4$ \\
\end{tabular}
\end{minipage}

\item $7_1$ (T)\\\\
\noindent
\begin{minipage}{0.5\textwidth}
\small
\begin{tabular}{r|lllllllll}
$20$ & $.$ & $.$ & $.$ & $.$ & $.$ & $.$ & $.$ & $\FF$ & $\FF$ \\
$18$ & $.$ & $.$ & $.$ & $.$ & $.$ & $.$ & $\FF$ & $\FF$ & $.$ \\
$16$ & $.$ & $.$ & $.$ & $.$ & $.$ & $\FF$ & $\FF$ & $.$ & $.$ \\
$14$ & $.$ & $.$ & $.$ & $.$ & $\FF$ & $\FF$ & $.$ & $.$ & $.$ \\
$12$ & $.$ & $.$ & $.$ & $\FF$ & $\FF$ & $.$ & $.$ & $.$ & $.$ \\
$10$ & $.$ & $.$ & $\FF$ & $\FF$ & $.$ & $.$ & $.$ & $.$ & $.$ \\
$8$ & $.$ & $.$ & $.$ & $.$ & $.$ & $.$ & $.$ & $.$ & $.$ \\
$6$ & $\FF$ & $\FF$ & $.$ & $.$ & $.$ & $.$ & $.$ & $.$ & $.$ \\
\hline
$ $ & $0$ & $1$ & $2$ & $3$ & $4$ & $5$ & $6$ & $7$ & $8$ \\
\end{tabular}
\end{minipage}
\begin{minipage}{0.5\textwidth}
\small
\begin{tabular}{r|lllllllll}
$20$ & $.$ & $.$ & $.$ & $.$ & $.$ & $.$ & $.$ & $\FF$ & $\FF$ \\
$18$ & $.$ & $.$ & $.$ & $.$ & $.$ & $.$ & $.$ & $.$ & $.$ \\
$16$ & $.$ & $.$ & $.$ & $.$ & $.$ & $\FF$ & $\FF$ & $.$ & $.$ \\
$14$ & $.$ & $.$ & $.$ & $.$ & $.$ & $.$ & $.$ & $.$ & $.$ \\
$12$ & $.$ & $.$ & $.$ & $\FF$ & $\FF$ & $.$ & $.$ & $.$ & $.$ \\
$10$ & $.$ & $.$ & $.$ & $.$ & $.$ & $.$ & $.$ & $.$ & $.$ \\
$8$ & $.$ & $.$ & $.$ & $.$ & $.$ & $.$ & $.$ & $.$ & $.$ \\
$6$ & $\FF[H]$ & $\FF[H]$ & $.$ & $.$ & $.$ & $.$ & $.$ & $.$ & $.$ \\
\hline
$ $ & $0$ & $1$ & $2$ & $3$ & $4$ & $5$ & $6$ & $7$ & $8$ \\
\end{tabular}
\end{minipage}

\item $7_{2a}, 7_{2b}$ (HN)\\\\
\noindent
\begin{minipage}{0.5\textwidth}
\small
\begin{tabular}{r|lllllllll}
$16$ & $.$ & $.$ & $.$ & $.$ & $.$ & $.$ & $.$ & $\FF$ & $\FF$ \\
$14$ & $.$ & $.$ & $.$ & $.$ & $.$ & $.$ & $\FF$ & $\FF$ & $.$ \\
$12$ & $.$ & $.$ & $.$ & $.$ & $.$ & $\FF$ & $\FF$ & $.$ & $.$ \\
$10$ & $.$ & $.$ & $.$ & $.$ & $\FF^{2}$ & $\FF^{2}$ & $.$ & $.$ & $.$ \\
$8$ & $.$ & $.$ & $.$ & $\FF^{2}$ & $\FF^{2}$ & $.$ & $.$ & $.$ & $.$ \\
$6$ & $.$ & $.$ & $\FF^{2}$ & $\FF^{2}$ & $.$ & $.$ & $.$ & $.$ & $.$ \\
$4$ & $.$ & $\FF$ & $\FF$ & $.$ & $.$ & $.$ & $.$ & $.$ & $.$ \\
$2$ & $\FF$ & $\FF$ & $.$ & $.$ & $.$ & $.$ & $.$ & $.$ & $.$ \\
\hline
$ $ & $0$ & $1$ & $2$ & $3$ & $4$ & $5$ & $6$ & $7$ & $8$ \\
\end{tabular}
\end{minipage}
\begin{minipage}{0.5\textwidth}
\small
\begin{tabular}{r|lllllllll}
$16$ & $.$ & $.$ & $.$ & $.$ & $.$ & $.$ & $.$ & $\FF$ & $\FF$ \\
$14$ & $.$ & $.$ & $.$ & $.$ & $.$ & $.$ & $.$ & $.$ & $.$ \\
$12$ & $.$ & $.$ & $.$ & $.$ & $.$ & $\FF$ & $\FF$ & $.$ & $.$ \\
$10$ & $.$ & $.$ & $.$ & $.$ & $\FF$ & $\FF$ & $.$ & $.$ & $.$ \\
$8$ & $.$ & $.$ & $.$ & $\FF$ & $\FF$ & $.$ & $.$ & $.$ & $.$ \\
$6$ & $.$ & $.$ & $\FF$ & $\FF$ & $.$ & $.$ & $.$ & $.$ & $.$ \\
$4$ & $.$ & $.$ & $.$ & $.$ & $.$ & $.$ & $.$ & $.$ & $.$ \\
$2$ & $\FF[H]$ & $\FF[H]$ & $.$ & $.$ & $.$ & $.$ & $.$ & $.$ & $.$ \\
\hline
$ $ & $0$ & $1$ & $2$ & $3$ & $4$ & $5$ & $6$ & $7$ & $8$ \\
\end{tabular}
\end{minipage}

\pagebreak

\item $7_{3a}, 7_{3b}$ (HN)\\\\
\noindent
\begin{minipage}{0.5\textwidth}
\small
\begin{tabular}{r|lllllllll}
$18$ & $.$ & $.$ & $.$ & $.$ & $.$ & $.$ & $.$ & $\FF$ & $\FF$ \\
$16$ & $.$ & $.$ & $.$ & $.$ & $.$ & $.$ & $\FF$ & $\FF$ & $.$ \\
$14$ & $.$ & $.$ & $.$ & $.$ & $.$ & $\FF^{2}$ & $\FF^{2}$ & $.$ & $.$ \\
$12$ & $.$ & $.$ & $.$ & $.$ & $\FF^{3}$ & $\FF^{3}$ & $.$ & $.$ & $.$ \\
$10$ & $.$ & $.$ & $.$ & $\FF^{2}$ & $\FF^{2}$ & $.$ & $.$ & $.$ & $.$ \\
$8$ & $.$ & $.$ & $\FF^{2}$ & $\FF^{2}$ & $.$ & $.$ & $.$ & $.$ & $.$ \\
$6$ & $.$ & $\FF$ & $\FF$ & $.$ & $.$ & $.$ & $.$ & $.$ & $.$ \\
$4$ & $\FF$ & $\FF$ & $.$ & $.$ & $.$ & $.$ & $.$ & $.$ & $.$ \\
\hline
$ $ & $0$ & $1$ & $2$ & $3$ & $4$ & $5$ & $6$ & $7$ & $8$ \\
\end{tabular}
\end{minipage}
\begin{minipage}{0.5\textwidth}
\small
\begin{tabular}{r|lllllllll}
$18$ & $.$ & $.$ & $.$ & $.$ & $.$ & $.$ & $.$ & $\FF$ & $\FF$ \\
$16$ & $.$ & $.$ & $.$ & $.$ & $.$ & $.$ & $.$ & $.$ & $.$ \\
$14$ & $.$ & $.$ & $.$ & $.$ & $.$ & $\FF^{2}$ & $\FF^{2}$ & $.$ & $.$ \\
$12$ & $.$ & $.$ & $.$ & $.$ & $\FF$ & $\FF$ & $.$ & $.$ & $.$ \\
$10$ & $.$ & $.$ & $.$ & $\FF$ & $\FF$ & $.$ & $.$ & $.$ & $.$ \\
$8$ & $.$ & $.$ & $\FF$ & $\FF$ & $.$ & $.$ & $.$ & $.$ & $.$ \\
$6$ & $.$ & $.$ & $.$ & $.$ & $.$ & $.$ & $.$ & $.$ & $.$ \\
$4$ & $\FF[H]$ & $\FF[H]$ & $.$ & $.$ & $.$ & $.$ & $.$ & $.$ & $.$ \\
\hline
$ $ & $0$ & $1$ & $2$ & $3$ & $4$ & $5$ & $6$ & $7$ & $8$ \\
\end{tabular}
\end{minipage}

\item $7_{4a}$ (HN)\\\\
\noindent
\begin{minipage}{0.5\textwidth}
\small
\begin{tabular}{r|lllllllll}
$16$ & $.$ & $.$ & $.$ & $.$ & $.$ & $.$ & $.$ & $\FF$ & $\FF$ \\
$14$ & $.$ & $.$ & $.$ & $.$ & $.$ & $.$ & $\FF$ & $\FF$ & $.$ \\
$12$ & $.$ & $.$ & $.$ & $.$ & $.$ & $\FF^{2}$ & $\FF^{2}$ & $.$ & $.$ \\
$10$ & $.$ & $.$ & $.$ & $.$ & $\FF^{3}$ & $\FF^{3}$ & $.$ & $.$ & $.$ \\
$8$ & $.$ & $.$ & $.$ & $\FF^{2}$ & $\FF^{2}$ & $.$ & $.$ & $.$ & $.$ \\
$6$ & $.$ & $.$ & $\FF^{3}$ & $\FF^{3}$ & $.$ & $.$ & $.$ & $.$ & $.$ \\
$4$ & $.$ & $\FF^{2}$ & $\FF^{2}$ & $.$ & $.$ & $.$ & $.$ & $.$ & $.$ \\
$2$ & $\FF$ & $\FF$ & $.$ & $.$ & $.$ & $.$ & $.$ & $.$ & $.$ \\
\hline
$ $ & $0$ & $1$ & $2$ & $3$ & $4$ & $5$ & $6$ & $7$ & $8$ \\
\end{tabular}
\end{minipage}
\begin{minipage}{0.5\textwidth}
\small
\begin{tabular}{r|lllllllll}
$16$ & $.$ & $.$ & $.$ & $.$ & $.$ & $.$ & $.$ & $\FF$ & $\FF$ \\
$14$ & $.$ & $.$ & $.$ & $.$ & $.$ & $.$ & $.$ & $.$ & $.$ \\
$12$ & $.$ & $.$ & $.$ & $.$ & $.$ & $\FF^{2}$ & $\FF^{2}$ & $.$ & $.$ \\
$10$ & $.$ & $.$ & $.$ & $.$ & $\FF$ & $\FF$ & $.$ & $.$ & $.$ \\
$8$ & $.$ & $.$ & $.$ & $\FF$ & $\FF$ & $.$ & $.$ & $.$ & $.$ \\
$6$ & $.$ & $.$ & $\FF^{2}$ & $\FF^{2}$ & $.$ & $.$ & $.$ & $.$ & $.$ \\
$4$ & $.$ & $.$ & $.$ & $.$ & $.$ & $.$ & $.$ & $.$ & $.$ \\
$2$ & $\FF[H]$ & $\FF[H]$ & $.$ & $.$ & $.$ & $.$ & $.$ & $.$ & $.$ \\
\hline
$ $ & $0$ & $1$ & $2$ & $3$ & $4$ & $5$ & $6$ & $7$ & $8$ \\
\end{tabular}
\end{minipage}

\item $7_{4b}$ (HN)\\\\
\noindent
\begin{minipage}{0.5\textwidth}
\small
\begin{tabular}{r|lllllllll}
$16$ & $.$ & $.$ & $.$ & $.$ & $.$ & $.$ & $.$ & $\FF$ & $\FF$ \\
$14$ & $.$ & $.$ & $.$ & $.$ & $.$ & $.$ & $\FF$ & $\FF$ & $.$ \\
$12$ & $.$ & $.$ & $.$ & $.$ & $.$ & $\FF$ & $\FF$ & $.$ & $.$ \\
$10$ & $.$ & $.$ & $.$ & $.$ & $\FF^{2}$ & $\FF^{2}$ & $.$ & $.$ & $.$ \\
$8$ & $.$ & $.$ & $.$ & $\FF$ & $\FF$ & $.$ & $.$ & $.$ & $.$ \\
$6$ & $.$ & $.$ & $\FF^{2}$ & $\FF^{2}$ & $.$ & $.$ & $.$ & $.$ & $.$ \\
$4$ & $.$ & $\FF$ & $\FF$ & $.$ & $.$ & $.$ & $.$ & $.$ & $.$ \\
$2$ & $\FF$ & $\FF$ & $.$ & $.$ & $.$ & $.$ & $.$ & $.$ & $.$ \\
\hline
$ $ & $0$ & $1$ & $2$ & $3$ & $4$ & $5$ & $6$ & $7$ & $8$ \\
\end{tabular}
\end{minipage}
\begin{minipage}{0.5\textwidth}
\small
\begin{tabular}{r|lllllllll}
$16$ & $.$ & $.$ & $.$ & $.$ & $.$ & $.$ & $.$ & $\FF$ & $\FF$ \\
$14$ & $.$ & $.$ & $.$ & $.$ & $.$ & $.$ & $.$ & $.$ & $.$ \\
$12$ & $.$ & $.$ & $.$ & $.$ & $.$ & $\FF$ & $\FF$ & $.$ & $.$ \\
$10$ & $.$ & $.$ & $.$ & $.$ & $\FF[H]/(H^2)$ & $\FF$ & $.$ & $.$ & $.$ \\
$8$ & $.$ & $.$ & $.$ & $\FF$ & $.$ & $.$ & $.$ & $.$ & $.$ \\
$6$ & $.$ & $.$ & $\FF$ & $\FF$ & $.$ & $.$ & $.$ & $.$ & $.$ \\
$4$ & $.$ & $.$ & $.$ & $.$ & $.$ & $.$ & $.$ & $.$ & $.$ \\
$2$ & $\FF[H]$ & $\FF[H]$ & $.$ & $.$ & $.$ & $.$ & $.$ & $.$ & $.$ \\
\hline
$ $ & $0$ & $1$ & $2$ & $3$ & $4$ & $5$ & $6$ & $7$ & $8$ \\
\end{tabular}
\end{minipage}

\item $7_{5a}, 7_{5b}$ (HN)\\\\
\noindent
\begin{minipage}{0.5\textwidth}
\small
\begin{tabular}{r|lllllllll}
$18$ & $.$ & $.$ & $.$ & $.$ & $.$ & $.$ & $.$ & $\FF$ & $\FF$ \\
$16$ & $.$ & $.$ & $.$ & $.$ & $.$ & $.$ & $\FF^{2}$ & $\FF^{2}$ & $.$ \\
$14$ & $.$ & $.$ & $.$ & $.$ & $.$ & $\FF^{3}$ & $\FF^{3}$ & $.$ & $.$ \\
$12$ & $.$ & $.$ & $.$ & $.$ & $\FF^{3}$ & $\FF^{3}$ & $.$ & $.$ & $.$ \\
$10$ & $.$ & $.$ & $.$ & $\FF^{3}$ & $\FF^{3}$ & $.$ & $.$ & $.$ & $.$ \\
$8$ & $.$ & $.$ & $\FF^{3}$ & $\FF^{3}$ & $.$ & $.$ & $.$ & $.$ & $.$ \\
$6$ & $.$ & $\FF$ & $\FF$ & $.$ & $.$ & $.$ & $.$ & $.$ & $.$ \\
$4$ & $\FF$ & $\FF$ & $.$ & $.$ & $.$ & $.$ & $.$ & $.$ & $.$ \\
\hline
$ $ & $0$ & $1$ & $2$ & $3$ & $4$ & $5$ & $6$ & $7$ & $8$ \\
\end{tabular}
\end{minipage}
\begin{minipage}{0.5\textwidth}
\small
\begin{tabular}{r|lllllllll}
$18$ & $.$ & $.$ & $.$ & $.$ & $.$ & $.$ & $.$ & $\FF$ & $\FF$ \\
$16$ & $.$ & $.$ & $.$ & $.$ & $.$ & $.$ & $\FF$ & $\FF$ & $.$ \\
$14$ & $.$ & $.$ & $.$ & $.$ & $.$ & $\FF^{2}$ & $\FF^{2}$ & $.$ & $.$ \\
$12$ & $.$ & $.$ & $.$ & $.$ & $\FF$ & $\FF$ & $.$ & $.$ & $.$ \\
$10$ & $.$ & $.$ & $.$ & $\FF^{2}$ & $\FF^{2}$ & $.$ & $.$ & $.$ & $.$ \\
$8$ & $.$ & $.$ & $\FF$ & $\FF$ & $.$ & $.$ & $.$ & $.$ & $.$ \\
$6$ & $.$ & $.$ & $.$ & $.$ & $.$ & $.$ & $.$ & $.$ & $.$ \\
$4$ & $\FF[H]$ & $\FF[H]$ & $.$ & $.$ & $.$ & $.$ & $.$ & $.$ & $.$ \\
\hline
$ $ & $0$ & $1$ & $2$ & $3$ & $4$ & $5$ & $6$ & $7$ & $8$ \\
\end{tabular}
\end{minipage}

\pagebreak

\item $7_{6a}, 7_{6b}$ (HN)\\\\
\noindent
\begin{minipage}{0.45\textwidth}
\footnotesize
\begin{tabular}{r|lllllllll}
$2$ & $.$ & $.$ & $.$ & $.$ & $.$ & $.$ & $.$ & $\FF$ & $\FF$ \\
$0$ & $.$ & $.$ & $.$ & $.$ & $.$ & $.$ & $\FF^{2}$ & $\FF^{2}$ & $.$ \\
$-2$ & $.$ & $.$ & $.$ & $.$ & $.$ & $\FF^{3}$ & $\FF^{3}$ & $.$ & $.$ \\
$-4$ & $.$ & $.$ & $.$ & $.$ & $\FF^{3}$ & $\FF^{3}$ & $.$ & $.$ & $.$ \\
$-6$ & $.$ & $.$ & $.$ & $\FF^{4}$ & $\FF^{4}$ & $.$ & $.$ & $.$ & $.$ \\
$-8$ & $.$ & $.$ & $\FF^{3}$ & $\FF^{3}$ & $.$ & $.$ & $.$ & $.$ & $.$ \\
$-10$ & $.$ & $\FF^{2}$ & $\FF^{2}$ & $.$ & $.$ & $.$ & $.$ & $.$ & $.$ \\
$-12$ & $\FF$ & $\FF$ & $.$ & $.$ & $.$ & $.$ & $.$ & $.$ & $.$ \\
\hline
$ $ & $-5$ & $-4$ & $-3$ & $-2$ & $-1$ & $0$ & $1$ & $2$ & $3$ \\
\end{tabular}
\end{minipage}
\begin{minipage}{0.55\textwidth}
\footnotesize
\begin{tabular}{r|lllllllll}
$2$ & $.$ & $.$ & $.$ & $.$ & $.$ & $.$ & $.$ & $\FF$ & $\FF$ \\
$0$ & $.$ & $.$ & $.$ & $.$ & $.$ & $.$ & $\FF$ & $\FF$ & $.$ \\
$-2$ & $.$ & $.$ & $.$ & $.$ & $.$ & $\FF[H] \oplus \FF$ & $\FF[H] \oplus \FF$ & $.$ & $.$ \\
$-4$ & $.$ & $.$ & $.$ & $.$ & $\FF^{2}$ & $\FF^{2}$ & $.$ & $.$ & $.$ \\
$-6$ & $.$ & $.$ & $.$ & $\FF^{2}$ & $\FF^{2}$ & $.$ & $.$ & $.$ & $.$ \\
$-8$ & $.$ & $.$ & $\FF$ & $\FF$ & $.$ & $.$ & $.$ & $.$ & $.$ \\
$-10$ & $.$ & $\FF$ & $\FF$ & $.$ & $.$ & $.$ & $.$ & $.$ & $.$ \\
$-12$ & $.$ & $.$ & $.$ & $.$ & $.$ & $.$ & $.$ & $.$ & $.$ \\
\hline
$ $ & $-5$ & $-4$ & $-3$ & $-2$ & $-1$ & $0$ & $1$ & $2$ & $3$ \\
\end{tabular}
\end{minipage}

\item $7_{7a}$ (HN)\\\\
\noindent
\begin{minipage}{0.45\textwidth}
\footnotesize
\begin{tabular}{r|lllllllll}
$6$ & $.$ & $.$ & $.$ & $.$ & $.$ & $.$ & $.$ & $\FF$ & $\FF$ \\
$4$ & $.$ & $.$ & $.$ & $.$ & $.$ & $.$ & $\FF^{3}$ & $\FF^{3}$ & $.$ \\
$2$ & $.$ & $.$ & $.$ & $.$ & $.$ & $\FF^{3}$ & $\FF^{3}$ & $.$ & $.$ \\
$0$ & $.$ & $.$ & $.$ & $.$ & $\FF^{4}$ & $\FF^{4}$ & $.$ & $.$ & $.$ \\
$-2$ & $.$ & $.$ & $.$ & $\FF^{4}$ & $\FF^{4}$ & $.$ & $.$ & $.$ & $.$ \\
$-4$ & $.$ & $.$ & $\FF^{3}$ & $\FF^{3}$ & $.$ & $.$ & $.$ & $.$ & $.$ \\
$-6$ & $.$ & $\FF^{2}$ & $\FF^{2}$ & $.$ & $.$ & $.$ & $.$ & $.$ & $.$ \\
$-8$ & $\FF$ & $\FF$ & $.$ & $.$ & $.$ & $.$ & $.$ & $.$ & $.$ \\
\hline
$ $ & $-4$ & $-3$ & $-2$ & $-1$ & $0$ & $1$ & $2$ & $3$ & $4$ \\
\end{tabular}
\end{minipage}
\begin{minipage}{0.55\textwidth}
\footnotesize
\begin{tabular}{r|lllllllll}
$6$ & $.$ & $.$ & $.$ & $.$ & $.$ & $.$ & $.$ & $\FF$ & $\FF$ \\
$4$ & $.$ & $.$ & $.$ & $.$ & $.$ & $.$ & $\FF^{2}$ & $\FF^{2}$ & $.$ \\
$2$ & $.$ & $.$ & $.$ & $.$ & $.$ & $\FF$ & $\FF$ & $.$ & $.$ \\
$0$ & $.$ & $.$ & $.$ & $.$ & $\FF[H] \oplus \FF^{2}$ & $\FF[H] \oplus \FF^{2}$ & $.$ & $.$ & $.$ \\
$-2$ & $.$ & $.$ & $.$ & $\FF^{2}$ & $\FF^{2}$ & $.$ & $.$ & $.$ & $.$ \\
$-4$ & $.$ & $.$ & $\FF$ & $\FF$ & $.$ & $.$ & $.$ & $.$ & $.$ \\
$-6$ & $.$ & $\FF$ & $\FF$ & $.$ & $.$ & $.$ & $.$ & $.$ & $.$ \\
$-8$ & $.$ & $.$ & $.$ & $.$ & $.$ & $.$ & $.$ & $.$ & $.$ \\
\hline
$ $ & $-4$ & $-3$ & $-2$ & $-1$ & $0$ & $1$ & $2$ & $3$ & $4$ \\
\end{tabular}
\end{minipage}

\item $7_{7b}$ (HN)\\\\
\noindent
\begin{minipage}{0.45\textwidth}
\footnotesize
\begin{tabular}{r|lllllllll}
$6$ & $.$ & $.$ & $.$ & $.$ & $.$ & $.$ & $.$ & $\FF$ & $\FF$ \\
$4$ & $.$ & $.$ & $.$ & $.$ & $.$ & $.$ & $\FF^{2}$ & $\FF^{2}$ & $.$ \\
$2$ & $.$ & $.$ & $.$ & $.$ & $.$ & $\FF^{2}$ & $\FF^{2}$ & $.$ & $.$ \\
$0$ & $.$ & $.$ & $.$ & $.$ & $\FF^{3}$ & $\FF^{3}$ & $.$ & $.$ & $.$ \\
$-2$ & $.$ & $.$ & $.$ & $\FF^{2}$ & $\FF^{2}$ & $.$ & $.$ & $.$ & $.$ \\
$-4$ & $.$ & $.$ & $\FF^{2}$ & $\FF^{2}$ & $.$ & $.$ & $.$ & $.$ & $.$ \\
$-6$ & $.$ & $\FF$ & $\FF$ & $.$ & $.$ & $.$ & $.$ & $.$ & $.$ \\
$-8$ & $\FF$ & $\FF$ & $.$ & $.$ & $.$ & $.$ & $.$ & $.$ & $.$ \\
\hline
$ $ & $-4$ & $-3$ & $-2$ & $-1$ & $0$ & $1$ & $2$ & $3$ & $4$ \\
\end{tabular}
\end{minipage}
\begin{minipage}{0.55\textwidth}
\footnotesize
\begin{tabular}{r|lllllllll}
$6$ & $.$ & $.$ & $.$ & $.$ & $.$ & $.$ & $.$ & $\FF$ & $\FF$ \\
$4$ & $.$ & $.$ & $.$ & $.$ & $.$ & $.$ & $\FF$ & $\FF$ & $.$ \\
$2$ & $.$ & $.$ & $.$ & $.$ & $.$ & $\FF$ & $\FF$ & $.$ & $.$ \\
$0$ & $.$ & $.$ & $.$ & $.$ & $\FF[H] \oplus \FF$ & $\FF[H] \oplus \FF$ & $.$ & $.$ & $.$ \\
$-2$ & $.$ & $.$ & $.$ & $\FF$ & $\FF$ & $.$ & $.$ & $.$ & $.$ \\
$-4$ & $.$ & $.$ & $\FF[H]/(H^2)$ & $\FF$ & $.$ & $.$ & $.$ & $.$ & $.$ \\
$-6$ & $.$ & $\FF$ & $.$ & $.$ & $.$ & $.$ & $.$ & $.$ & $.$ \\
$-8$ & $.$ & $.$ & $.$ & $.$ & $.$ & $.$ & $.$ & $.$ & $.$ \\
\hline
$ $ & $-4$ & $-3$ & $-2$ & $-1$ & $0$ & $1$ & $2$ & $3$ & $4$ \\
\end{tabular}
\end{minipage}

\end{itemize}
\restoregeometry
    
    \clearpage
    \printbibliography
    \addresses
\end{document}